\documentclass{amsart}
\usepackage{amsfonts,amssymb,latexsym}
\usepackage{amsmath,amsthm}
\usepackage{eucal}
\usepackage[latin1]{inputenc}

\newtheorem{theorem}{Theorem}[section]
\newtheorem{lemma}[theorem]{Lemma}
\newtheorem{proposition}[theorem]{Proposition}
\newtheorem{corollary}[theorem]{Corollary}

\theoremstyle{definition}

\theoremstyle{remark}
\newtheorem{remark}[theorem]{Remark}

\numberwithin{equation}{section}

\title[Well-posedness for the fifth-order KdV equation]{Well-posedness for the  fifth-order KdV equation in the energy space} 

\author[C. E. Kenig and D. Pilod]{Carlos E. Kenig$^{\star}$ and Didier Pilod$^{\dagger}$}

\thanks{$^{\star}$ Partially supported by NSF Grant DMS-0968472.}
\thanks{$^{\dagger}$ Partially supported by CNPq/Brazil, Grant 200001/2011-6.}

\date{}
\begin{document}
\maketitle 

\vspace{-0.5cm}

{\scriptsize \centerline{$^{\star}$Department of Mathematics, University of Chicago, }
             \centerline{ Chicago, IL, 60637 USA.}
             \centerline{email: cek@math.chicago.edu}

\vspace{0.1cm}
             \centerline{$^{\dagger}$ Instituto de Matemática,
             Universidade Federal do Rio de Janeiro,}
              \centerline{Caixa Postal 68530, CEP: 21945-970, Rio
              de Janeiro, RJ, Brazil.}
              \centerline{email: didier@im.ufrj.br}}

\begin{abstract}
We prove that the initial value problem (IVP) associated to the fifth order KdV equation
\begin{equation} \label{05KdV}
\partial_tu-\alpha\partial^5_x u=c_1\partial_xu\partial_x^2u+c_2\partial_x(u\partial_x^2u)+c_3\partial_x(u^3),
\end{equation}
where $x \in \mathbb R$, $t \in \mathbb R$, $u=u(x,t)$ is a real-valued function and $\alpha, \ c_1, \ c_2, \ c_3$ are real constants with $\alpha \neq 0$, is locally well-posed in $H^s(\mathbb R)$ for $s \ge 2$. In the Hamiltonian case (\textit{i.e.} when $c_1=c_2$), the IVP associated to \eqref{05KdV} is then globally well-posed in the energy space $H^2(\mathbb R)$.
\end{abstract}

\section{Introduction}
Considered here is the initial value problem (IVP) associated to the fifth-order Korteweg-de Vries equation
\begin{equation} \label{5KdV}
\left\{\begin{array}{l} 
\partial_tu-\alpha\partial^5_x u=c_1\partial_xu\partial_x^2u+c_2\partial_x(u\partial_x^2u)+c_3\partial_x(u^3) \\ 
u(\cdot,0)=u_0,
\end{array} \right.
\end{equation}
where $x \in \mathbb R$, $t \in \mathbb R$, $u=u(x,t)$ is a real-valued function and $\alpha, \ c_1, \ c_2, \ c_3$ are real constants with $\alpha \neq 0$.  Such equations and its generalizations 
\begin{equation} \label{gene5KdV}
\partial_tu-\alpha\partial^5_x u+\beta\partial_x^3u=c_0u\partial_xu+c_1\partial_xu\partial_x^2u+c_2\partial_x(u\partial_x^2u)+c_3\partial_x(u^3)
\end{equation}
arise as long-wave approximations to the water-wave equation. They have been derived as second-order asymptotic expansions for unidirectional wave propagation in the so-called Boussinesq regime (see Craig, Guyenne and Kalisch \cite{CGK}, Olver \cite{Ol} and the references therein), the first order expansions being of course the Korteweg-de Vries (KdV) equation, 
\begin{equation} \label{KdV} 
\partial_tu+\beta\partial_x^3u=c_0u\partial_xu.
\end{equation} 
The equation in \eqref{5KdV} was also proposed by Benney \cite{Be} as a model for interaction of short and long waves.

When $c_1=c_2$, the Hamiltonian 
\begin{equation} \label{H} 
H(u)=\frac12 \int_{\mathbb R}\Big(\alpha(\partial_x^2u)^2-c_1u(\partial_xu)^2+\frac{c_3}2u^4
\Big)dx
\end{equation} 
as well as the quantity 
\begin{equation} \label{M} 
M(u)=\int_{\mathbb R}u^2dx,
\end{equation}
are conserved by the flow of \eqref{5KdV}. Indeed, it is easy to check that 
\begin{displaymath}
H'(u)\varphi =\int_{\mathbb R}\big(\alpha\partial_x^4u-\frac{c_1}2(\partial_xu)^2+c_1\partial_x(u\partial_xu)+c_3u^3 \big)\varphi dx =:\big(\text{grad} \, H(u),\varphi \big)_{L^2}.
\end{displaymath}
Thus the equation in \eqref{5KdV} has the form 
$ \partial_tu=\partial_x\text{grad} \, H(u)$, so that 
\begin{displaymath} 
\frac{d}{dt}H(u)=\big(\text{grad} \, H(u),\partial_tu \big)_{L^2}=\big(\text{grad} \, H(u),\partial_x\text{grad} \, H(u)\big)_{L^2}=0.
\end{displaymath}
Moreover in the special case where  $c_2=c_1=-10\alpha$ and $c_3=10\alpha$, the equation in \eqref{5KdV} is the equation following KdV in the KdV hierarchy discovered by Lax \cite{La} and writes in the case $\alpha=1$
\begin{equation} \label{Lax5KdV} 
\partial_tu-\partial_x^5u+10\partial_xu\partial_x^2u+10\partial_x(u\partial_x^2u)-10\partial_x(u^3)=0.
\end{equation}
 Therefore equation \eqref{Lax5KdV} is completely integrable and possesses an infinite number of conservation laws. We refer to the introductions in \cite{Gr,Po,Sa} for more details on this subject.

Our purpose is to study the IVP \eqref{5KdV} in classical $L^2$-based Sobolev spaces $H^s(\mathbb R)$. We shall say that the IVP is locally (resp. globally) well-posed in the function space $X$ if it induces a dynamical system on $X$ by generating a continuous local (resp. global) flow. 

First, it is worth mentioning that without dispersion (\textit{i.e.} when $\alpha=0$) and when $c_1\neq 0$ or $c_2 \neq 0$, the IVP \eqref{5KdV} is likely to be ill-posed in any $H^s(\mathbb R)$ (see the comments in the introduction of \cite{Po}). This is in sharp contrast with the KdV equation. Indeed, when $\beta=0$ in \eqref{KdV}, we obtain the Burgers equation, which is still well-posed in $H^s(\mathbb R)$ for $s>3/2$ by using standard energy methods. However, the direct energy estimate for equation \eqref{5KdV}  (after fixing  $c_3=0$ for simplicity) gives only
\begin{equation} \label{standardenergy}
\frac{d}{dt}\|\partial_x^ku(t)\|_{L^2}^2 \lesssim \|\partial_x^3u\|_{L^{\infty}_x}\|\partial_x^ku(t)\|_{L^2} ^2+\Big|\int_{\mathbb R}\partial_xu\partial_x^{k+1}u\partial_x^{k+1}udx\Big|.
\end{equation}
Observe that the last term on the right-hand side of \eqref{standardenergy} has still higher-order derivatives and cannot be treated by using only integration by parts. To overcome this difficulty, Ponce \cite{Po} used a recursive argument based on the dispersive smoothing effects associated to the linear part of \eqref{5KdV}, combined to a parabolic regularization method, to establish that the IVP \eqref{5KdV} is locally well-posed in $H^s(\mathbb R)$ for $s \ge 4$. Later, Kwon \cite{Kw} improved Ponce's result by proving local well-posedness for \eqref{5KdV} in $H^s(\mathbb R)$ for $s>5/2$. The main new idea was to modify the energy by adding a correctional lower-order cubic term to cancel the last term on the right-hand side of \eqref{standardenergy}. Note that he also used a refined Strichartz estimate derived by chopping the time interval in small pieces whose length depends on the spatial frequency. This estimate was first established by Koch and Tzvetkov \cite{KT} (see also Kenig and Koenig \cite{KK} for an improved version) in the Benjamin-Ono context.

On the other hand, it was proved\footnote{Strictly speaking the result was proved only in the case where $c_3=0$, but as observed in the introduction of \cite{Gr}, the cubic term $\partial_x(u^3)$ in \eqref{5KdV} is well behaved and no cancellations occur, so that the proof remains true even when $c_3 \neq 0$.} by the second author in \cite{Pi}, by using an argument due to Molinet, Saut and Tzvetkov for the Benjamin-Ono equation \cite{MST}, that, in the case $c_2 \neq 0$, the flow map associated to \eqref{5KdV} fails to be $C^2$ in $H^s(\mathbb R)$, for any $s \in \mathbb R$.  This result was improved by Kwon \cite{Kw}  who showed that the flow map fails to be even uniformly continuous in $H^s(\mathbb R)$ when $s>\frac52$ (and $s>0$ in the completely integrable case). Those results are based on the fact that the dispersive smoothing effects associated to the linear part of \eqref{5KdV} are not strong enough to control the high-low frequency interactions in the nonlinear term $\partial_x(u\partial_x^2u)$. As a consequence, one cannot solve the IVP \eqref{5KdV} by a Picard iterative method implemented on the integral equation associated to \eqref{5KdV} for initial data in any Sobolev space $H^s(\mathbb R)$ with $s \in \mathbb R$.  

However, the fixed point method may be employed to prove well-posedness for \eqref{5KdV} in other function spaces. For example in \cite{KPV3,KPV4}, Kenig, Ponce and Vega proved that the more general class of IVPs 
\begin{equation} \label{KdV_h}
\left\{  \begin{array}[pos]{ll}
          \partial_tu+\partial_x^{2j+1}u+P(u,\partial_xu,\ldots,\partial_x^{2j}u),
          \quad x,\ t \in \mathbb R, \ j \in \mathbb N \\          u(0)=u_0, \\
     \end{array} \right.
\end{equation}
where
\begin{displaymath}
P: \mathbb R^{2j+1}\rightarrow \mathbb R \quad (\mbox{or} \ P:
\mathbb C^{2j+1}\rightarrow \mathbb C)
\end{displaymath}
is a polynomial having no constant or linear terms, is well-posed in weighted Sobolev spaces of the type $H^k(\mathbb R) \cap H^l(\mathbb R; x^2dx)$ with $k, \ l \in \mathbb Z_+$, $k \ge k_0, \ l \ge l_0$ for some $k_0, \ l_0 \in \mathbb Z_+$. We also refer to \cite{Pi} for sharper results in the case of small initial data and when the nonlinearity in \eqref{KdV_h} is quadratic. Recently, Gr\"unrock \cite{Gr}, respectively Kato \cite{Ka}, used variants of the Fourier restriction norm method to prove well-posedness in $\widehat{H}^s_r(\mathbb R)$ for $1<r \le \frac{4}{3}$ and $s>\frac14+\frac{3}{2r'}$, respectively in $H^{s,a}(\mathbb R)$ for $s \ge \max\{-\frac14,-2a-2\}$ with $-\frac32 < a \le -\frac14$ and $(s,a) \neq (-\frac14,-\frac78)$. The spaces $\widehat{H}^s_r(\mathbb R)$ and $H^{s,a}(\mathbb R)$ are respectively defined by the norms
$ \|\varphi\|_{\widehat{H}^s_r}=\|\langle \xi\rangle^s \widehat{\varphi}\|_{L^{r'}} $with $\frac1r+\frac1{r'}=1$ and $\|\varphi\|_{H^{s,a}}=\|\langle \xi\rangle^{s-a}|\xi|^a \widehat{\varphi}\|_{L^2}$.

Nevertheless, the $L^2$-based Sobolev spaces $H^s(\mathbb R)$ remain the natural\footnote{When the equation in \eqref{5KdV} is Hamiltonian (\textit{i.e.} when $c_1=c_2$),  the space $H^2(\mathbb R)$ is the natural space where the Hamiltonian $H$ in \eqref{H} is well defined.} spaces to study well-posedness for the fifth order KdV equation.  Our main result states that the IVP \eqref{5KdV} is locally well-posed in $H^s(\mathbb R)$ for $s \ge 2$.
\begin{theorem} \label{maintheo}
Assume that $s \ge 2$. Then, for every $u_0 \in H^s(\mathbb R)$, there exists a positive time $T=T(\|u_0\|_{H^s})$ and a unique solution $u$ to \eqref{5KdV} in the class 
\begin{equation} \label{maintheo1} 
 C([-T,T];H^s(\mathbb R)) \cap F^s(T) \cap B^s(T).
\end{equation}
Moreover, for any $0<T'<T$, there exists a neighbourhood $\mathcal{U}$ of $u_0$ in $H^s(\mathbb R)$ such that the flow map data-solution 
\begin{equation} \label{maintheo2} 
S^s_{T'}: \mathcal{U} \longrightarrow C([-T',T'];H^s(\mathbb R)) , \ u_0 \longmapsto u,
\end{equation}
is continuous.
\end{theorem}

\begin{remark} The short-time Bourgain space $F^s(T):=F^s_2(T)$  and the energy space $B^s(T)$ are defined in Subsection \ref{functionspace}.
\end{remark}

\begin{remark} The result of Theorem \ref{maintheo} is also valid for equation \eqref{gene5KdV} and the proof is similar. 
\end{remark}

\begin{remark} For sake of simplicity, we assume that $\alpha=1$ and $c_3=0$ in the proof of Theorem \ref{maintheo} since the cubic term $\partial_x(u^3)$ has low order derivative when compared to the two other nonlinear terms in \eqref{5KdV}. Nevertheless, we indicate in the appendix what modifications are needed to deal with the case $c_3 \neq 0$.
\end{remark}

\begin{remark} Observe that at this level of regularity ($s \ge 2$), the limits of smooth solutions are still weak solutions to the equation in \eqref{5KdV}. 
\end{remark}

\begin{remark} As a byproduct of the proof of Theorem \ref{maintheo}, we obtain \textit{a priori} estimates on smooth solutions of \eqref{5KdV} in $H^s(\mathbb R)$ for $s \ge \frac54$ (see Proposition \ref{apriori} below). In other word, the flow map data-solutions in $H^{\infty}(\mathbb R)$ satisfies 
\begin{equation} \label{a priori1}
\big\|S^{\infty}_T(u_0) \big\|_{L^{\infty}_TH^s} \lesssim \|u_0\|_{H^s},
\end{equation}
for any $s \ge \frac54$ and where $T$ only depends on $\|u_0\|_{H^s}$. However, we were not able to prove well-posedness at this level of regularity.
\end{remark}

In the Hamiltonian case, the conserved quantities $H$ and $M$ defined in \eqref{H}--\eqref{M} provide a control on the $H^2$-norm and allow to prove that the IVP \eqref{5KdV} is globally well-posed in $H^2(\mathbb R)$. 
\begin{corollary} \label{maincorro} 
In the case $c_1=c_2$, the results of Theorem \ref{maintheo} are true for $T>0$ arbitrarily large.
\end{corollary}

\begin{remark} 
Corollary \ref{maincorro} remains true for equation \eqref{gene5KdV} (still in the case $c_1=c_2$).
\end{remark} 

\begin{remark} 
In \cite{Sa}, Saut already proved the existence of global weak solutions in the Hamiltonian case. However Corollary \ref{maincorro} is the first existence result of global strong solutions for the fifth-order KdV equation in the Hamiltonian case\footnote{Except of course in the completely integrable case.}.
\end{remark}

\begin{remark} In his study of stability of solitary waves for Hamiltonian fifth-order water wave models of the form \eqref{gene5KdV} with quadratic nonlinearities\footnote{The question of existence of solitary waves for such models with nonhomogeneous nonlinearities was addressed in \cite{Kic}.}, Levandosky assumed well-posedness in $H^2(\mathbb R)$ (c.f. Assumption 1.1 in \cite{Lev}). Therefore, Corollary \ref{maincorro} provides an affirmative answer to this issue. We also refer to \cite{An,Lev2} for further results on stability/instability of such fifth-order water wave models.
\end{remark}

We now discuss the main ingredients in the proof of Theorem \ref{maintheo}. We follow the method introduced by Ionescu, Kenig and Tataru \cite{IKT} in the context of the KP1 equation, which is based on the dyadic Bourgain's spaces $F^s_{\alpha}$ and their dual $N^s_{\alpha}$, defined in Subsection \ref{functionspace}. We refer to \cite{CCT,KT2} for previous works using similar spaces to prove \textit{a priori} bounds for the 1 D cubic NLS at low regularity and also to \cite{Guo,Guo1,Mo} for applications to other dispersive equations. 

The $F^s_{\alpha}$ spaces enjoy a $X^{s,b}$-type structure but with a localization in small time dependent intervals whose length is of order $2^{-\alpha k}$ when the spatial frequency of the function is localized around $2^{k}$. This prevents the time frequency modulation\footnote{Here, $w(\xi)=\xi^5$ denotes the dispersive symbol of the linear equation.} $|\tau-w(\xi)|$ to be too small, which allows for suitable $\alpha$, $\alpha=2$ in our case, to prove a bilinear estimate of the form (c.f Proposition \ref{bilinE} for a precise statment )
\begin{equation} \label{Introbilin} 
\big\| \partial_xu\partial_x^2v\big\|_{N^s_2(T)}+\big\| \partial_x(u\partial_x^2v)\big\|_{N^s_2(T)}  \lesssim \|u\|_{F^s_2(T)}\|v\|_{F^s_2(T)},
\end{equation}
as soon as $s>1$. Of course\footnote{This would be in contradiction with the $C^2$-ill-posedness results in \cite{Pi}.}, we cannot conclude directly by using a contraction argument since the linear estimate estimate 
\begin{equation} \label{Introlinear} 
\|u\|_{F^s_{\alpha}(T)} \lesssim \|u\|_{B^s(T)}+\big\| \partial_xu\partial_x^2v\big\|_{N^s_{\alpha}(T)}+\big\| \partial_x(u\partial_x^2v)\big\|_{N^s_{\alpha}(T)}
\end{equation}
requires the introduction of the energy norm $\|u\|_{B^s(T)}$, instead of the usual $H^s$-norm of the initial data $\|u_0\|_{H^s}$, in order to control the small time localization appearing in the $F^s_{\alpha}$-structure. Therefore it remains to derive the frequency localized energy estimate 
\begin{equation} \label{Introenergy} 
\|u\|_{B^s(T)}^2 \lesssim \|u_0\|_{H^s}^2+\big(1+\|u\|_{F^s(T)}\big)\|u\|_{F^s(T)}\big(\|u\|_{F^s(T)}^2+\|u\|_{B^s(T)}^2\big),
\end{equation}
which is true if $s\ge \frac54$ and $\|u\|_{L^{\infty}_TH^s_x}$ is small. The main new difficulty in our case is that after using suitable frequency localized commutator estimates, we are not able to handle directly the remaining lower-order terms  (see Lemma \ref{tec2EE} and Remark \ref{rematecEE} below). This is somehow the price to pay for the choice of $\alpha=2$ which enabled to derive the bilinear estimate \eqref{Introbilin}. Then, we modify the energy by adding a cubic lower-order term to $\|u\|_{B^s(T)}^2$ in order to cancel those terms. This can be viewed as a localized version of Kwon's argument in \cite{Kw}.

We deduce the \textit{a priori} bound \eqref{a priori1} by combining \eqref{Introbilin}--\eqref{Introenergy} and using a scaling argument. To finish the  proof of Theorem \ref{maintheo}, we apply this method to the difference of two solutions. However, due to the lack of symmetry of the new equation, we only are able to prove the corresponding energy estimate for $s \ge 2$. Finally, we conclude the proof by adapting the classical Bona-Smith argument \cite{BS}. 

Around the time when we completed this work, we learned that Guo, Kwak and Kwon \cite{GKK} had also worked on the same problem and obtained the same results as ours (in Theorem \ref{maintheo} and Proposition \ref{apriori}). They also used the short-time $X^{s,b}$ method. However, instead of modifying the energy as we did, they put an additional weight in the $X^{s,b}$ structure of the spaces in order to derive the key energy estimates.

The rest of the paper is organized as follows: in Section 2, we introduce the notations, define the function spaces and prove some of their basic properties as well the main linear estimates. In Section 3, we derive the $L^2$ bilinear and trilinear estimates, which are used to prove the bilinear estimates in Section 4 and the energy estimates in Section 5. The proof of Theorem \ref{maintheo} is given in Section 6. We conclude the paper with an appendix explaining how to treat the cubic term $\partial_x(u^3)$, which we omit in the previous sections to simplify the exposition.

\section{Notation, function spaces and linear estimates} 

\subsection{Notation}
For any positive numbers $a$ and $b$, the notation $a \lesssim b$
means that there exists a positive constant $c$ such that $a \le c
b$. We also denote $a \sim b$ when $a \lesssim b$ and $b \lesssim
a$. Moreover, if $\alpha \in \mathbb R$, $\alpha_+$, respectively
$\alpha_-$, will denote a number slightly greater, respectively
lesser, than $\alpha$. 

For $a_1, \ a_2, \ a_3 \in \mathbb R$, it will be convenient to define the quantities $a_{max} \ge a_{med} \ge a_{min}$ to be the maximum, median and minimum of $a_1, \ a_2$  and $a_3$ respectively. For $a_1, \ a_2, \ a_3, \ a_4 \in \mathbb R$, we define the quantities $a_{max} \ge a_{sub} \ge a_{thd} \ge a_{min}$ to be the maximum, sub-maximum, third-maximum and minimum of $a_1, \ a_2, \ a_3$  and $a_4$ respectively. Usually, we use $k_i$ and $j_i$ to denote integers and $N_i=2^{k_i}\ , L_i=2^{j_i}$ to denote dyadic numbers. 

For $u=u(x,t) \in \mathcal{S}(\mathbb R^2)$,
$\mathcal{F}u=\widehat{u}$ will denote its space-time Fourier
transform, whereas $\mathcal{F}_xu=(u)^{\wedge_x}$, respectively
$\mathcal{F}_tu=(u)^{\wedge_t}$, will denote its Fourier transform
in space, respectively in time.  Moreover, we generally omit the index $x$ or $t$ when the function depends only on one variable. For $s \in \mathbb R$, we define the
Bessel and Riesz potentials of order $-s$, $J^s_x$ and $D_x^s$, by
\begin{displaymath}
J^s_xu=\mathcal{F}^{-1}_x\big((1+|\xi|^2)^{\frac{s}{2}}
\mathcal{F}_xu\big) \quad \text{and} \quad
D^s_xu=\mathcal{F}^{-1}_x\big(|\xi|^s \mathcal{F}_xu\big).
\end{displaymath}

The unitary group $e^{t\partial_x^5}$ associated to the linear dispersive equation 
\begin{equation} \label{linear5KdV} 
\partial_tu-\partial_x^5u=0,
\end{equation}
is defined via Fourier transform by 
\begin{equation} \label{group} 
e^{t\partial_x^5}u_0= \mathcal{F}_x^{-1}\big( e^{itw(\xi)}\mathcal{F}_xu_0\big),
\end{equation}
where $w(\xi)=\xi^5$.

For $k \in \mathbb Z_+$, let us define 
\begin{displaymath} 
I_k=\big\{ \xi \in \mathbb R \ : \ 2^{k-1} \le |\xi| \le 2^{k+1} \big\},
\end{displaymath}
if $k \ge 1$ and
\begin{displaymath}
I_0=\big\{ \xi \in \mathbb R \ : \  |\xi| \le 2 \big\}.
\end{displaymath}
Throughout the paper, we fix an even smooth cutoff function $\eta_0:\mathbb R \rightarrow [0,1]$ supported in $[-8/5,8/5] $ and such that $\eta_0$ is equal to $1$ in $[-5/4,5/4]$. 
For $k \in \mathbb Z \cap [1,+\infty)$, we define the functions $\eta_k$ and $\eta_{\le k}$ respectively by 
\begin{equation} \label{eta}
\eta_k(\xi)=\eta_0(2^{-k}\xi)-\eta_0(2^{-(k-1)}\xi)=:\eta(2^{-k}\xi) \quad \text{and}\quad \eta_{\le k}=\sum_{j=0}^k \eta_{j}. 
\end{equation} 
Then, $(\eta_k)_{k\ge 0}$ is dyadic partition of the unity satisfying $\text{supp} \ \eta_k \subset I_k$. 

Let $(\widetilde{\eta}_k)_{k \ge 0}$ be another nonhomogeneous dyadic partition of the unity satisfying $\text{supp} \, \widetilde{\eta}_k \subset I_k$ and $\widetilde{\eta}_k=1$ on $\text{supp} \; \eta_k$.

Finally, for $k \in \mathbb Z \cap [1,+\infty)$, let us define the Fourier multiplier $P_{k}$, $P_{\le 0}$ and  $P_{\le k}$ by 
\begin{displaymath}
P_{k}u=\mathcal{F}^{-1}_x\big(\eta_k\mathcal{F}_xu\big), \ P_{\le 0}u=\mathcal{F}^{-1}_x\big(\eta_{0}\mathcal{F}_xu\big), \ \text{and} \ P_{\le k}=P_{\le 1}+\sum_{j=1}^kP_{j}. 
\end{displaymath}  
Then it is clear that $P_{\le 0} +\sum_{k=1}^{+\infty}P_{k}=1$. Often, when there is no risk of confusion, we also denote $P_0=P_{\le 0}.$

\subsection{Function spaces}  \label{functionspace}
For $1 \le p \le \infty$, $L^p(\mathbb R)$ is the usual Lebesgue space with the norm $\|\cdot \|_{L^p}$, and for $s \in \mathbb R$, the Sobolev spaces $H^s(\mathbb R)$ is defined via its usual norm $\|\phi \|_{H^s}= \|J_x^s \phi\|_{L^2}$. 

Let $f=f(x,t)$ be a function defined for $x \in
\mathbb R$ and $t$ in the time interval $[-T,T]$, with $T>0$ or in the whole line $\mathbb R$. Then if $X$
is one of the spaces defined above, we define the spaces $L^p_TX_x$ and 
$L^p_tX_x$ by the norms
\begin{displaymath}
\|f\|_{L^p_TX_x} =\Big(
\int_{-T}^T\|f(\cdot,t)\|_{X}^pdt\Big)^{\frac1p} \quad \text{and} \quad
\|f\|_{L^p_tX_x} =\Big( \int_{\mathbb R}\|f(\cdot,t)\|_{X}^pdt\Big)^{\frac1p}, 
\end{displaymath}
when $1 \le p < \infty$, with the natural modifications for $p=\infty$.

We will work with the short time localized Bourgain spaces introduced in \cite{IKT}. 
First, for $k \in \mathbb Z_+$, we introduce the $l^1$-Besov type space $X_k$ of regularity $1/2$ with respect to modulations, 
\begin{displaymath} 
X_k=\Big\{\phi \in L^2(\mathbb R^2) \ : \ \text{supp} \, \phi \subset I_k \times \mathbb R \ \text{and} \  \|\phi\|_{X_k}= < \infty \Big\},
\end{displaymath}
where 
\begin{equation} \label{Xk} 
\|\phi\|_{X_k}=\sum_{j=0}^{+\infty}2^{j/2}\|\eta_{j}(\tau-w(\xi))\phi(\xi,\tau)\|_{L^2_{\xi,\tau}}.
\end{equation}
Let $\alpha \ge 0$ be fixed. For $k \in \mathbb Z_+$, we introduce the space $F_{k,\alpha}$ possessing a  $X_k$-structure in short time intervals of length $2^{-\alpha k}$, 
\begin{displaymath} 
F_{k,\alpha}=\Big\{f \in L^{\infty}(\mathbb R;L^2(\mathbb R))\ : \ \text{supp} \, \mathcal{F}(u) \subset I_k \times \mathbb R \ \text{and} \  \|f\|_{F_{k,\alpha}} < \infty \Big\}
\end{displaymath}
where
\begin{equation} \label{Fk} 
\|f\|_{F_{k,\alpha}}= \sup_{\tilde{t} \in \mathbb R} \|\mathcal{F}\big(\eta_0(2^{\alpha k}(\cdot-\tilde{t}))f\big)\|_{X_k}.
\end{equation}
Its dual version $N_{k,\alpha}$ is defined by 
\begin{displaymath} 
N_{k,\alpha}=\Big\{f \in L^{\infty}(\mathbb R;H^{-1}(\mathbb R)) \ : \ \text{supp} \, \mathcal{F}(f) \subset I_k \times \mathbb R \ \text{and} \  \|f\|_{N_{k,\alpha}} < \infty \Big\},
\end{displaymath}
where
\begin{equation} \label{Nk} 
\|f\|_{N_{k,\alpha}}= \sup_{\tilde{t} \in \mathbb R} \big\|\big(\tau-w(\xi)+i2^{\alpha k} \big)^{-1}\mathcal{F}\big(\eta_0(2^{\alpha k}(\cdot-\tilde{t}))f\big) \big\|_{X_k}.
\end{equation}
Now for $s \in \mathbb R_+$, we define the global $F_{\alpha}^s$ and $N_{\alpha}^s$  spaces from their frequency localized versions $F_{k,\alpha}$ and $N_{k,\alpha}$, by using a nonhomogeneous Littlewood-Paley decomposition as follows
\begin{equation} \label{Fs}
F^s_{\alpha}=\Big\{f \in L^{\infty}(\mathbb R;L^2(\mathbb R)) \ : \  \|f\|_{F^s_{\alpha}}=\Big(\sum_{k=0}^{+\infty}2^{2ks}\|P_{k}f\|_{F_{k,\alpha}}^2 \Big)^{\frac12} < \infty \Big\}.
\end{equation}
and 
\begin{equation} \label{Ns}
N^s_{\alpha}=\Big\{f \in L^{\infty}(\mathbb R;H^{-1}(\mathbb R)) \ : \  \|f\|_{N^s_{\alpha}}=\Big(\sum_{k=0}^{+\infty}2^{2ks}\|P_{k}f\|_{N_{k,\alpha}}^2 \Big)^{\frac12} < \infty \Big\}.
\end{equation}
We also define a localized (in time) version of those spaces. Let $T$ be a positive time and $Y$ denote $F_{\alpha}^s$ or $N_{\alpha}^s$. If $f:\mathbb R \times [-T,T] \rightarrow \mathbb C$, we define 
\begin{displaymath} 
\|f\|_{Y(T)}= \inf \big\{ \|\tilde{f}\|_{Y}\ : \ \tilde{f}: \mathbb R^2 \rightarrow \mathbb C\ \text{and} \ \tilde{f}_{|_{\mathbb R \times [-T,T]}}=f\big\}.
\end{displaymath}
Then, 
\begin{displaymath} 
F^s_{\alpha}(T)=\Big\{f \in L^{\infty}([-T,T];L^2(\mathbb R)) \ : \  \|f\|_{F^s_{\alpha}(T)}< \infty \Big\}.
\end{displaymath}
and similarly 
\begin{displaymath} 
N^s_{\alpha}(T)=\Big\{f \in L^{\infty}([-T,T];H^{-1}(\mathbb R)) \ : \  \|f\|_{N^s_{\alpha}(T)}< \infty \Big\}.
\end{displaymath}

Finally for $s \in \mathbb R_+$ and $T>0$, we define the energy space $B^s(T)$ by 
\begin{displaymath} 
B^s(T)=\Big\{f \in L^{\infty}([-T,T];L^2(\mathbb R)) \ : \  \|f\|_{B^s(T)}< \infty \Big\},
\end{displaymath}
where
\begin{equation} \label{Bs} 
\|f\|_{B^s(T)}=\Big(\|P_{\le 0}f(\cdot,0)\|_{L^2}^2+\sum_{k=1}^{+\infty}2^{2ks}\sup_{t_k \in [-T,T]}\|P_{k}f(\cdot,t_k)\|_{L^2}^2 \Big)^{\frac12}.
\end{equation}

\subsection{First properties} Following \cite{IKT}, we state some important properties of the $F^s_{\alpha}(T)$ spaces. First, we show that $F^s_{\alpha}(T) \hookrightarrow L^{\infty}([-T,T];H^s(\mathbb R))$.
\begin{lemma} \label{lemma1} 
Let $T>0$, $s \in \mathbb R_+$ and $\alpha \in \mathbb R_+$. Then it holds that 
\begin{equation} \label{lemma1.1} 
\|f\|_{L^{\infty}_TH^s_x} \lesssim \|f\|_{F^s_{\alpha}(T)},
\end{equation}
for all $f \in F^s_{\alpha}(T)$.
\end{lemma} 
\begin{proof} Let $f \in F^s_{\alpha}(T)$. We choose $\widetilde{f} \in F_{\alpha}^s$ such that 
\begin{equation} \label{lemma1.2}
\widetilde{f}_{|_{[-T,T]}}=f \quad \text{and} \quad \|\widetilde{f}\|_{F^s_{\alpha}} \le 2\|f\|_{F^s_{\alpha}(T)}.
\end{equation}
It follows that for every $t \in [-T,T]$, 
\begin{equation} \label{lemma1.3} 
\|f(\cdot,t)\|_{H^s}=\|\widetilde{f}(\cdot,t)\|_{H^s} \lesssim \Big(\sum_{k=0}^{+\infty}2^{2ks}\|\widetilde{f_k}(\cdot,t)\|_{L^2}^2\Big)^{\frac12},
\end{equation}
where $\widetilde{f_0}=P_{\le 0}\widetilde{f}$ and $\widetilde{f_k}=P_{k}\widetilde{f}$  for any $k \in \mathbb Z \cap [1,+\infty)$.

Now fix $t \in [-T,T]$ and $k \in \mathbb Z_+$.  The Fourier inversion formula gives that
\begin{equation} \label{lemma1.4}
\mathcal{F}_x(\widetilde{f}_k)(\xi,t)=c\int_{\mathbb R}\mathcal{F}\big(\eta_0(2^{\alpha k}(\cdot-t))\widetilde{f}_k \big)(\xi,\tau)e^{it\tau}d\tau.
\end{equation}
On the other hand, the definition $X_k$ in \eqref{Xk} and the Cauchy-Schwarz inequality in $\tau$ implies that
\begin{equation} \label{lemma1.5}
\Big\| \int_{\mathbb R}\big|\phi(\xi,\tau)\big|d\tau\Big\|_{L^2_{\xi}} \lesssim \|\phi\|_{X_k},
\end{equation}
for all $\phi \in X_k$. Therefore, it is deduced from \eqref{Fk}, \eqref{lemma1.4} and \eqref{lemma1.5} that 
\begin{equation} \label{lemma1.6}
\|\widetilde{f_k}(\cdot,t)\|_{L^2} \lesssim \|\widetilde{f_k}\|_{F_{k,\alpha}},
\end{equation}
for all $k \in \mathbb Z_+$.
Then,  estimate \eqref{lemma1.1} follows gathering \eqref{lemma1.2}, \eqref{lemma1.3}, \eqref{lemma1.6} and taking the supreme over $t \in [-T,T]$.
\end{proof}

Then, we derive an important property involving the space $X_k$ (see \cite{IKT}).
\begin{lemma} \label{lemma2} 
Let $\alpha \ge 0$ and $l\in \mathbb Z_+$ be given. Then, if $[\alpha l]$ denotes the integer part of $\alpha l$, we have that 
\begin{equation} \label{lemma2.1} 
2^{\frac{\alpha l}2}\Big\| \eta_{\le [\alpha l]}(\tau-w(\xi))\int_{\mathbb R}|\phi(\xi,\tau')|2^{-\alpha l}\big(1+2^{-\alpha l}|\tau-\tau'|\big)^{-4}d\tau'\Big\|_{L^2_{\xi,\tau}} \lesssim \|\phi\|_{X_k},
\end{equation}
and 
\begin{equation} \label{lemma2.2} 
\sum_{j>[\alpha l]}2^{\frac j2}\Big\| \eta_{j}(\tau-w(\xi))\int_{\mathbb R}|\phi(\xi,\tau')|2^{-\alpha l}\big(1+2^{-\alpha l}|\tau-\tau'|\big)^{-4}d\tau'\Big\|_{L^2_{\xi,\tau}} \lesssim \|\phi\|_{X_k},
\end{equation}
for all $\phi \in X_k$.
\end{lemma}

\begin{proof} We fix $\tilde{l}=[\alpha l]$. We begin proving estimate \eqref{lemma2.1}. Following \cite{Mo}, we use that $(\eta_k)_{k \ge 0}$ is dyadic partition of the unity and the Cauchy-Schwarz inequality in $\tau'$ to get that 
\begin{displaymath} 
\begin{split}
I & :=  2^{\frac{\alpha l}2}\Big\| \eta_{\le \tilde{l}}(\tau-w(\xi))\int_{\mathbb R}|\phi(\xi,\tau')|2^{-\alpha l}\big(1+2^{-\alpha l}|\tau-\tau'|\big)^{-4}d\tau'\Big\|_{L^2_{\xi,\tau}} \\ 
& = 2^{-\frac{\alpha l}2}\Big\| \eta_{\le \tilde{l}}(\tau-w(\xi))\sum_{q=0}^{+\infty}\int_{\mathbb R}\eta_{q}(\tau'-w(\xi))|\phi(\xi,\tau')|\big(1+2^{-\alpha l}|\tau-\tau'|\big)^{-4}d\tau'\Big\|_{L^2_{\xi,\tau}} \\ 
& \le 2^{-\frac{\alpha l}2} 
\Big\| \eta_{\le \tilde{l}}(\tau-w(\xi))\sum_{q=0}^{+\infty}I_{q,l}(\xi,\tau)2^{\frac{q}2}\big\|\eta_{q}(\cdot-w(\xi))\phi(\xi,\cdot)\big\|_{L^2}\Big\|_{L^2_{\xi,\tau}},
\end{split}
\end{displaymath}
where 
\begin{displaymath}
I_{q,l}(\xi,\tau):= \big\|\widetilde{\eta}_{q}(\tau'-w(\xi))\big(1+2^{-\alpha l}|\tau-\tau'|\big)^{-4}\langle \tau'-w(\xi) \rangle^{-\frac12} \big\|_{L^2_{\tau'}}.
\end{displaymath} 
Now, we get trivially that 
\begin{displaymath} 
I_{q,l}(\xi,\tau) \lesssim 2^{\frac{q}2}2^{-\frac{q}2} \lesssim 1,
\end{displaymath}
which concludes the proof of \eqref{lemma2.1} recalling the definition of the space $X_k$ in \eqref{Xk}. 

Next, we turn to the proof of estimate \eqref{lemma2.2}. The mean-value theorem yields 
\begin{displaymath} 
\big|\eta_j(\tau-w(\xi))-\eta_j(\tau'-w(\xi)) \big| \lesssim 2^{-j}|\tau-\tau'|,
\end{displaymath}
which implies that 
\begin{equation} \label{lemma2.3} 
\sum_{j>\tilde{l}}2^{\frac j2}\Big\| \eta_{j}(\tau-w(\xi))\int_{\mathbb R}|\phi(\xi,\tau')|2^{-\alpha l}\big(1+2^{-\alpha l}|\tau-\tau'|\big)^{-4}d\tau'\Big\|_{L^2_{\xi,\tau}} \le II_a + II_b,
\end{equation}
where
\begin{displaymath} 
II_a:=\sum_{j > \tilde{l}}2^{\frac{j}2}\Big\|\big[ \eta_j(\cdot-w(\xi))\phi(\xi,\cdot)\big]\ast \big[ 2^{-\alpha l}(1+2^{-\alpha l}|\cdot|)^{-4}\big](\tau) \Big\|_{L^2_{\xi,\tau}},
\end{displaymath}
and 
\begin{displaymath} 
II_b:=\sum_{j > \tilde{l}}2^{-\frac{j}2}\Big\| \widetilde{\eta}_{j}(\tau-w(\xi))\int_{\mathbb R}|\phi(\xi,\tau')|2^{-\alpha l}|\tau-\tau'|\big(1+2^{-\alpha l}|\tau-\tau'|\big)^{-4}d\tau'\Big\|_{L^2_{\xi,\tau}}.
\end{displaymath}
Applying Young's theorem on convolution ($L^2_{\tau} \ast L^1_{\tau} \rightarrow L^2_{\tau}$), we get that  
\begin{equation} \label{lemma2.4} 
II_a \lesssim \sum_{j > \tilde{l}}2^{\frac{j}2}\big\| \eta_j(\tau-w(\xi))\phi(\xi,\tau)\big\|_{L^2_{\xi,\tau}} \le \|\phi\|_{X_k}.
\end{equation}
To deal with $II_b$ we just proceed as in the proof of estimate \eqref{lemma2.1} and obtain that 
\begin{displaymath} 
II_b \lesssim 2^{-\tilde{l}}\sum_{j > \tilde{l}}2^{-\frac{j}2}
\Big\| \widetilde{\eta}_{j}(\tau-w(\xi))\sum_{q=0}^{+\infty}J_{q,l}(\xi,\tau)2^{\frac{q}2}\big\|\eta_{q}(\cdot-w(\xi))\phi(\xi,\cdot)\big\|_{L^2}\Big\|_{L^2_{\xi,\tau}},
\end{displaymath}
where 
\begin{displaymath} 
J_{q,l}(\xi,\tau):= \big\|\widetilde{\eta}_{q}(\tau'-w(\xi))|\tau-\tau'|\big(1+2^{-\alpha l}|\tau-\tau'|\big)^{-4}\langle \tau'-w(\xi) \rangle^{-\frac12} \big\|_{L^2_{\tau'}}.
\end{displaymath}
In the case where $j \ge q+5$, we have that 
\begin{displaymath} 
J_{q,l}(\xi,\tau) \lesssim 2^j2^{-4(j-\tilde{l})},
\end{displaymath}
since $|\tau-\tau'| \sim 2^j$.  In the case where $q \ge j-4$, we get that 
\begin{displaymath} 
J_{q,l}(\xi,\tau) \lesssim 2^{-\frac{j}2}\big\||\cdot|(1+2^{-\alpha l}|\cdot|)^{-4} \big\|_{L^2} \lesssim 2^{-\frac{j}2}2^{\frac{3\widetilde{l}}2}. 
\end{displaymath}
Then, after summing in $j$, we deduce that in both cases 
\begin{equation} \label{lemma2.5}
II_b \lesssim \sum_{q=0}^{+\infty}2^{\frac{q}2}\big\|\eta_{q}(\tau-w(\xi))\phi(\xi,\tau)\big\|_{L^2_{\xi,\tau}}=\|\phi\|_{X_k}.
\end{equation}
Estimate \eqref{lemma2.2} follows gathering \eqref{lemma2.3}--\eqref{lemma2.5}, which concludes the proof of Lemma \ref{lemma2}.
\end{proof}

\begin{corollary} \label{coro1} 
Let $k \in \mathbb Z_+$, $\alpha \ge 0,$ $\tilde{t} \in \mathbb R$ and $\gamma \in \mathcal{S}(\mathbb R)$. Then it holds that 
\begin{equation} \label{coro1.1} 
\big\| \mathcal{F}\big[\gamma(2^{\alpha k}(\cdot-\tilde{t}))f \big]\big\|_{X_k} \lesssim \|\mathcal{F}(f)\|_{X_k},
\end{equation} 
for all $f$ such that $\mathcal{F}(f) \in X_k$.
\end{corollary}

\begin{proof} Since $\widehat{\gamma} \in \mathcal{S}(\mathbb R^2)$, we have that 
\begin{equation} \label{coro1.2}
\begin{split}
\Big| \mathcal{F}\big[\gamma(2^{\alpha k}(\cdot-\tilde{t}))f \big](\xi,\tau)\Big|
&=\Big| \mathcal{F}(f)(\xi,\cdot) \ast \big[e^{-i\tilde{t}(\cdot)}2^{-\alpha k}\widehat{\gamma}(2^{-\alpha k}(\cdot))\big](\tau)\Big|
\\ & \le \int_{\mathbb R} \big|\mathcal{F}(f)(\xi,\tau')\big| 2^{-\alpha k}\big(1+2^{-\alpha k}|\tau-\tau'| \big)^{-4}d\tau'.
\end{split}
\end{equation}
Therefore estimate \eqref{coro1.1} follows by using the definition of $X_k$ and applying estimates \eqref{lemma2.1}--\eqref{lemma2.2} to the right-hand side of \eqref{coro1.2}.
\end{proof}

\begin{corollary} \label{coro1b} 
Let $k \in \mathbb Z_+$, $\alpha \ge 0,$ $\tilde{t} \in \mathbb R$ and $\gamma \in \mathcal{S}(\mathbb R)$. Then it holds that 
\begin{equation} \label{coro1b.1} 
\big\|(\tau-w(\xi)+i2^{\alpha k})^{-1} \mathcal{F}\big[\gamma(2^{\alpha k}(\cdot-\tilde{t}))f \big]\big\|_{X_k} \lesssim \|(\tau-w(\xi)+i2^{\alpha k})^{-1}\mathcal{F}(f)\|_{X_k},
\end{equation} 
for all $f$ such that $\mathcal{F}(f) \in X_k$.
\end{corollary}

\begin{proof} 
We have that 
\begin{equation} \label{coro1b.2}
\begin{split}
\big\|(\tau-w(\xi)&+i2^{\alpha k})^{-1} \mathcal{F}\big[\gamma(2^{\alpha k}(\cdot-\tilde{t}))f \big]\big\|_{X_k} \\ & \lesssim 
2^{-\frac{[\alpha k]}2} \sum_{j \le [\alpha k]} 2^{\frac{j}2}\big\|\eta_j(\tau-w(\xi)) \mathcal{F}\big[\gamma(2^{\alpha k}(\cdot-\tilde{t}))f \big]\big\|_{L^2_{\xi,\tau}} \\ 
& \quad +  \sum_{j > [\alpha k]} 2^{-\frac{j}2}\big\|\eta_j(\tau-w(\xi)) \mathcal{F}\big[\gamma(2^{\alpha k}(\cdot-\tilde{t}))f \big]\big\|_{L^2_{\xi,\tau}}.
\end{split}
\end{equation}
We treat the first term on the right-hand side of \eqref{coro1b.2} by using Lemma \ref{lemma2} as in the proof of Corollary \ref{coro1} and the second one by using Lemma \ref{lemma2} and duality. This implies estimate \eqref{coro1b.1}.
\end{proof}

\begin{remark} \label{rema1} 
For $s \in \mathbb R_+$, the classical dyadic Bourgain space $X^{s,\frac12,1}$ (introduced for instance in \cite{Ta}) is defined by the norm
\begin{displaymath} 
\|f\|_{X^{s,\frac12,1}}=\Big(\|\mathcal{F}(P_{\le 0}f)\|_{X_0}^2+\sum_{k=1}^{+\infty}2^{2ks}\|\mathcal{F}(P_{k}f)\|_{X_k}^2 \Big)^{\frac12}.
\end{displaymath}
Thus, if $f \in X^{s,\frac12,1}$, one deduce after applying estimate \eqref{coro1.1} to each $P_kf$, taking the supreme in $\widetilde{t}$ and summing in $k$ that $\|f\|_{F^s_{\alpha}} \lesssim \|f\|_{X^{s,\frac12,1}}$, for any $\alpha \ge 0$. In other words, we have that 
\begin{displaymath}
X^{s,\frac12,1} \hookrightarrow F^s_{\alpha} \hookrightarrow L^{\infty}(\mathbb R;H^s(\mathbb R)).
\end{displaymath}
\end{remark} 

More generally for any $k \in \mathbb Z_+$ and $\alpha \ge 0$, we define the set $S_{k,\alpha}$ of \textit{$k$-acceptable time multiplication factors} (c.f. \cite{IKT}) as 
\begin{displaymath} 
S_{k,\alpha}=\Big\{ m_k: \mathbb R \rightarrow \mathbb R \ : \ \|m_k\|_{S_{k,\alpha}}=\sum_{j=0}^{10}2^{-j \alpha k}\|\partial_jm_k\|_{L^{\infty}}<\infty\Big\}.
\end{displaymath}
\begin{corollary}\label{coro3} 
Let $k \in \mathbb Z_+$, $\alpha \ge 0$ and $m_k \in S_{k,\alpha}$. Then it holds that 
\begin{equation} \label{coro3.1} 
\|m_k f\|_{F_{k,\alpha}} \lesssim \|m_k\|_{S_{k,\alpha}}\|f\|_{F_{k,\alpha}},
\end{equation}
and 
\begin{equation} \label{coro3.2} 
\|m_k f\|_{N_{k,\alpha}} \lesssim \|m_k\|_{S_{k,\alpha}}\|f\|_{N_{k,\alpha}}.
\end{equation}
\end{corollary}

\begin{proof} We prove estimate \eqref{coro3.1}. The proof of estimate \eqref{coro3.2} would follow in a similar way. Arguing as in the proof of Corollary \ref{coro1} it suffices to prove 
\begin{equation} \label{coro3.3} 
\big|\mathcal{F}_t\big[m_{k}(\cdot)\eta_0(2^{\alpha k}(\cdot-\tilde{t}))(\tau)\big] \big| \lesssim \|m_k\|_{S_{k,\alpha}}
2^{-\alpha k}(1+2^{-\alpha k}|\tau|)^{-4},
\end{equation}
for all $\tilde{t}, \ \tau \in \mathbb R$. 

It follows from the definition of the Fourier transform that 
\begin{equation} \label{coro3.4} 
\begin{split}
\big\|\mathcal{F}_t\big[m_{k}(\cdot)\eta_0(2^{\alpha k}(\cdot-\tilde{t}))\big] \big\|_{L^{\infty}} & \lesssim
\big\|m_{k}(\cdot)\eta_0(2^{\alpha k}(\cdot-\tilde{t})) \big\|_{L^1} \\ 
& \lesssim 2^{-\alpha k}\|m_k\|_{L^{\infty}}\|\eta_0\|_{L^1}.
\end{split}
\end{equation}
By using again basic properties of the Fourier transform and the Leibniz rule, we deduce that 
\begin{equation} \label{coro3.5} 
\begin{split}
2^{-4\alpha k}|\tau|^4\big|\mathcal{F}_t\big[m_{k}(\cdot)\eta_0&(2^{\alpha k}(\cdot-\tilde{t}))\big](\tau) \big| \\& \lesssim
2^{-4\alpha k}\big\|\partial_t^4\big[m_{k}(\cdot)\eta_0(2^{\alpha k}(\cdot-\tilde{t}))\big]\big\|_{L^1} \\ 
& \lesssim 2^{-4\alpha k}\sum_{j=0}^4\|\partial_t^jm_k\|_{L^{\infty}}2^{(4-j)\alpha k}2^{-\alpha k}\|\partial_t^{(4-j)}\eta_0\|_{L^1}.
\end{split}
\end{equation}
Estimates \eqref{coro3.4}--\eqref{coro3.5} and the definition of $S_{k,\alpha}$ imply estimate \eqref{coro3.3} which concludes the proof of Corollary \ref{coro3}.
\end{proof}

The next Corollary of Lemma \ref{lemma2} will be useful in the proof of the bilinear and energy estimates (c.f. Sections \ref{STBL} and \ref{EE}).
\begin{corollary} \label{coro2} 
Let  $\alpha \ge 0,$ $\tilde{t} \in \mathbb R$ and $l, \ k \in \mathbb Z_+$ be such that $l +5\ge  k$. Then it holds that 
\begin{equation} \label{coro2.1} 
2^{\frac{\alpha l}2}\big\|\eta_{\le [\alpha l]}(\tau-w(\xi)) \mathcal{F}\big[\eta_0(2^{\alpha l}(\cdot-\tilde{t}))f \big]\big\|_{L^2_{\xi,\tau}} \lesssim \|f\|_{F_{k,\alpha}},
\end{equation} 
and 
\begin{equation} \label{coro2.2} 
\sum_{j>[\alpha l]}2^{\frac j2}\big\|\eta_j(\tau-w(\xi)) \mathcal{F}\big[\eta_0(2^{\alpha l}(\cdot-\tilde{t}))f \big]\big\|_{L^2_{\xi,\tau}} \lesssim \|f\|_{F_{k,\alpha}},
\end{equation}
for all $f \in F_{k,\alpha}$.
\end{corollary}

\begin{proof} Observe that 
$$\eta_0(2^{\alpha l}(\cdot-\tilde{t}))f=\eta_0(2^{\alpha l}(\cdot-\tilde{t}))\eta_0(2^{\alpha (k-6)}(\cdot-\tilde{t}))f.$$ 
Moreover, it follows from Corollary \ref{coro3} that 
\begin{displaymath}
\big\| \mathcal{F}\big[\eta_0(2^{\alpha (k-6)}(\cdot-\tilde{t}))f \big]\big\|_{X^k} \lesssim \|f\|_{F_{k,\alpha}}. 
\end{displaymath}
Therefore, we conclude estimates \eqref{coro2.1} and \eqref{coro2.2} by applying  \eqref{lemma2.1} and \eqref{lemma2.2} arguing as it was done in the proof of Corollary \ref{coro1}.
\end{proof}

\begin{remark} Estimate \eqref{coro2.1} can be viewed as a consequence of the uncertainty Heisenberg principle. It is of fundamental importance in the proof of the short time bilinear estimates (c.f. Section \ref{STBL}), since it allows to consider only regions where the modulation $|\tau-w(\xi)|$ is not too small, and therefore to avoid the regions giving troubles in the low-high frequency interactions (c.f. \cite{Pi}).
\end{remark}

\subsection{Linear estimates} In this subsection, we derive the linear estimate associated to the spaces $F_{\alpha}^s(T)$ (c.f. \cite{IKT}).

\begin{proposition} \label{linear}
Assume $s \in \mathbb R_+$, $\alpha>0$ and $T \in (0,1]$. Then we have that 
\begin{equation} \label{linear1} 
\|u\|_{F_{\alpha}^s(T)} \lesssim \|u\|_{B^s(T)}+\|f\|_{N^s_{\alpha}(T)},
\end{equation}
for all  $u \in B^s(T)$ and $f \in N_{\alpha}^s$ satisfying 
\begin{equation} \label{linear2}
\partial_tu-\partial_x^5u=f, \quad \text{on} \quad \mathbb R\times[-T,T].
\end{equation} 
\end{proposition} 

\begin{remark} Observe that, when working in the classical Bourgain space $X^{0,\frac12,1}(T)$ defined in Remark \ref{rema1}, 
one would obtain an estimate of the form 
\begin{displaymath} 
\|u\|_{X^{s,\frac12,1}(T)} \lesssim \|u(0)\|_{H^s}+\|f\|_{X^{s,-\frac12,1}(T)}.
\end{displaymath}
Here, we need to introduce the energy norm $\|u\|_{B^s(T)}$ instead of $\|u(0)\|_{H^s}$, since we are working on very short time intervals, whose length depends on the spatial frequency.
\end{remark}

We first derive a homogeneous and a nonhomogeneous linear estimate in the spaces $X_k$.
\begin{lemma} \label{homogeneous} [Homogeneous linear estimate] 
Let $\alpha \ge 0$ and $k \in \mathbb Z_+$. Then it holds that 
\begin{equation} \label{homogeneous1} 
\big\| \mathcal{F}\big[\eta_0(2^{\alpha k}t)e^{it\partial_x^5}u_0 \big]\big\|_{X_k} \lesssim \|u_0\|_{L^2},
\end{equation}
for all $u_0 \in L^2(\mathbb R)$ such that $\text{supp} \, \mathcal{F}_x(u_0) \in I_k$.
\end{lemma} 

\begin{proof} A direct computation shows that 
\begin{displaymath} 
\mathcal{F}\big[\eta_0(2^{\alpha k}t)e^{it\partial_x^5}u_0 \big](\xi,\tau)=2^{-\alpha k}\widehat{\eta}_0\big(2^{-\alpha k}(\tau-w(\xi))\big)\widehat{u}_0(\xi).
\end{displaymath}
Thus, it follows from the definition of $X_k$ and Plancherel's identity that 
\begin{equation} \label{homogeneous2}
\big\| \mathcal{F}\big[\eta_0(2^{\alpha k}t)e^{it\partial_x^5}u_0 \big]\big\|_{X_k} \le 
\sum_{j \ge 0} 2^{j/2}\big\|\eta_j(\cdot)2^{-\alpha k}\widehat{\eta}_0(2^{-\alpha k}\cdot) \big\|_{L^2}\|u_0\|_{L^2}. 
\end{equation}
Moreover, it is clear since $\widehat{\eta}_0 \in \mathcal{S}(\mathbb R)$ that
\begin{displaymath} 
\|\eta_j(\cdot)2^{-\alpha k}\widehat{\eta}_0(2^{-\alpha k}\cdot) \big\|_{L^2} \lesssim 2^{-\alpha k}
\|\eta_j(\cdot)(1+2^{-\alpha k}|\cdot|)^{-4} \big\|_{L^2}  \lesssim 2^{-\alpha k}2^{j/2}\min(1,2^{4(\alpha k-j)}),
\end{displaymath}
which combined with \eqref{homogeneous2} implies estimate \eqref{homogeneous1}.
\end{proof}

\begin{lemma} \label{nonhomogeneous} [Non-homogeneous linear estimate] 
Let $\alpha \ge 0$ and $k \in \mathbb Z_+$. Then it holds that 
\begin{equation} \label{nonhomogeneous1} 
\big\| \mathcal{F}\big[\eta_0(2^{\alpha k}t)\int_0^te^{i(t-s)\partial_x^5}f(\cdot,s)ds \big]\big\|_{X_k} \lesssim 
\big\| (\tau-w(\xi)+i2^{\alpha k})^{-1}\mathcal{F}(f)\big\|_{X_k},
\end{equation}
for all $f$ such that $\text{supp} \, \mathcal{F}(f) \in I_k \times \mathbb R$.
\end{lemma}

\begin{proof} Straightforward computations yield
\begin{equation} \label{nonhomogeneous3}
\begin{split}
&\mathcal{F}\big[\eta_0(2^{\alpha k}t)\int_0^te^{i(t-s)\partial_x^5}f(\cdot,s)ds\big](\xi,\tau)
\\ &=\mathcal{F}_t\big[\eta_0(2^{\alpha k}t)\int_{\mathbb R}\frac{e^{it\tilde{\tau}}-e^{itw(\xi)}}{i(\tilde{\tau}-w(\xi))}\mathcal{F}(f)(\xi,\tilde{\tau})d\tilde{\tau}\big](\tau) \\ & 
= 2^{-\alpha k}\widehat{\eta}_0(2^{-\alpha k}\cdot) \ast \big[\frac{\mathcal{F}(f)(\xi,\cdot)}{i(\cdot-w(\xi))}\big](\tau)
-\mathcal{F}_t\big[\eta_0(2^{\alpha k}t)e^{itw(\xi)}\big](\tau)\int_{\mathbb R}\frac{\mathcal{F}(f)(\xi,\tilde{\tau})}{i(\tilde{\tau}-w(\xi))}d\tilde{\tau} \\ & 
=2^{-\alpha k}\int_{\mathbb R}\frac{\widehat{\eta}_0(2^{-\alpha k}(\tau-\tilde{\tau}))-\widehat{\eta}_0(2^{-\alpha k}(\tau-w(\xi)))}{i(\tilde{\tau}-w(\xi))}\mathcal{F}(f)(\xi,\tilde{\tau})d\tilde{\tau}.
\end{split}
\end{equation}

Now, we observe that 
\begin{equation} \label{nonhomogeneous2}
\begin{split}
2^{-\alpha k}\Big|&\frac{\widehat{\eta}_0(2^{-\alpha k}(\tau-\tilde{\tau}))-\widehat{\eta}_0(2^{-\alpha k}(\tau-w(\xi)))}{i(\tilde{\tau}-w(\xi))} \big( \tilde{\tau}-w(\xi)+i2^{\alpha k} \big) \Big| \\ & 
\lesssim 2^{-\alpha k}\big(1+2^{-\alpha k}|\tau-\tilde{\tau}| \big)^{-4}+2^{-\alpha k}\big(1+2^{-\alpha k}|\tau-w(\xi)| \big)^{-4}.
\end{split}
\end{equation}
Indeed, in the case where $|\tilde{\tau}-w(\xi)| \ge 2^{\alpha k}$, then $\big|\tilde{\tau}-w(\xi)+i2^{\alpha k}  \big| \lesssim |\tilde{\tau}-w(\xi)|$ and estimate \eqref{nonhomogeneous2} follows directly from the fact that $\widehat{\eta}_0 \in \mathcal{S}(\mathbb R)$ and the triangle inequality. Now we deal with the case where $|\tilde{\tau}-w(\xi)| \le 2^{\alpha k}$. We deduce by applying the mean value theorem to the radial function $\widehat{\eta}_0$ that 
\begin{displaymath} 
\big| \widehat{\eta}_0(2^{-\alpha k}(\tau-\tilde{\tau}))-\widehat{\eta}_0(2^{-\alpha k}(\tau-w(\xi)))\big| \le 2^{-\alpha k}|\widehat{\eta}_0'(2^{-\alpha k}\theta)||\tilde{\tau}-w(\xi)|,
\end{displaymath}
for some $\theta \in ] |\tau-\tilde{\tau}|,|\tau-w(\xi)|[$ or $\theta \in ] |\tau-w(\xi)|,|\tau-\tilde{\tau}|[$, depending wether $|\tau-\tilde{\tau}| < |\tau-w(\xi)|$  or $|\tau-w(\xi)|<|\tau-\tilde{\tau}|$. Thus, since $\widehat{\eta}_0' \in \mathcal{S}(\mathbb R)$, the left-hand side of \eqref{nonhomogeneous2} can be bounded by 
$2^{-\alpha k}\big( 1+2^{-\alpha k}|\theta|\big)^{-4}$. This implies estimate \eqref{nonhomogeneous2} in this case by using the assumption on $\theta$. 

On the one hand, we deduce from Lemma \ref{lemma2} that
\begin{equation}  \label{nonhomogeneous4}
\big\| \int_{\mathbb R}\frac{|\mathcal{F}(f)(\xi,\tilde{\tau})|}{|\tilde{\tau}-w(\xi)+i2^{\alpha k}|} 2^{-\alpha k}\big(1+2^{-\alpha k}|\tau-\tilde{\tau}| \big)^{-4}d\tilde{\tau} \big\|_{X_k} \lesssim 
\big\| \frac{\mathcal{F}(f)(\xi,\tau)}{\tau-w(\xi)+i2^{\alpha k}} \big\|_{X_k}.
\end{equation} 
On the other hand, it follows arguing as in the proof of Lemma \ref{homogeneous} and using estimate \eqref{lemma1.5} that
\begin{equation} \label{nonhomogeneous5}
\begin{split}
&\big\| 2^{-\alpha k}\big(1+2^{-\alpha k}|\tau-w(\xi)| \big)^{-4}\int_{\mathbb R}\frac{|\mathcal{F}(f)(\xi,\tilde{\tau})|}{|\tilde{\tau}-w(\xi)+i2^{\alpha k}|} d\tilde{\tau} \big\|_{X_k} \\ & \lesssim
\sum_{j \ge 0}2^{j/2}\big\|\eta_j(\cdot)2^{-\alpha k}(1+2^{-\alpha k}|\cdot|)^{-4}\big\|_{L^2}\big\|\int_{\mathbb R} \frac{|\mathcal{F}(f)(\xi,\tilde{\tau})|}{|\tilde{\tau}-w(\xi)+i2^{\alpha k}|} d\tilde{\tau}\big\|_{L^2_{\xi}} \\ 
& \lesssim \sum_{j \ge 0}2^j2^{-\alpha k}\min(1,2^{4(\alpha k-j)})\big\|\int_{\mathbb R} \frac{|\mathcal{F}(f)(\xi,\tilde{\tau})|}{|\tilde{\tau}-w(\xi)+i2^{\alpha k}|} d\tilde{\tau}\big\|_{L^2_{\xi}} \\ & 
\lesssim  \big\| \frac{\mathcal{F}(f)(\xi,\tau)}{\tau-w(\xi)+i2^{\alpha k}} \big\|_{X_k}.
\end{split}
\end{equation}

Finally, we conclude the proof of Proposition \ref{nonhomogeneous} gathering \eqref{nonhomogeneous3}--\eqref{nonhomogeneous5}.
\end{proof}

A proof of Proposition \ref{linear} is now in sight. 
\begin{proof}[Proof of Proposition \ref{linear}] 
Let $u, \ f:\mathbb R \times [-T,T]$ satisfying \eqref{linear2}. First, we choose an extension $\widetilde{f}$ of $f$ on $\mathbb R^2$ satisfying 
\begin{equation} \label{linear3}
\|\widetilde{f}\|_{N^s_{\alpha}} \le 2\|f\|_{N^s_{\alpha}(T)}. 
\end{equation} 

Fix $\theta \in C_0^{\infty}(\mathbb R)$ such that $\theta(t)=1$ if $t \ge 1$ and $\theta(t)=0$ if $t \le 0$. For $k \in \mathbb Z_+$, we define 
\begin{displaymath} 
\widetilde{f}_k=\theta\big(2^{\alpha k+10}(t+T+2^{-\alpha k-10})\big)
\theta\big(-2^{\alpha k+10}(t-T-2^{-\alpha k-10})\big)P_k\widetilde{f}.
\end{displaymath} 
Then, it follows from \eqref{coro3.2} and the definition of $\theta$ that
\begin{equation} \label{linear4}
\|\widetilde{f}_k\|_{N_{k,\alpha}} \lesssim \|P_{k}\widetilde{f}\|_{N_{k,\alpha}},
\end{equation}
\begin{displaymath}
\text{supp}\, \widetilde{f}_k \subset \mathbb R \times [-T-2^{-\alpha k-10},T+2^{-\alpha k -10}] \quad \text{and} \quad  \widetilde{f}_{k|_{[-T,T]}}=P_kf.
\end{displaymath}

Moreover, for all $k \in \mathbb Z_+$, we also extend $P_ku$ on $\mathbb R^2$, by defining $\widetilde{u}_k(t)$ as
\begin{displaymath} 
\left\{
\begin{array}{cll}
\eta_0(2^{\alpha k+5}(t-T))\big( e^{i(t-T)\partial_x^5}P_ku(T)+\int_T^te^{i(t-s)\partial_x^5}\widetilde{f}_k(s)ds\big) & \text{if} & t >T\\
P_ku(t) & \text{if} & t \in [-T,T] 
 \\ 
 \eta_0(2^{\alpha k+5}(t+T))\big( e^{i(t+T)\partial_x^5}P_ku(-T)+\int_{-T}^te^{i(t-s)\partial_x^5}\widetilde{f}_k(s)ds\big) & \text{if} & t <-T.
\end{array}
\right.
\end{displaymath}
Next, we show that 
\begin{equation} \label{linear5} 
\|\widetilde{u}_k\|_{F_{k,\alpha}} \lesssim \sup_{t_k \in [-T,T]}\big\|\mathcal{F}\big(\eta_0(2^{\alpha k}(t-t_k))\widetilde{u}_k\big) \big\|_{X_k}
\end{equation}
It is clear from the definition that $\widetilde{u}_k$ is supported in $\mathbb R \times [T-2^{-\alpha k-5},T+2^{-\alpha k-5}]$. Thus, if $t_k > T$, we get,
\begin{displaymath} 
\eta_0(2^{\alpha k}(t-t_k))\widetilde{u}_k=\eta_0(2^{\alpha k}(t-t_k))\eta_0(2^{\alpha k}(t-\tilde{t}_k))\widetilde{u}_k
\end{displaymath}
for some $\tilde{t}_k \in [T-2^{-\alpha k},T]$, so that \eqref{coro1.1} implies 
\begin{displaymath} 
\sup_{t_k >T}\big\| \mathcal{F}\big[\eta_0(2^{\alpha k}(\cdot-t_k))\widetilde{u}_k \big]\big\|_{X_k} \lesssim \sup_{\tilde{t}_k \in [-T,T]}\big\| \mathcal{F}\big[\eta_0(2^{\alpha k}(\cdot-\tilde{t}_k))\widetilde{u}_k \big]\big\|_{X_k}.\end{displaymath}
We could argue similarly for $t <T$, which implies estimate \eqref{linear5}.

Now we fix $t_k \in [-T,T]$. Observe that 
\begin{displaymath} 
\big\| \mathcal{F}\big[\eta_0(2^{\alpha k}(\cdot-t_k))\widetilde{u}_k \big]\big\|_{X_k}=
\big\| \mathcal{F}\big[\eta_0(2^{\alpha k}\cdot)\widetilde{u}_k(\cdot+t_k) \big]\big\|_{X_k}
\end{displaymath}
and by the Duhamel principle, 
\begin{displaymath} 
\eta_0(2^{\alpha k}t)\widetilde{u}_k(t+t_k)=m_k(t)\eta_0(2^{\alpha k}t)
\Big(e^{it\partial_x^5}P_ku(t_k)+\int_0^t e^{i(t-s)\partial_x^5}\widetilde{\eta}_0(2^{\alpha k}s)\widetilde{f}_k(s+t_k)ds\Big),
\end{displaymath}
where $m_k \in S_{k,\alpha}$. Thus, we deduce from estimates \eqref{coro3.1}, \eqref{homogeneous1} and \eqref{nonhomogeneous1} that 
\begin{displaymath} 
\big\| \mathcal{F}\big[\eta_0(2^{\alpha k}\cdot)\widetilde{u}_k(\cdot+t_k) \big]\big\|_{X_k} \lesssim \|P_ku(t_k)\|_{L^2}
+\big\| (\tau-w(\xi)+i2^{\alpha k})^{-1}\mathcal{F}(\widetilde{\eta}_0(2^{\alpha k}\cdot)\widetilde{f}_k(\cdot+t_k))\big\|_{X_k},
\end{displaymath}
which implies estimate \eqref{linear1} after taking the supreme in $t_k \in [-T,T]$, summing over $k \in \mathbb Z_+$ and using \eqref{coro3.2}, \eqref{linear3}--\eqref{linear5}.
\end{proof}

\subsection{Strichartz estimates}
We recall the Strichartz estimates associated to $\{e^{t\partial_x^5} \}$ proved by Kenig, Ponce and Vega in \cite{KPV}. 
\begin{proposition} \label{strichartz} 
Let $2 \le q, \ r \le +\infty$ and $0 \le s \le \frac3q$ satisfy $-s+\frac5q=\frac12$. Then, 
\begin{equation} \label{strichartz1} 
\|D^{s}_xe^{t\partial_x^5}u_0 \|_{L^q_tL^r_x} \lesssim \|u_0\|_{L^2},
\end{equation}
for all $u_0 \in L^2(\mathbb R)$.
\end{proposition}
As a consequence, we obtain a Strichartz estimate in the context of the Bourgain spaces $F^s_{\alpha}(T)$.
\begin{corollary} \label{Bstrichartz}
Assume $0<T \le 1$, $\alpha \ge 0$ and $\epsilon>0$. Then, it holds that 
\begin{equation} \label{Bstrichartz1} 
\|D^{\frac34-\frac{\alpha}4}_xu\|_{L^2_TL^{\infty}_x} \lesssim \|u\|_{F_{\alpha}^{\epsilon}(T)},
\end{equation}
and 
\begin{equation} \label{Bstrichartz1b} 
\|D^{\frac34-\frac{\alpha}4}_xu\|_{\widetilde{L^2_TL^{\infty}_x}} := \Big( \sum_{k \ge 0} \|D^{\frac34-\frac{\alpha}4}_xP_ku\|_{L^2_TL^{\infty}_x}^2 \Big)^{\frac12}\lesssim \|u\|_{F_{\alpha}^{\epsilon}(T)},
\end{equation}
for any $u \in F_{\alpha}^{\epsilon}(T)$.
\end{corollary}

\begin{proof} Let $0<T \le 1$, $\alpha \ge 0$, $\epsilon>0$ and $u \in F_{\alpha}^{\epsilon}(T)$. Choose $\widetilde{u} \in F^{\epsilon}_{\alpha}$ such that 
\begin{displaymath}
\widetilde{u}_{|_{[-T,T]}}=u \quad \text{and} \quad
\|\widetilde{u}\|_{F_{\alpha}^{\epsilon}} \le 2 \|u\|_{F_{\alpha}^{\epsilon}(T)}.
\end{displaymath} 
For $k \in \mathbb Z_+$, we denote $\widetilde{u}_k=P_k\widetilde{u}$ (recall that $P_0=P_{\le 0}$). Then we deduce using the Sobolev embedding $W^{\epsilon',r}(\mathbb R) \hookrightarrow L^{\infty}(\mathbb R)$, the square function theorem and Minkowski's inequality that
\begin{displaymath} 
\|D^{\frac34}_xu\|_{L^2_TL^{\infty}_x} \lesssim \|D^{\frac34}_xJ_x^{\epsilon'}\widetilde{u}\|_{L^2_TL^r_x} \lesssim 
\Big( \sum_{k \ge 0}2^{2k\epsilon'}\|D^{\frac34}_x\widetilde{u}_k\|_{L^2_TL^r_x}^2\Big)^{\frac12}.
\end{displaymath}
where $\epsilon'$ and $r(>1/\epsilon')$ will be chosen later. Therefore, according to the definition of $F^{\epsilon}_{\alpha}$ in \eqref{Fs}, it suffices to prove that 
\begin{equation} \label{Bstrichartz2} 
\|D^{\frac34}_x\widetilde{u}_k\|_{L^2_TL^r_x} \lesssim 2^{k(\frac{\alpha}4+\frac{3-\alpha}{2r})}\|\widetilde{u}_k\|_{F_{k,\alpha}},
\end{equation}
for all $k \ge 0$ in order to prove estimate \eqref{Bstrichartz1}. Indeed, it is enough then to choose $r$ and $\epsilon'$ such that $r\epsilon'>1$ and $\epsilon'+\frac{3-\alpha}{2r}<\epsilon$.

Next, we prove estimate \eqref{Bstrichartz2}. For $k \ge 0$, we chop the interval $[-T,T]$ in subintervals $I_j$ of length $2^{-\alpha k}$. Let $[-T,T]=U_{j}I_j$ where $\big| I_j\big| \sim 2^{-\alpha k}$ and $\eta_0(2^{\alpha k}(\cdot-c_j))=1$ on $I_j$ (here $c_j$ denotes the center of $I_j$). Note that the number of intervals $I_j$ is of order $2^{\alpha k}T$. Let $2 \le q$ be so that $-\frac3q+\frac5q+\frac1r=\frac12$. Then, we deduce applying H\"older's inequality in time that 
\begin{equation} \label{Bstrichartz3} 
\begin{split} 
\|D^{\frac34}_x\widetilde{u}_k\|_{L^2_TL^r_x}&=\Big(\sum_{j}\big\|D_x^{\frac34}\widetilde{u}_k \eta_0(2^{\alpha k}(\cdot-c_j))\big\|_{L^2_{I_j}L^r_x}^2 \Big)^{\frac12} \\ 
& \lesssim 2^{-\alpha k (\frac12-\frac1q)}\Big(\sum_{j}\big\|D_x^{\frac34}\widetilde{u}_k \eta_0(2^{\alpha k}(\cdot-c_j))\big\|_{L^q_{I_j}L^r_x}^2 \Big)^{\frac12}.
\end{split}
\end{equation}

Due to the Fourier inversion formula, we have that 
\begin{displaymath}
 D^{\frac34}_x\widetilde{u}_k(x,t)\eta_0(2^{\alpha k}(t-c_j))=c\int_{\mathbb R}D^{\frac34}_xe^{t\partial_x^5}\mathcal{F}_s\big(e^{-s\partial_x^5}\eta_0(2^{\alpha k}(s-c_j))\widetilde{u}_k(\cdot,s) \big)(\tau)e^{it\tau}d\tau.
\end{displaymath}
Thus, Minkowski's inequality,  estimate \eqref{strichartz1}, Plancherel's identity and the Cauchy-Schwarz inequality in $q$ imply that
\begin{equation} \label{Bstrichartz4}
\begin{split} 
\|D^{\frac34}_x\widetilde{u}_k\|_{L^q_{I_j}L^r_x} &\lesssim \int_{\mathbb R}\big\| D^{\frac34}_xe^{t\partial_x^5}\mathcal{F}_s\big(e^{-s\partial_x^5}\eta_0(2^{\alpha k}(s-c_j))\widetilde{u}_k \big)\big\|_{L^q_{I_j}L^r_x}d\tau \\ 
& \lesssim 2^{k(\frac34-\frac3q)}\sum_{q \ge 0}2^{q/2}\big\|\eta_q(\tau)\mathcal{F}\big(e^{-\cdot\partial_x^5}\eta_0(2^{\alpha k}(\cdot-c_j))\widetilde{u}_k \big)(\xi,\tau) \big\|_{L^2_{\xi,\tau}}.
\end{split}
\end{equation}
Then, we observe that 
\begin{displaymath} 
\mathcal{F}\big(e^{-\cdot\partial_x^5}\eta_0(2^{\alpha k}(\cdot-c_j))\widetilde{u}_k \big)(\xi,\tau)
=\mathcal{F}\big(\eta_0(2^{\alpha k}(\cdot-c_j))\widetilde{u}_k \big)(\xi,\tau+w(\xi)),
\end{displaymath}
which together with \eqref{Bstrichartz4} and the definition of $F_{k,\alpha}$ in \eqref{Fk} implies that 
\begin{equation} \label{Bstrichartz5}
\|D^{\frac34}_x\widetilde{u}_k\|_{L^q_{I_j}L^r_x} \lesssim 2^{k(\frac34-\frac3q)}\|\widetilde{u}_k\|_{F_{k,\alpha}}.
\end{equation}

Finally, we deduce combining \eqref{Bstrichartz3} and \eqref{Bstrichartz5} that 
\begin{displaymath} 
\|D^{\frac34}_x\widetilde{u}_k\|_{L^2_TL^r_x} \lesssim 2^{k(\frac{\alpha}q+\frac34-\frac3q)}\|\widetilde{u}_k\|_{F_{k,\alpha}},
\end{displaymath}
which yields estimate \eqref{Bstrichartz2} since $\frac1q=\frac14-\frac1{2r}$. The proof of estimate \eqref{Bstrichartz1b} is similar.
\end{proof}

Next, we derive a bilinear Strichartz estimate for the group $\{e^{t\partial_x^5}\}$, which is an extension of the one proved in \cite{Gr2} for the Airy equation (see also Lemma 3.4 in \cite{He} for the dispersion generalized Benjamin-Ono equation). Let $\zeta \in C^{\infty}$ be an even function such that $\zeta_{|_{[-1,1]}}=0$, $\zeta_{|_{\mathbb R \setminus [-2,2]}}=1$ and $0 \le \zeta \le1$.  We define $|x|_1=\zeta(x)|x|$. 

\begin{lemma} \label{bilinStrichartz}
For $s \in \mathbb R$, we define the bilinear operator $\mathcal{I}^s$ by 
\begin{displaymath}
\mathcal{F}_x\big(\mathcal{I}^s(u_1,u_2)\big)(\xi)=\int_{\xi=\xi_1+\xi_2}\big| |\xi_1|^{2s}-|\xi_2|^{2s}\big|_1^{\frac12}\widehat{u}_1(\xi_1)\widehat{u}_2(\xi_2)d\xi_1.
\end{displaymath}
Then, it holds that 
\begin{equation} \label{bilinStrichartz1} 
\big\|\mathcal{I}^2(e^{t\partial_x^5}u_1,e^{t\partial_x^5}u_2) \big\|_{L^2_{x,t}} \lesssim \|u_1\|_{L^2}\|u_2\|_{L^2},
\end{equation}
for any $u_1, \ u_2 \in L^2(\mathbb R)$.
\end{lemma}

\begin{proof} For a fixed $t \in \mathbb R$, we get by using Plancherel's identity that 
\begin{displaymath} 
\begin{split}
\big\|\mathcal{I}^2(e^{t\partial_x^5}u_1&,e^{t\partial_x^5}u_2)  \big\|_{L^2_x}^2  
\\ &=\int_{\mathbb R} \Big|\int_{\xi=\xi_1+\xi_2}\big| |\xi_1|^{4}-|\xi_2|^{4}\big|_1^{\frac12}e^{it(\xi_1^5+\xi_2^5)}\widehat{u}_1(\xi_1)\widehat{u}_2(\xi_2)d\xi_1\Big|^2d\xi \\ & 
=\int_{\mathbb R^3}e^{it\phi(\xi,\xi_1,\eta_1)}f(\xi,\xi_1,\eta_1)d\xi d\xi_1d\eta_1,
\end{split}
\end{displaymath}
where the phase function $\phi$ is given by 
\begin{displaymath} 
\phi(\xi,\xi_1,\eta_1):=\xi_1^5+(\xi-\xi_1)^5-\eta_1^5-(\xi-\eta_1)^5,
\end{displaymath}
and $f$ is defined by 
\begin{displaymath}
\begin{split} 
f(\xi&,\xi_1,\eta_1) \\ & :=\big| |\xi_1|^{4}-|\xi-\xi_1|^{4}\big|_1^{\frac12}
\big| |\eta_1|^{4}-|\xi-\eta_1|^{4}\big|_1^{\frac12}\widehat{u}_1(\xi_1)\widehat{u}_2(\xi-\xi_1)
\overline{\widehat{u}_1(\eta_1)\widehat{u}_2(\xi-\eta_1)}.
\end{split}
\end{displaymath}

Now, observed that for $(\xi,\xi_1)$ fixed, the function $\phi_1(\eta_1):=\phi(\xi,\xi_1,\eta_1)$ has only two simple roots $\eta_1=\xi_1$ and $\eta_1=\xi-\xi_1$ in the support of $f$. Moreover,
\begin{displaymath} 
|\phi_1'(\eta_1)|=5\big|(\eta_1^4-(\xi-\eta_1)^4) \big| \ge 5 \quad \text{in} \ \text{supp}\, f, 
\end{displaymath}
and 
\begin{displaymath}
|\phi_1'(\xi_1)|=|\phi_1'(\xi-\xi_1)|=5\big| \xi_1^4-(\xi-\xi_1)^4\big|.
\end{displaymath}
Therefore, it follows from the Fourier inversion formula, Fubini's theorem and Plancherel's identity that 
\begin{displaymath} 
\begin{split}
\big\|\mathcal{I}^2(e^{t\partial_x^5}u_1,e^{t\partial_x^5}u_2)  \big\|_{L^2_{x,t}}^2
&=c\int_{\mathbb R^3}\delta_0(\phi(\xi,\xi_1,\eta_1))f(\xi,\xi_1\eta_1)d\eta_1d\xi d\xi_1\\ 
&=c\int_{\mathbb R^2}\Big(\frac{f(\xi,\xi_1,\xi_1)}{\big|\phi_1'(\xi_1) \big|}+\frac{f(\xi,\xi_1,\xi-\xi_1)}{\big|\phi_1'(\xi-\xi_1) \big|}\Big)d\xi d\xi_1 \\ 
& \lesssim \|u_1\|_{L^2_x}^2\|u_2\|_{L^2_x}^2.
\end{split}
\end{displaymath}
\end{proof}

\section{$L^2$ bilinear and trilinear estimates} 

\subsection{$L^2$ bilinear estimates}
Recall that $w(\xi)=\xi^5$. Then we define the resonance functions $\Omega:=\Omega(\xi_1,\xi_2)$ by 
\begin{equation} \label{resonnance} 
\Omega(\xi_1,\xi_2):=w(\xi_1)+w(\xi_2)-w(\xi_1+\xi_2).
\end{equation}
We first derive a technical lemma  (see Lemma 3.1 in \cite{CLMW}).
\begin{lemma} \label{lresonance} 
If $|\xi_1| \sim N_1$, $\xi_2 \sim N_2$ and $|\xi_1+\xi_2| \sim N$, then 
\begin{equation} \label{lresonance1} 
\big|\Omega\big| \sim N_{max}^4N_{min}.
\end{equation}
\end{lemma} 

\begin{proof} A direct computation shows that 
\begin{equation} \label{lresonance2}
\Omega(\xi_1,\xi_2)=-5\xi_1\xi_2(\xi_1+\xi_2)\big(\xi_1^2+\xi_1\xi_2+\xi_2^2  \big).
\end{equation} 
We affirm that 
\begin{equation} \label{lresonance3}
\xi_1^2+\xi_1\xi_2+\xi_2^2 \sim \max\{\xi_1^2,\xi_2^2 \}, 
\end{equation}
which, together with \eqref{lresonance2}, would implies \eqref{lresonance1}.  

Next, we prove \eqref{lresonance3}. It is clear that $\xi_1^2+\xi_1\xi_2+\xi_2^2 \lesssim \max\{\xi_1^2,\xi_2^2 \}$. To prove the reverse inequality, we can always assume by symmetry that $|\xi_2| \le |\xi_1|$. Then in the case where 
$|\xi_2| \le \frac12|\xi_1|$, we have that 
\begin{displaymath} 
\xi_1^2+\xi_1\xi_2+\xi_2^2=(\xi_1+\xi_2)^2-\xi_1\xi_2 \ge \frac12(\xi_1+\xi_2)^2 \ge \frac18\xi_1^2.
\end{displaymath} 
In the other case, \textit{i.e.} $\frac12|\xi_1| \le |\xi_2| \le |\xi_1|$, then 
\begin{displaymath} 
\xi_1^2+\xi_1\xi_2+\xi_2^2 \ge \xi_1^2-|\xi_1\xi_2|+\xi_2^2 \ge \frac14\xi_1^2,
\end{displaymath}
which concludes the proof of \eqref{lresonance3}.
\end{proof}

For $k \in \mathbb Z_+$ and $j \in \mathbb Z_+$, let us define $D_{k,j}$ by 
\begin{equation} \label{D}
D_{k,j}=\big\{(\xi,\tau) : \xi \in I_k \ \text{and} \ |\tau+w(\xi)| \le 2^j \big\}. 
\end{equation}

We state a useful lemma (see also Lemma 2.3 in \cite{CGL}). 
\begin{lemma} \label{L2bilin}
Assume that $k_1, \ k_2, \ k_3 \in \mathbb Z_+$, $j_1, \ j_2, \ j_3 \in \mathbb Z_+$ and $f_i:\mathbb R^2 \to \mathbb R_+ $ are $L^2$ functions supported in $D_{k_i,j_i}$ for $i=1,2,3$. 
\begin{itemize}
\item[(a)]
Then it follows that 
\begin{equation} \label{L2bilin1}
\int_{\mathbb R^2}\big(f_1\ast f_2\big)\cdot f_3 \lesssim 2^{k_{min}/2}2^{j_{min}/2}\|f_1\|_{L^2}\|f_2\|_{L^2}\|f_3\|_{L^2}.
\end{equation}
\item[(b)]
Let us suppose that $k_{min} \le k_{max}-5$. If we are in the case where $(k_i,j_i)=(k_{min},j_{max})$ for some $i \in \{1,2,3\}$, then it holds that
\begin{equation} \label{L2bilin2b}
\int_{\mathbb R^2}\big(f_1\ast f_2\big)\cdot f_3 \lesssim 2^{(j_1+j_2+j_3)/2}2^{-j_{med}/2}2^{-2k_{max}}\|f_1\|_{L^2}\|f_2\|_{L^2}\|f_3\|_{L^2}.
\end{equation}
If moreover $k_{min} \ge 1$, then 
\begin{equation} \label{L2bilin2}
\int_{\mathbb R^2}\big(f_1\ast f_2\big)\cdot f_3 \lesssim 2^{(j_1+j_2+j_3)/2}2^{-j_{max}/2}2^{-(3k_{max}+k_{min})/2}\|f_1\|_{L^2}\|f_2\|_{L^2}\|f_3\|_{L^2}.
\end{equation}

In all the others cases, we have that 
\begin{equation} \label{L2bilin3}
\int_{\mathbb R^2}\big(f_1\ast f_2\big)\cdot f_3 \lesssim 2^{(j_1+j_2+j_3)/2}2^{-j_{max}/2}2^{-2k_{max}}\|f_1\|_{L^2}\|f_2\|_{L^2}\|f_3\|_{L^2}.
\end{equation}
\item[(c)] In the case $|k_{min}-k_{max}| \le 10$, $k_{min} \ge 10$, then we have that
\begin{equation} \label{L2bilin3b}
\int_{\mathbb R^2}\big(f_1\ast f_2\big)\cdot f_3 \lesssim 2^{j_{min}/2}2^{j_{med}/4}2^{-3k_{max}/4}\|f_1\|_{L^2}\|f_2\|_{L^2}\|f_3\|_{L^2}.
\end{equation}
\end{itemize}
\end{lemma}

\begin{proof} First, we begin with the proof of item (a).  We observe that 
\begin{equation} \label{L2bilin4} 
I:=\int_{\mathbb R^2}\big(f_1\ast f_2\big)\cdot f_3=\int_{\mathbb R^2}\big(\widetilde{f_1}\ast f_3\big)\cdot f_2
=\int_{\mathbb R^2}\big(\widetilde{f_2}\ast f_3\big)\cdot f_1,
\end{equation}
where $\widetilde{f_i}(\xi,\tau)=f_i(-\xi,-\tau)$. Therefore, we can always assume that $j_1=j_{min}$. Moreover, 
let us define $f^{\sharp}_i(\xi,\theta)=f_i(\xi,\theta+w(\xi))$, for $i=1,2,3$. In view of the assumptions on $f_i$, the functions $f^{\sharp}_i$ are supported in the sets 
\begin{displaymath} 
D_{k_i,j_i}^{\sharp}=\big\{(\xi,\theta) : \xi \in I_{k_i} \ \text{and} \ |\theta| \le 2^{j_i} \big\}.
\end{displaymath}
We also note that $\|f_i\|_{L^2}=\|f_i^{\sharp}\|_{L^2}$. Then, it follows changing variables that 
\begin{equation} \label{L2bilin5}
I=
\int_{\mathbb R^4}f^{\sharp}_1(\xi_1,\theta_1)f^{\sharp}_2(\xi_2,\theta_2)
f^{\sharp}_3(\xi_1+\xi_2,\theta_1+\theta_2+\Omega(\xi_1,\xi_2))d\xi_1d\xi_2d\theta_1d\theta_2,
\end{equation}
where $\Omega(\xi_1,\xi_2)$ is defined in \eqref{resonnance}.
For $i=1,2,3$, we define $F_i(\xi)=\big(\int_{\mathbb R}f^{\sharp}_i(\xi,\theta)^2d\theta \big)^{\frac12}$. 
Thus, it follows by applying the Cauchy-Schwarz and Young inequalities in the $\theta$ variables that 
\begin{equation} \label{L2bilin6} 
\begin{split}
I &\le \int_{\mathbb R^2}\|f^{\sharp}_1(\xi_1,\cdot)\|_{L^1_{\theta}}F_2(\xi_2)F_3(\xi_1+\xi_2)d\xi_1d\xi_2 \\ 
& \lesssim 2^{j_{min}/2} \int_{\mathbb R^2}F_1(\xi_1)F_2(\xi_2)F_3(\xi_1+\xi_2)d\xi_1d\xi_2.
\end{split}
\end{equation}
 Estimate \eqref{L2bilin1} is deduced from \eqref{L2bilin6} by applying the same arguments in the $\xi$ variables.

Next we turn to the proof of item (b). According to \eqref{L2bilin4}, we can assume that $j_3=j_{max}$. Moreover, it is enough to consider the two cases $k_{min}=k_2$ and $k_{min}=k_3$  (since by symmetry the case $k_{min}=k_1$ is equivalent to the case $k_{min}=k_2$). 

We prove estimate \eqref{L2bilin3} in the case $j_3=j_{max}$ and $k_{min}=k_2$. It suffices to prove that if $g_i:\mathbb R \to \mathbb R_+$ are $L^2$ functions supported in $I_{k_i}$ for $i=1,2$ and $g: \mathbb R^2 \to \mathbb R_+$ is an $L^2$ function supported in $I_{k_3} \times [-2^{j_3},2^{j_3}]$, then 
\begin{equation} \label{L2bilin7} 
J(g_1,g_2,g):=\int_{\mathbb R^2}g_1(\xi_1)g_2(\xi_2)g(\xi_1+\xi_2,\Omega(\xi_1,\xi_2))d\xi_1d\xi_2
\end{equation} 
satisfies that 
\begin{equation} \label{L2bilin8}
J(g_1,g_2,g) \lesssim  2^{-2k_{max}}\|g_1\|_{L^2}\|g_2\|_{L^2}\|g\|_{L^2}.
\end{equation} 
Indeed, if estimate \eqref{L2bilin8} holds,  let us define $g_i(\xi_i)=f_i^{\sharp}(\xi_i,\theta_i)$, $i=1,2$,  and $g(\xi,\Omega)=f_3^{\sharp}(\xi,\theta_1+\theta_2+\Omega)$, for $\theta_1$ and $\theta_2$ fixed. Hence, we would deduce applying \eqref{L2bilin8} and the Cauchy-Schwarz inequality to \eqref{L2bilin5} that 
\begin{equation} \label{L2bilin9} 
\begin{split}
I &\lesssim 2^{-2k_{max}}\|f_3^{\sharp}\|_{L^2_{\xi,\theta}}\int_{\mathbb R^2}\|f_1^{\sharp}(\cdot,\theta_1)\|_{L^2_{\xi}}\|f_2^{\sharp}(\cdot,\theta_2)\|_{L^2_{\xi}}d\theta_1d\theta_2 \\ 
& \lesssim 2^{-2k_{max}}2^{(j_1+j_2)/2}\|f_1^{\sharp}\|_{L^2_{\xi,\theta}}\|f_2^{\sharp}\|_{L^2_{\xi,\theta}}\|f_3^{\sharp}\|_{L^2_{\xi,\theta}},
\end{split}
\end{equation}
which is estimate \eqref{L2bilin3} in this case. To prove estimate \eqref{L2bilin8}, we apply twice the Cauchy-Schwarz inequality to get that 
\begin{displaymath}  
J(g_1,g_2,g) \le \|g_1\|_{L^2}\|g_2\|_{L^2}\Big( \int_{\mathbb R^2}g(\xi_1+\xi_2,\Omega(\xi_1,\xi_2))^2d\xi_1d\xi_2\Big)^{\frac12}.
\end{displaymath}
Then we change variables $(\xi_1',\xi_2')=(\xi_1+\xi_2,\xi_2)$, so that 
\begin{equation} \label{L^2bilin10}
J(g_1,g_2,g) \le \|g_1\|_{L^2}\|g_2\|_{L^2}
=\Big( \int_{\mathbb R^2}g(\xi_1',\Omega(\xi_1'-\xi_2',\xi_2'))^2d\xi_1'd\xi_2'\Big)^{\frac12}.
\end{equation}
We observe that 
\begin{displaymath}
\Big|\frac{\partial }{\partial \xi_2'}\Omega(\xi_1'-\xi_2',\xi_2') \Big|=5\big|(\xi_2')^4-(\xi_1'-\xi_2')^4 \big| \sim 2^{4k_{max}},
\end{displaymath}
since $2^{k_1}\sim 2^{k_{max}}$ by the frequency localization. Then, the change of variables $\mu_1=\xi_1'$ and $\mu_2=\Omega(\xi_1'-\xi_2',\xi_2')$ in \eqref{L^2bilin10} yields \eqref{L2bilin8}, which concludes the proof of estimate \eqref{L2bilin3} in this case. 

To prove estimate \eqref{L2bilin2} in the case $(k_{min},j_{max})=(k_3,j_3)$ and $k_3 \ge 1$, we observe arguing as above that it suffices to prove that 
\begin{equation} \label{L2bilin11} 
J(g_1,g_2,g) \lesssim  2^{-(3k_{max}+k_{min})/2}\|g_1\|_{L^2}\|g_2\|_{L^2}\|g\|_{L^2}, 
\end{equation}
 where $J(g_1,g_2,g)$ is defined in \eqref{L2bilin7}. First, we change variables $\xi_1'=\xi_1$ and $\xi_2'=\xi_1+\xi_2$, so that 
\begin{displaymath} 
J(g_1,g_2,g)=\int_{\mathbb R^2}g_1(\xi_1')g_2(\xi_2'-\xi_1')g(\xi_2',\Omega(\xi_1',\xi_2'-\xi_1'))d\xi_1'd\xi_2'.
\end{displaymath}
The Cauchy-Schwarz inequality implies that
\begin{equation} \label{L^2bilin12} 
J(g_1,g_2,g) \le \|g_1\|_{L^2}\|g_2\|_{L^2}\Big( \int_{\mathbb R^2}g(\xi_2',\Omega(\xi_1',\xi_2'-\xi_1'))^2d\xi_1'd\xi_2'\Big)^{\frac12}.
\end{equation}
We compute that
\begin{displaymath}
\Big|\frac{\partial }{\partial \xi_1'}\Omega(\xi_1',\xi_2'-\xi_1') \Big|=5\big|(\xi_1')^4-(\xi_2'-\xi_1')^4 \big| \sim 2^{3k_{max}+k_{min}},
\end{displaymath}
since $|\xi_1'| \sim 2^{k_{max}}$ and $|\xi_2'|\sim 2^{k_{min}}$ due to the frequency localization. Therefore estimate \eqref{L2bilin11} is deduced by performing the change of variables $\mu_1'=\Omega(\xi_1',\xi_2'-\xi_1')$ and $\mu_2'=\xi_2'$ in \eqref{L^2bilin12}. On the other hand, by writing, 
\begin{displaymath} 
I=\int_{\mathbb R^2}(\widetilde{f}_1\ast f_3)\cdot f_2
\end{displaymath}
and arguing as in \eqref{L2bilin9}, we get estimate \eqref{L2bilin2b} in the case $(k_{min},j_{max})=(k_3,j_3)$. 

Estimate \eqref{L2bilin3b} is stated in Lemma  2.3 (c) of \cite{CGL} and its proof follows closely the one for the dispersion generalized BO in \cite{Guo}. However, for sake of completeness we will derive it here.  According to \eqref{L2bilin4}, we may assume that $j_{max}=j_3$. Furthermore, we have following \eqref{L2bilin5} that
\begin{equation} \label{L2bilin13} 
\begin{split}
I&=\sum_{i=1}^3\int_{\mathcal{R}_i}f^{\sharp}_1(\xi_1,\theta_1)f^{\sharp}_2(\xi_2,\theta_2)
f^{\sharp}_3(\xi_1+\xi_2,\theta_1+\theta_2+\Omega(\xi_1,\xi_2))d\xi_1d\xi_2d\theta_1d\theta_2 \\ & 
=: I_1+I_2+I_3,
\end{split}
\end{equation}
where 
\begin{displaymath}
\begin{split}
\mathcal{R}_1&=\big\{(\xi_1,\xi_2,\theta_1,\theta_2) \in \mathbb R^4 \ : \ \xi_1 \cdot \xi_2 <0 \big\}, \\ 
\mathcal{R}_2&=\big\{(\xi_1,\xi_2,\theta_1,\theta_2) \in \mathbb R^4 \ : \ \xi_1 \cdot \xi_2 >0 \ \text{and} \ |\xi_1-\xi_2|< R \big\}, \\
\mathcal{R}_3&=\big\{(\xi_1,\xi_2,\theta_1,\theta_2) \in \mathbb R^4 \ : \ \xi_1 \cdot \xi_2 >0 \ \text{and} \ |\xi_1-\xi_2| >R  \big\},
\end{split}
\end{displaymath}
and  $R$ is a positive number which will be chosen later.

First we prove that 
\begin{equation} \label{L2bilin14} 
I_1 \lesssim 2^{(j_{min}+j_{med})/2}2^{-2k_{max}} \|f_1\|_{L^2}\|f_2\|_{L^2}\|f_3\|_{L^2},
\end{equation}
which would imply
\begin{equation} \label{L2bilin15} 
I_1 \lesssim 2^{j_{min}/2}2^{j_{med}/4}2^{-k_{max}} \|f_1\|_{L^2}\|f_2\|_{L^2}\|f_3\|_{L^2},
\end{equation}
after interpolating with estimate \eqref{L2bilin1}. 

To prove \eqref{L2bilin14}, we argue as for \eqref{L2bilin3}, so that it suffices to prove 
\begin{equation} \label{L2bilin16} 
J(g_1,g_2,g) \lesssim 2^{-2k_{max}}\|g_1\|_{L^2}\|g_2\|_{L^2}\|g\|_{L^2},
\end{equation}
where $J(g_1,g_2,g)$ is defined as in \eqref{L2bilin8}. By symmetry, we can always assume that $|\xi_1| \le |\xi_2|$. We apply twice the Cauchy-Schwarz inequality and perform the change of variables $(\xi_1',\xi_2')=(\xi_1,\xi_1+\xi_2)$ to obtain that 
\begin{equation} \label{L2bilin17}
J(g_1,g_2,g) \le \|g_1\|_{L^2}\|g_2\|_{L^2}\Big( \int_{\mathbb R^2}g(\xi_2',\Omega(\xi_1',\xi_2'-\xi_1'))^2d\xi_1'd\xi_2'\Big)^{\frac12}.
\end{equation}
Now observe that 
\begin{displaymath}
\big|\frac{\partial}{\partial\xi_1'}\Omega(\xi_1',\xi_2'-\xi_1') \big| =5\big| (\xi_2')^4-4(\xi_2')^3\xi_1'+6(\xi_2')^2(\xi_1')^2-4\xi_2'(\xi_1')^3\big| \sim 2^{4k_{max}},
\end{displaymath}
due to the frequency localization and the restriction $\xi_1' \cdot \xi_2' \le 0$ (which is a consequence of the assumptions $\xi_1 \cdot \xi_2 <0$ and $|\xi_1| \le |\xi_2|$). Therefore, the change of variables $(\mu_1',\mu_2')=(\Omega(\xi_1',\xi_2'-\xi_1'),\xi_2')$ in \eqref{L2bilin17} yields estimate \eqref{L2bilin16}.

To deal with $I_2$, we get as in \eqref{L2bilin6} that
\begin{displaymath}
I_2 \lesssim 2^{j_{min}/2} \int_{|\xi_1-\xi_2| < R}F_1(\xi_1)F_2(\xi_2)F_3(\xi_1+\xi_2)d\xi_1d\xi_2.
\end{displaymath}
Then, we obtain by letting $(\xi_1',\xi_2')=(\xi_1-\xi_2,\xi_2)$ and applying twice the Cauchy-Schwarz inequality that 
\begin{equation} \label{L2bilin18} 
\begin{split}
I_2 & \lesssim 2^{j_{min}/2}\int_{|\xi_1'| < R}F_1(\xi_1'+\xi_2')F_2(\xi_2')F_3(\xi_1'+2\xi_2')d\xi_1'd\xi_2' \\ 
& \lesssim 2^{j_{min}/2} R^{1/2}\|f_1\|_{L^2}\|f_2\|_{L^2}\|f_3\|_{L^2}.
\end{split}
\end{equation}

Next, we observe that in the region $\mathcal{R}_3$, 
\begin{displaymath} 
\big|\xi_1^4-\xi_2^4\big|^{\frac12}=\Big(|\xi_1-\xi_2| \cdot \big|\xi_1^3+\xi_1^2\xi_2+\xi_1\xi_2^2+\xi_2^3 \big| \Big)^{\frac12} \ge cR^{1/2}2^{3k_{max}/2} \ge 2,
\end{displaymath}
since $R$ will be chosen large enough.  Thus, the Cauchy-Schwarz inequality implies that 
\begin{equation} \label{L2bilin19} 
\begin{split}
&I_3 \lesssim R^{-1/2}2^{-3k_{max}/2}\|f_3\|_{L^2} \\ & \times \Big\|\int_{\scriptsize{\begin{array}{l}\xi_1+\xi_2=\xi \\ \theta_1+\theta_2=\theta \end{array}}}\big|\xi_1^4-\xi_2^4\big|_1^{\frac12}f_1(\xi_1,\theta_1+w(\xi_1))f_2(\xi_2,\theta_2+w(\xi_2))d\xi_1d\theta_1\Big\|_{L^2_{\xi,\theta}} \! ,
\end{split}
\end{equation}
where the definition of $|\cdot|_1$ is given just before Lemma \ref{bilinStrichartz}. By Plancherel's identity, the $L^2$-norm of the integral on the right-hand side of \eqref{L2bilin19} is equal to 
\begin{displaymath} 
\Big\|\int_{\theta_1,\theta_2}e^{-it(\theta_1+\theta_2)}\int_{\xi_1+\xi_2=\xi }\big|\xi_1^4-\xi_2^4\big|_1^{\frac12}f_1(\xi_1,\theta_1+w(\xi_1))f_2(\xi_2,\theta_2+w(\xi_2))d\xi_1d\theta_1d\theta_2\Big\|_{L^2_{\xi,t}} \!.
\end{displaymath}
This implies after changing variables $\tau_i=\theta_i+w(\xi_i)$ for $i=1,2$ and using  Minkowski's inequality that
\begin{displaymath} 
\begin{split}
&I_3  \lesssim R^{-1/2}2^{-3k_{max}/2}\|f_3\|_{L^2} \\ & \times \int_{\tau_1, \tau_2} \eta_{\le j_1}(\tau_1)\eta_{\le j_2}(\tau_2)\big\|\mathcal{I}^2\big(e^{t\partial_x^5}\mathcal{F}^{-1}_{\xi}(f_1(\cdot,\tau_1)),e^{t\partial_x^5}\mathcal{F}^{-1}_{\xi}(f_2(\cdot,\tau_2))\big)\big\|_{L^2_{x,t}} d\tau_1d\tau_2 ,
\end{split}
\end{displaymath}
where the bilinear operator $\mathcal{I}^2$ is defined in Lemma \ref{bilinStrichartz}. Therefore, we deduce from estimate \eqref{bilinStrichartz1} and the Cauchy-Schwarz inequality that 
\begin{equation} \label{L2bilin20}
I_3  \lesssim R^{-1/2}2^{-3k_{max}/2}2^{j_{min}/2}2^{j_{med}/2}\|f_1\|_{L^2}\|f_2\|_{L^2}\|f_3\|_{L^2}.
\end{equation} 

Finally, we conclude estimate \eqref{L2bilin3b} gathering estimates \eqref{L2bilin13}, \eqref{L2bilin15}, \eqref{L2bilin18}, \eqref{L2bilin20} and choosing $R=2^{-3k_{max}/2}2^{j_{med}/2}$.

This finishes the proof of Lemma \ref{L2bilin}.
\end{proof}

As a consequence of Lemma \ref{L2bilin}, we have the following $L^2$ bilinear estimates.
\begin{corollary} \label{coroL2bilin} 
Assume that $k_1, \ k_2, \ k_3 \in \mathbb Z_+$, $j_1, \ j_2, \ j_3 \in \mathbb Z_+$ and $f_i:\mathbb R^2 \to \mathbb R_+ $ are $L^2$ functions supported in $D_{k_i,j_i}$ for $i=1,2$. 
\begin{itemize}
\item[(a)]
Then it follows that 
\begin{equation} \label{coroL2bilin1}
\big \| \textbf{1}_{D_{k_3,j_3}}\cdot \big(f_1 \ast f_2\big) \big \|_{L^2} \lesssim 2^{k_{min}/2}2^{j_{min}/2}\|f_1\|_{L^2}\|f_2\|_{L^2}.
\end{equation}
\item[(b)]
Let us suppose that $k_{min} \le k_{max}-5$. If we are in the case where $(k_i,j_i)=(k_{min},j_{max})$ for some $i \in \{1,2,3\}$, then it holds that
\begin{equation} \label{coroL2bilin2b}
\big \| \textbf{1}_{D_{k_3,j_3}}\cdot \big(f_1 \ast f_2\big) \big \|_{L^2} \lesssim 2^{(j_1+j_2+j_3)/2}2^{-j_{med}/2}2^{-2k_{max}}\|f_1\|_{L^2}\|f_2\|_{L^2}.
\end{equation}
If moreover $k_{min} \ge 1$, then
\begin{equation} \label{coroL2bilin2}
\big \| \textbf{1}_{D_{k_3,j_3}}\cdot \big(f_1 \ast f_2\big) \big \|_{L^2} \lesssim 2^{(j_1+j_2+j_3)/2}2^{-j_{max}/2}2^{-(3k_{max}+k_{min})/2}\|f_1\|_{L^2}\|f_2\|_{L^2}.
\end{equation}

In all the others cases, we have that 
\begin{equation} \label{coroL2bilin3}
\big \| \textbf{1}_{D_{k_3,j_3}}\cdot \big(f_1 \ast f_2\big) \big \|_{L^2} \lesssim 2^{(j_1+j_2+j_3)/2}2^{-j_{max}/2}2^{-2k_{max}}\|f_1\|_{L^2}\|f_2\|_{L^2}.
\end{equation}
\end{itemize}
\end{corollary}

\begin{proof} Corolloray \ref{coroL2bilin} follows directly from Lemma \ref{L2bilin} by using a duality argument. 
\end{proof}

\subsection{$L^2$ trilinear estimates}
Now, we prove the $L^2$ trilinear estimates.  In this case, the resonance function $\widetilde{\Omega}:=\widetilde{\Omega}(\xi_1,\xi_2,\xi_3)$ is given by
\begin{equation} \label{resonancetilde} 
\widetilde{\Omega}(\xi_1,\xi_2,\xi_3):=w(\xi_1)+w(\xi_2)+w(\xi_3)-w(\xi_1+\xi_2+\xi_3).
\end{equation}

\begin{lemma} \label{L2trilin}
Assume that $k_1, \ k_2, \ k_3, \ k_4 \in \mathbb Z_+$, $j_1, \ j_2, \ j_3, \ j_4 \in \mathbb Z_+$ and $f_i:\mathbb R^2 \to \mathbb R_+ $ are $L^2$ functions supported in $D_{k_i,j_i}$ for $i=1,2,3,4$. 
\begin{itemize}
\item[(a)]
Then it follows that 
\begin{equation} \label{L2trilin1}
\int_{\mathbb R^2}\big(f_1\ast f_2 \ast f_3\big)\cdot f_4 \lesssim 2^{(k_{min}+k_{thd})/2}2^{(j_{min}+j_{thd})/2}\|f_1\|_{L^2}\|f_2\|_{L^2}\|f_3\|_{L^2}\|f_4\|_{L^2}.
\end{equation}
\item[(b)]
Moreover, let us suppose that $k_{thd} \le k_{max}-5$. If we are in the case where $(k_i,j_i)=(k_{thd},j_{max})$ for some $i \in \{1,2,3,4\}$, then it holds that
\begin{equation} \label{L2trilin2}
\begin{split}
&\int_{\mathbb R^2}\big(f_1\ast f_2 \ast f_3\big)\cdot f_4 \\ & \quad \lesssim 2^{(j_1+j_2+j_3+j_4)/2}2^{-j_{max}/2}2^{k_{thd}/2}2^{-2k_{max}}\|f_1\|_{L^2}\|f_2\|_{L^2}\|f_3\|_{L^2}\|f_4\|_{L^2},
\end{split}
\end{equation}
and 
\begin{equation} \label{L2trilin3}
\begin{split}
&  \int_{\mathbb R^2}\big(f_1\ast f_2 \ast f_3\big)\cdot f_4 \\ & \quad \lesssim 2^{(j_1+j_2+j_3+j_4)/2}2^{-j_{med}/2}2^{k_{min}/2}2^{-2k_{max}}\|f_1\|_{L^2}\|f_2\|_{L^2}\|f_3\|_{L^2}\|f_4\|_{L^2}.
\end{split}\end{equation}
In all the others cases, we have that 
\begin{equation} \label{L2trilin4}
\begin{split}
&\int_{\mathbb R^2}\big(f_1\ast f_2 \ast f_3\big)\cdot f_4 \\ & \quad \lesssim 2^{(j_1+j_2+j_3+j_4)/2}2^{-j_{max}/2}2^{k_{min}/2}2^{-2k_{max}}\|f_1\|_{L^2}\|f_2\|_{L^2}\|f_3\|_{L^2}\|f_4\|_{L^2}.
\end{split}
\end{equation}
\end{itemize}
\end{lemma}

\begin{proof} Estimate \eqref{L2trilin1} can be proved exactly as estimate \eqref{L2bilin1}. To prove part (b), we follow closely the arguments of Guo for the mBO equation \cite{Guo1}. Let us define 
\begin{equation} \label{L2trilin5}
\widetilde{I}:=\int_{\mathbb R^2}\big(f_1 \ast f_2 \ast f_3\big) \cdot f_4.
\end{equation}
Observe that 
\begin{equation} \label{L2trilin6} 
\widetilde{I} =\int_{\mathbb R^2} \big(\widetilde{f}_1 \ast \widetilde{f}_2 \ast f_4\big) \cdot f_3 = \int_{\mathbb R^2} \big(\widetilde{f}_2 \ast \widetilde{f}_2 \ast f_4\big) \cdot f_1= \int_{\mathbb R^2} \big(\widetilde{f}_1 \ast \widetilde{f}_3 \ast f_4\big) \cdot f_2,
\end{equation}
where $\widetilde{f}_i(\xi,\tau)=f_i(-\xi,-\tau)$. Therefore, we can always assume that $j_{max}=j_4$.  Moreover, we introduce 
$f^{\sharp}_i(\xi,\theta)=f_i(\xi,\theta+w(\xi))$, for $i=1,2,3$. In view of the assumptions on $f_i$, the functions $f^{\sharp}_i$ are supported in the sets 
\begin{displaymath} 
D_{k_i,j_i}^{\sharp}=\big\{(\xi,\theta) : \xi \in I_{k_i} \ \text{and} \ |\theta| \le 2^{j_i} \big\}.
\end{displaymath}
We also note that $\|f_i\|_{L^2}=\|f_i^{\sharp}\|_{L^2}$. Then, it follows changing variables that 
\begin{equation} \label{L2trilin7}
\widetilde{I}=
\int_{\mathbb R^6}f^{\sharp}_1(\xi_1,\theta_1)f^{\sharp}_2(\xi_2,\theta_2)f^{\sharp}_3(\xi_3,\theta_3)
f^{\sharp}_4(\xi_1+\xi_2+\xi_3,\theta_1+\theta_2+\theta_3+\widetilde{\Omega}(\xi_1,\xi_2,\xi_3))d\nu,
\end{equation}
where $d\nu=d\xi_1d\xi_2d\xi_3d\theta_1d\theta_2d\theta_3$ and $\tilde{\Omega}(\xi_1,\xi_2)$ is defined in \eqref{resonancetilde}. 

Since $k_{thd} \le k_{max}-5$ by hypothesis, we always have that $k_{max} \sim k_{sub}$. Thus, we only need to treat the following cases: $k_4 \sim k_{max}$, $k_4=k_{thd}$ and $k_4=k_{min}$.  

\noindent \textit{Case $k_4 \sim k_{max}$.} By symmetry, we can assume that $k_1 \le k_2 \le k_3 \le k_4$ in this case. 
For $g_i:\mathbb R \to \mathbb R_+$, $L^2$ functions supported in $I_{k_i}$ for $i=1,2,3$ and $g: \mathbb R^2 \to \mathbb R_+$, an $L^2$ function supported in $I_{k_4} \times [-2^{j_4},2^{j_4}]$, let us define
\begin{equation} \label{L2trilin8} 
\widetilde{J}(g_1,g_2,g_3,g) \!:=\int_{\mathbb R^3}g_1(\xi_1)g_2(\xi_2)g_3(\xi_3)g(\xi_1+\xi_2+\xi_3,\widetilde{\Omega}(\xi_1,\xi_2,\xi_3))d\xi_1d\xi_2d\xi_3.
\end{equation} 
Then, arguing as in \eqref{L2trilin9}, it suffices to show that
\begin{equation} \label{L2trilin9}
\widetilde{J}(g_1,g_2,g_3,g) \lesssim  2^{-2k_{max}}2^{k_{min}/2}\|g_1\|_{L^2}\|g_2\|_{L^2}\|g_3\|_{L^2}\|g\|_{L^2}.
\end{equation} 
in order to prove \eqref{L2trilin4} in this case. To prove estimate \eqref{L2trilin9}, we change variables $(\xi_1',\xi_2',\xi_3')=(\xi_1,\xi_2,\xi_1+\xi_2+\xi_3)$ and apply twice the Cauchy-Schwarz inequality in the $\xi_1'$ and $\xi_2'$ to deduce that 
\begin{equation} \label{L2trilin10} 
\widetilde{J}(g_1 ,g_2,g_3,g)  \lesssim  \|g_2\|_{L^2}\|g_3\|_{L^2} 
\int_{|\xi_1'| \sim 2^{k_1}} g_1(\xi_1') \Big( \int_{\mathbb R^2} g(\xi_3',\widetilde{\Omega})^2d\xi_2'd\xi_3 \Big)^{\frac12}d\xi_1.
\end{equation}
We observe that 
\begin{displaymath} 
\Big| \frac{\partial}{\partial \xi_2'}\widetilde{\Omega}(\xi_1',\xi_2',\xi_3'-\xi_2'-\xi_1') \Big|=5\big|(\xi_2')^4-(\xi_3'-\xi_2'+\xi_1')^4 \big| \sim 2^{4k_{max}},
\end{displaymath}
by using the frequency localization. Thus estimate \eqref{L2trilin9} is deduced by performing the change of variables $(\mu_2,\mu_3)=(\widetilde{\Omega},\xi_3')$ in the inner integral on the right-hand side of \eqref{L2trilin10} and by applying the Cauchy-Schwarz inequality in the variable $\xi_1'$.

\noindent \textit{Case $k_4=k_{min}$.} In this case, we can assume without loss of generality that $k_4 \le k_1 \le k_2 \le k_3$. It suffices to show that estimate \eqref{L2trilin9} remains valid in this case. First, we change variables $(\xi_1',\xi_2',\xi_3')=(\xi_1,\xi_2,\xi_1+\xi_2+\xi_3)$, so that $|\xi_1'| \sim 2^{k_{thd}}$, $|\xi_2'| \sim 2^{k_{max}}$, $|\xi_3'| \sim 2^{k_{min}}$ and $\widetilde{J}$ becomes 
\begin{displaymath}  
\begin{split}
\widetilde{J}&(g_1,g_2,g_3,g) \\ & =\int_{\mathbb R^3}g_1(\xi'_1)g_2(\xi'_2)g_3(\xi'_3-\xi'_1-\xi'_2)g(\xi_3',\widetilde{\Omega}(\xi'_1,\xi'_2,\xi'_3-\xi'_1-\xi'_2))d\xi'_1d\xi'_2d\xi_3'.
\end{split}
\end{displaymath}  
Thus the Cauchy-Schwarz inequality in $\xi'_1$ implies that 
\begin{equation} \label{L2trilin11} 
\begin{split}
&\widetilde{J}(g_1,g_2,g_3,g)  \\ &\le \int_{\mathbb R^2} g_2(\xi_2') \|g_1(\xi_1')g_3(\xi_3'-\xi_2'-\xi_1')\|_{L^2_{\xi_1'}}
\big\|g(\xi_3',\widetilde{\Omega}(\xi'_1,\xi'_2,\xi'_3-\xi'_1-\xi'_2)) \big\|_{L^2_{\xi_1'}}d\xi'_2d\xi_3'.
\end{split}
\end{equation}
Moreover, we have that 
\begin{displaymath}
\Big| \frac{\partial}{\partial \xi_1'}\widetilde{\Omega}(\xi_1',\xi_2',\xi_3'-\xi_2'-\xi_1') \Big|=5\big|(\xi'_1)^4-(\xi'_3-\xi'_1-\xi'_2)^4 \big| \sim 2^{4k_{max}},
\end{displaymath}
due to the frequency localization, so that we deduce through the change of variable $\mu_1'=\widetilde{\Omega}$ that 
\begin{equation} \label{L2trilin12}
\big\|g(\xi_3',\widetilde{\Omega}(\xi'_1,\xi'_2,\xi'_3-\xi'_1-\xi'_2)) \big\|_{L^2_{\xi_1'}} = c2^{-2k_{max}} \|g(\xi_3',\cdot)\|_{L^2}.
\end{equation}
Therefore, we deduce inserting \eqref{L2trilin12} in \eqref{L2trilin11} and applying twice the Cauchy-Schwarz inequality that 
\begin{displaymath} 
\begin{split}
\widetilde{J}(g_1,g_2,g_3,g) & \lesssim  2^{-2k_{max}}\|g_1\|_{L^2}\|g_2\|_{L^2}\|g_3\|_{L^2}
\int_{|\xi'_3| \sim 2^{k_{min}}}\|g(\xi_3',\cdot)\|_{L^2}d\xi'_3 \\
& \lesssim 2^{-2k_{max}}2^{k_{min}/2}\|g_1\|_{L^2}\|g_2\|_{L^2}\|g_3\|_{L^2}\|g\|_{L^2},
\end{split}
\end{displaymath}
which is exactly \eqref{L2trilin9}. 

\noindent \textit{Case $k_4=k_{thd}$.}
Estimate \eqref{L2trilin2} follows arguing exactly as in the case $k_4=k_{min}$. On the other hand, estimate \eqref{L2trilin3} can also be proved applying the arguments of the cases $k_4 \sim k_{max}$ or $k_4 = k_{min}$, depending on wether $j_{med}=j_1$, $j_2$ or $j_3$ and using the symmetry relation \eqref{L2trilin6}.
\end{proof} 

As a consequence of Lemma \ref{L2trilin}, we have the following $L^2$ trilinear estimates.
\begin{corollary} \label{coroL2trilin} 
Assume that $k_1, \ k_2, \ k_3, \ k_4\in \mathbb Z_+$, $j_1, \ j_2, \ j_3, \ j_4 \in \mathbb Z_+$ and $f_i:\mathbb R^2 \to \mathbb R_+ $ are $L^2$ functions supported in $D_{k_i,j_i}$ for $i=1,2,3$. 
\begin{itemize}
\item[(a)]
Then it follows that 
\begin{equation} \label{coroL2trilin1}
\big\| \textbf{1}_{D_{k_4,j_4}}\cdot \big(f_1 \ast f_2 \ast f_3 \big) \big\|_{L^2} \lesssim 2^{(k_{min}+k_{thd})/2}2^{(j_{min}+j_{thd})/2}\|f_1\|_{L^2}\|f_2\|_{L^2}\|f_3\|_{L^2}.
\end{equation}
\item[(b)]
Let us suppose that $k_{thd} \le k_{max}-5$. If we are in the case where $(k_i,j_i)=(k_{thd},j_{max})$ for some $i \in \{1,2,3,4\}$, then it holds that
\begin{equation} \label{coroL2trilin2} 
\begin{split}
\big\| \textbf{1}_{D_{k_4,j_4}}&\cdot \big(f_1 \ast f_2 \ast f_3\big) \big\|_{L^2} \\ & \lesssim 2^{(j_1+j_2+j_3+j_4)/2}2^{-j_{max}/2}2^{k_{thd}/2}2^{-2k_{max}}\|f_1\|_{L^2}\|f_2\|_{L^2}\|f_3\|_{L^2}.
\end{split}
\end{equation}
and
\begin{equation} \label{coroL2trilin3}
\begin{split}
\big\| \textbf{1}_{D_{k_4,j_4}}&\cdot \big(f_1 \ast f_2 \ast f_3\big) \big\|_{L^2} \\ & \lesssim 2^{(j_1+j_2+j_3+j_4)/2}2^{-j_{med}/2}2^{k_{min}/2}2^{-2k_{max}}\|f_1\|_{L^2}\|f_2\|_{L^2}\|f_3\|_{L^2}.
\end{split}
\end{equation}

In all the others cases, we have that 
\begin{equation} \label{coroL2trilin4}
\begin{split}
\big\| \textbf{1}_{D_{k_4,j_4}}&\cdot \big(f_1 \ast f_2 \ast f_3\big) \big \|_{L^2} \\ & \lesssim 2^{(j_1+j_2+j_3+j_4)/2}2^{-j_{max}/2}2^{k_{min}/2}2^{-2k_{max}}\|f_1\|_{L^2}\|f_2\|_{L^2}\|f_3\|_{L^2}.
\end{split}
\end{equation}
\end{itemize}
\end{corollary}

\begin{proof} Corolloray \ref{coroL2trilin} follows directly from Lemma \ref{L2trilin} by using a duality argument. 
\end{proof}

\section{Short time bilinear estimates} \label{STBL}
The main results of this section are the following bilinear estimates in the $F^s_{\alpha}(T)$ spaces. Note that to overcome the high-low frequency interaction problem (c.f. \cite{Pi}), we need to work with $\alpha=2$ (see Lemma \ref{high-low} below). Therefore, we will fix $\alpha=2$ in the rest of the paper and denote respectively $F^s_2(T), \ N^s_2(T), \ F^s_2, \ N^s_2, \ F_{k,2}$ and $N_{k,2}$ by $F^s(T), \ N^s(T), \ F^s, \ N^s, \ F_k$ and $N_k$. The main results of this section are the bilinear estimates at the $H^s$ and $L^2$ level.
\begin{proposition} \label{bilinE} 
Let $s>1$ and $T \in (0,1]$ be given. Then, it holds that 
\begin{equation} \label{bilinE1} 
\|\partial_x(u\partial_x^2v)\|_{N^s(T)} \lesssim \|u\|_{F^s(T)}\|v\|_{F^1(T)}+\|u\|_{F^1(T)}\|v\|_{F^s(T)},
\end{equation}
and 
\begin{equation} \label{bilinE2}
\|\partial_xu\partial_x^2v\|_{N^s(T)} \lesssim \|u\|_{F^s(T)}\|v\|_{F^1(T)}+\|u\|_{F^1(T)}\|v\|_{F^s(T)},
\end{equation}
for all $u, \ v \in F^s(T)$.
\end{proposition}

\begin{proposition} \label{bilinL2E} 
Let T $\in (0,1]$ be given. Then, it holds that 
\begin{equation} \label{bilinL2E1} 
\|\partial_x(u\partial_x^2v)\|_{N^0(T)}+\|\partial_x(v\partial_x^2u)\|_{N^0(T)} \lesssim \|u\|_{F^2(T)}\|v\|_{F^0(T)},
\end{equation}
and 
\begin{equation} \label{bilinL2E2}
\|\partial_x\big(\partial_xu\partial_xv\big)\|_{N^0(T)} \lesssim \|u\|_{F^2(T)}\|v\|_{F^0(T)},
\end{equation}
for all $u \in F^2(T)$ and $v \in F^0(T)$.
\end{proposition}

We split the proof of Propositions \ref{bilinE} and \ref{bilinL2E} in several technical lemmas. 
\begin{lemma} \label{high-low}  [high $\times$ low $\rightarrow$ high] 
Assume that $k, \ k_1, \ k_2 \in \mathbb Z_+$ satisfy $|k-k_2| \le 3$ and $0 \le k_1 \le \max(k,k_2)-5$. Then, 
\begin{equation} \label{high-low1} 
\big\| P_k\partial_x\big(u_{k_1}\partial_x^2v_{k_2} \big)\big\|_{N_k} \lesssim \|u_{k_1}\|_{F_{k_1}}\|v_{k_2}\|_{F_{k_2}},
\end{equation}
and 
\begin{equation} \label{high-low2} 
\big\| P_k\big(\partial_xu_{k_1}\partial_x^2v_{k_2} \big)\big\|_{N_k} \lesssim \|u_{k_1}\|_{F_{k_1}}\|v_{k_2}\|_{F_{k_2}},
\end{equation}
for all $u_{k_1} \in F_{k_1}$ and $v_{k_2} \in F_{k_2}$.
\end{lemma} 

\begin{remark} In the case $k_1=0$, the function $u_0 \in F_0$ is localized in spatial low frequencies corresponding to the projection $P_{\le 0}$, since we choose to use a nonhomogeneous dyadic partition of the unity to define the function spaces $F^s$ and $N^s$ (see Section 2).
\end{remark}

\begin{remark} \label{remahigh-low} Lemma \ref{high-low} still holds true under the assumptions $k, \ k_1, \ k_2 \in \mathbb Z_+$, $|k-k_1| \le 3$ and $0 \le k_2 \le \max(k,k_1)-5$. The proof is exactly the same, therefore we will omit it.
\end{remark}

\begin{proof}[Proof of Lemma \ref{high-low}] We only prove estimate \eqref{high-low1}, since the proof of estimate \eqref{high-low2} is similar (and even easier). First, observe from the definition of $N_k$ in \eqref{Nk} that
\begin{equation} \label{high-low3}
\big\| P_k\partial_x\big(u_{k_1}\partial_x^2v_{k_2} \big)\big\|_{N_k}  \lesssim  
 \sup_{t_k \in \mathbb R} \big\|\big(\tau-w(\xi)+i2^{2k} \big)^{-1}2^{3k}\textbf{1}_{I_k}f_{k_1}\ast f_{k_2}\big\|_{X_k}, 
 \end{equation} 
 where  
 \begin{displaymath}
 f_{k_1}=\big|\mathcal{F}\big(\eta_0(2^{2 k}(\cdot-t_k))u_{k_1}\big)\big| \quad \text{and} \quad f_{k_2}=\big|\mathcal{F}\big(\widetilde{\eta}_0(2^{2 k}(\cdot-t_k))v_{k_2}\big)\big| .
\end{displaymath}
Now, we set 
\begin{displaymath} 
f_{k_i,2k}(\xi,\tau)=\eta_{\le 2k}(\tau-w(\xi))f_{k_i}(\xi,\tau)\ \text{and} \ f_{k_i,j_i}(\xi,\tau)=\eta_j(\tau-w(\xi))f_{k_i}(\xi,\tau),
\end{displaymath}
for $j_i>2k$. Thus, we deduce from \eqref{high-low3} and the definition of $X_k$ that 
\begin{equation} \label{high-low4} 
\big\| P_k\partial_x\big(u_{k_1}\partial_x^2v_{k_2} \big)\big\|_{N_k} \lesssim \sup_{t_k \in \mathbb R}2^{3k}\sum_{j, j_1, j_2 \ge 2k}2^{-j/2}\|\textbf{1}_{D_{k,j}}\cdot f_{k_1,j_1}\ast f_{k_2,j_2} \|_{L^2_{\xi,\tau}},
\end{equation}
where $D_{k,j}$ is defined in \eqref{D}. Here, we use that since $\big| \big(\tau-w(\xi)+i2^{2k} \big)^{-1}\big| \le 2^{-2k}$ the sum from $j=0$ to $2k-1$ appearing implicitely on the right-hand side of \eqref{high-low3} is controlled by the term corresponding to $j=2k$ on right-hand side of \eqref{high-low4}. 
Therefore, according to Corollary \ref{coro2} and estimate \eqref{high-low4} it suffices to prove that 
\begin{equation} \label{high-low5} 
2^{3k}\sum_{j \ge 2k}2^{-j/2}\big\|\textbf{1}_{D_{k,j}}\cdot \big(f_{k_1,j_1}\ast f_{k_2,j_2}\big) \big\|_{L^2_{\xi,\tau}} \lesssim 2^{j_1/2}\|f_{k_1,j_1}\|_{L^2}2^{j_2/2}\|f_{k_2,j_2}\|_{L^2},
\end{equation}
with $j_1, \ j_2 \ge 2k$, in order to prove estimate \eqref{high-low1}.

But, we deduce from estimates \eqref{coroL2bilin2b} and \eqref{coroL2bilin3} that 
\begin{displaymath} 
\begin{split}
2^{3k}\sum_{j \ge 2k}2^{-j/2}\big\|\textbf{1}_{D_{k,j}}&\cdot \big(f_{k_1,j_1}\ast f_{k_2,j_2}\big) \big\|_{L^2_{\xi,\tau}} 
\\ & \lesssim 2^{3k}\sum_{j \ge 2k}2^{-j/2}2^{-2k}2^{j_1/2}\|f_{k_1,j_1}\|_{L^2}2^{j_2/2}\|f_{k_2,j_2}\|_{L^2},
\end{split}
\end{displaymath}
which implies estimate \eqref{high-low5} after summing over $j$. This finishes the proof of Lemma \ref{high-low}.
\end{proof}

\begin{lemma} \label{high-high}  [high $\times$ high $\rightarrow$ high] 
Assume that $k, \ k_1, \ k_2 \in \mathbb Z_+$ satisfy $k \ge 20$, $|k-k_2| \le 5$ and $| k_1- k_2| \le 5$. Then, 
\begin{equation} \label{high-high1} 
\big\| P_k\partial_x\big(u_{k_1}\partial_x^2v_{k_2} \big)\big\|_{N_k} \lesssim \|u_{k_1}\|_{F_{k_1}}\|v_{k_2}\|_{F_{k_2}},
\end{equation}
and 
\begin{equation} \label{high-high2} 
\big\| P_k\big(\partial_xu_{k_1}\partial_x^2v_{k_2} \big)\big\|_{N_k} \lesssim \|u_{k_1}\|_{F_{k_1}}\|v_{k_2}\|_{F_{k_2}},
\end{equation}
for all $u_{k_1} \in F_{k_1}$ and $v_{k_2} \in F_{k_2}$.
\end{lemma} 

\begin{proof} Once again we only prove estimate \eqref{high-high1}. Arguing as in the proof of Lemma \ref{high-low}, it is enough to prove that 
\begin{equation} \label{high-high3} 
2^{3k}\sum_{j \ge 2k}2^{-j/2}\big\|\textbf{1}_{D_{k,j}}\cdot \big(f_{k_1,j_1}\ast f_{k_2,j_2}\big) \big\|_{L^2_{\xi,\tau}} \lesssim 2^{j_1/2}\|f_{k_1,j_1}\|_{L^2}2^{j_2/2}\|f_{k_2,j_2}\|_{L^2},
\end{equation}
where  $f_{k_i,j_i}$ is localized in $D_{k_i,j_i}$ with $j_i \ge 2k$ for $i=1, \ 2$.

We deduce by applying estimate \eqref{coroL2bilin1} to the left-hand side of \eqref{high-high3} that 
\begin{equation} \label{high-high4}
\begin{split}
2^{3k}\sum_{j \ge 2k}2^{-j/2}\big\|\textbf{1}_{D_{k,j}}&\cdot \big(f_{k_1,j_1}\ast f_{k_2,j_2}\big) \big\|_{L^2_{\xi,\tau}} \\ &  \lesssim 
2^{3k}\sum_{j \ge 2k}2^{-j/2}2^{k/2}2^{j_{min}/2}\|f_{k_1,j_1}\|_{L^2}\|f_{k_2,j_2}\|_{L^2}.
\end{split}
\end{equation}
According to Lemma \ref{lresonance} and the frequency localization, we have that 
\begin{equation} \label{high-high5}
2^{j_{max}} \sim \max\{2^{j_{med}},2^{5k}\}.
\end{equation}
Finally, we observe that \eqref{high-high4} and \eqref{high-high5} imply estimate \eqref{high-high3}. This is clear in 
the cases where $j_{max}=j_1$ or $j_2$ by using that $2^{j_{max}} \gtrsim 2^{5k}$ and summing over $j \ge 2k$. In the case where $j_{max}=j$, we have from \eqref{high-high5} that either $2^{j} \sim 2^{5k}$ or $2^{j} \sim 2^{j_{med}}$. When $2^{j} \sim 2^{5k}$, estimate \eqref{high-high3} follows directly from \eqref{high-high4} since we do not need to sum over $j$, whereas when $2^{j} \sim 2^{j_{med}}$, we can use one of the cases $2^{j_{max}}=2^{j_1}$ or $2^{j_{max}}=2^{j_2}$ to conclude.
\end{proof}

\begin{lemma} \label{high-high-low}  [high $\times$ high $\rightarrow$ low] 
Assume that $k, \ k_1, \ k_2 \in \mathbb Z_+$ satisfy $k_2 \ge 20$, $| k_1- k_2| \le 3$ and $0 \le k \le \max(k_1,k_2)-5$. Then, 
\begin{equation} \label{high-high-low1} 
\big\| P_k\partial_x\big(u_{k_1}\partial_x^2v_{k_2} \big)\big\|_{N_k} \lesssim k_22^{k_2-k}\|u_{k_1}\|_{F_{k_1}}\|v_{k_2}\|_{F_{k_2}},
\end{equation}
\begin{equation} \label{high-high-low1b} 
\big\| P_k\partial_x\big(\partial_xu_{k_1}\partial_xv_{k_2} \big)\big\|_{N_k} \lesssim k_22^{k_2-k}\|u_{k_1}\|_{F_{k_1}}\|v_{k_2}\|_{F_{k_2}},
\end{equation}
and 
\begin{equation} \label{high-high-low2} 
\big\| P_k\big(\partial_xu_{k_1}\partial_x^2v_{k_2} \big)\big\|_{N_k} \lesssim k_22^{2(k_2-k)}\|u_{k_1}\|_{F_{k_1}}\|v_{k_2}\|_{F_{k_2}},
\end{equation}
for all $u_{k_1} \in F_{k_1}$ and $v_{k_2} \in F_{k_2}$.
\end{lemma}

\begin{remark} It is interesting to observe that the restriction $s>1$ in Proposition \ref{bilinE} appears in estimate \eqref{high-high-low2}.
\end{remark}

\begin{remark} Note that in the case $k=0$, by convention $P_0=P_{\le 0}$. 
\end{remark}

\begin{proof} We prove estimate \eqref{high-high-low2}, since estimates \eqref{high-high-low1} and  \eqref{high-high-low1b} could be proved in a similar way. Let $\gamma:\mathbb R \rightarrow [0,1]$ be a smooth function supported in $[-1,1]$ with the property that 
$$\sum_{m \in \mathbb Z}\gamma^2(x-m)=1, \quad \forall \ x \in \mathbb R.$$
We observe from the definition of $N_k$ in \eqref{Nk} that
\begin{equation} \label{high-high-low3}
\begin{split}
\big\| P_k\partial_x\big(u_{k_1}&\partial_x^2v_{k_2} \big)\big\|_{N_k}  \\ &\lesssim  
 \sup_{t_k \in \mathbb R} \big\|\big(\tau-w(\xi)+i2^{2k} \big)^{-1}2^{3k_2}\textbf{1}_{I_k}\sum_{|m| \lesssim 2^{2(k_2-k)}}f_{k_1}^m\ast f_{k_2}^m\big\|_{X_k}, 
 \end{split}
 \end{equation} 
 where  
 \begin{displaymath}
 f_{k_1}^m=\big|\mathcal{F}\big(\eta_0(2^{2 k}(\cdot-t_k))\gamma(2^{2k_2}(\cdot-t_k)-m)u_{k_1}\big)\big|,
 \end{displaymath}
 and
  \begin{displaymath}
 f_{k_2}^m=\big|\mathcal{F}\big(\widetilde{\eta}_0(2^{2 k}(\cdot-t_k))\gamma(2^{2k_2}(\cdot-t_k)-m)u_{k_2}\big)\big|,
 \end{displaymath}
 for $i=1,2$.  
 
 Now, we set 
\begin{displaymath} 
f_{k_i,2k_2}^m(\xi,\tau)=\eta_{\le 2k_2}(\tau-w(\xi))f_{k_i}^m(\xi,\tau)\ \text{and} \ f_{k_i,j_i}^m=\eta_j(\tau-w(\xi))f_{k_i}^m(\xi,\tau),
\end{displaymath}
for $j_i>2k_2$. Thus, we deduce from \eqref{high-high-low3} and the definition of $X_k$ that 
\begin{equation} \label{high-high-low4} 
\begin{split}
\big\| P_k&\partial_x\big(u_{k_1}\partial_x^2v_{k_2} \big)\big\|_{N_k}\\ & \lesssim \sup_{t_k \in \mathbb R, \ m \in \mathbb Z}2^{5k_2}2^{-2k}\sum_{j \ge 0}\sum_{j_1, j_2 \ge 2k_2}2^{-j/2}\|\textbf{1}_{D_{k,j}}\cdot f_{k_1,j_1}^m\ast f_{k_2,j_2}^m \|_{L^2_{\xi,\tau}}.
\end{split}
\end{equation}
Therefore, according to Lemma \ref{lemma2} and estimate \eqref{high-high-low4} it suffices to prove that 
\begin{equation} \label{high-high-low5} 
2^{3k_2}\sum_{j \ge 0}2^{-j/2}\big\|\textbf{1}_{\widetilde{D}_{k,j}}\cdot \big(f_{k_1,j_1}^m\ast f_{k_2,j_2}^m \big)\big\|_{L^2_{\xi,\tau}} \lesssim k_22^{j_1/2}\|f_{k_1,j_1}^m\|_{L^2}2^{j_2/2}\|f_{k_2,j_2}^m\|_{L^2},
\end{equation}
with $j_1, \ j_2 \ge 2k_2$, in order to prove estimate \eqref{high-high-low2}.

In the cases  $j_{max}=j_1$ or $j_{max}=j_2$, say for example $j_{max}=j_1$, we deduce from estimate \eqref{coroL2bilin3}  that 
\begin{displaymath}
\begin{split}
&2^{3k_2}\sum_{j \ge 0}2^{-j/2}\big\|\textbf{1}_{\widetilde{D}_{k,j}}\cdot \big(f_{k_1,j_1}^m\ast f_{k_2,j_2}^m\big) \big\|_{L^2_{\xi,\tau}} \\ & 
  \lesssim 2^{k_2}\sum_{j \ge 0}2^{-j/2}2^{j/2}\|f_{k_1,j_1}^m\|_{L^2}2^{j_2/2}\|f_{k_2,j_2}^m\|_{L^2}  \\ & \lesssim 
  k_22^{j_1/2}\|f_{k_1,j_1}^m\|_{L^2}2^{j_2/2}\|f_{k_2,j_2}^m\|_{L^2} +2^{k_2}\sum_{j \ge 2k_2}2^{-j/2}2^{j_1/2}\|f_{k_1,j_1}^m\|_{L^2}2^{j_2/2}\|f_{k_2,j_2}^m\|_{L^2},
\end{split}
\end{displaymath}
which implies estimate \eqref{high-high-low4} by summing over $j$.

In the case $j_{max}=j$,  we have that 
$2^{j} \sim \max\{2^{j_{med}},|\Omega|\}$, where $\Omega$ is defined in \eqref{resonnance}. If $2^{j} \sim 2^{j_{med}}$, then we are in one of the above cases, whereas in the case $2^{j} \sim |\Omega|$, we deduce from \eqref{lresonance1} that 
$j \le 4k_2+5$. Therefore, we get from \eqref{coroL2bilin2b} that
\begin{displaymath}
\begin{split}
2^{3k_2}\sum_{j \ge 0}2^{-j/2}\big\|\textbf{1}_{\widetilde{D}_{k,j}}&\cdot \big(f_{k_1,j_1}^m\ast f_{k_2,j_2}^m\big) \big\|_{L^2_{\xi,\tau}} \\ & \lesssim 2^{k_2}k_22^{-j/2}2^{(j+j_1+j_2)/2}2^{-j_{med}/2}\|f_{k_1,j_1}^m\|_{L^2}\|f_{k_2,j_2}^m\|_{L^2},  
\end{split}
\end{displaymath}
which yields \eqref{high-high-low4}, since $j_{med} \ge 2k_2$.
\end{proof}

\begin{lemma} \label{low-low}  [low $\times$ low $\rightarrow$ low] 
Assume that $k, \ k_1, \ k_2 \in \mathbb Z_+$ satisfy $0 \le k, \ k_1, \ k_2 \le 100$. Then, 
\begin{equation} \label{low-low1} 
\big\| P_k\partial_x\big(u_{k_1}\partial_x^2v_{k_2} \big)\big\|_{N_k} \lesssim \|u_{k_1}\|_{F_{k_1}}\|v_{k_2}\|_{F_{k_2}},
\end{equation}
and 
\begin{equation} \label{low-low2} 
\big\| P_k\big(\partial_xu_{k_1}\partial_x^2v_{k_2} \big)\big\|_{N_k} \lesssim \|u_{k_1}\|_{F_{k_1}}\|v_{k_2}\|_{F_{k_2}},
\end{equation}
for all $u_{k_1} \in F_{k_1}$ and $v_{k_2} \in F_{k_2}$.
\end{lemma} 

\begin{proof} 
Once again we only prove estimate \eqref{low-low1}. Arguing as in the proof of Lemma \ref{high-low}, it is enough to prove that 
\begin{equation} \label{low-low3} 
\sum_{j \ge 0}2^{-j/2}\big\|\textbf{1}_{D_{k,j}}\cdot \big(f_{k_1,j_1}\ast f_{k_2,j_2}\big) \big\|_{L^2_{\xi,\tau}} \lesssim 2^{j_1/2}\|f_{k_1,j_1}\|_{L^2}2^{j_2/2}\|f_{k_2,j_2}\|_{L^2},
\end{equation}
where  $f_{k_i,j_i}$ is localized in $D_{k_i,j_i}$ with $j_i \ge 0$ for $i=1, \ 2$, which is a direct consequence of estimate \eqref{coroL2bilin1}.
\end{proof}

Finally, we give the proof of Proposition \ref{bilinE}.  Note that the proof of Proposition \ref{bilinL2E} would be similar.
\begin{proof}[Proof of Proposition \ref{bilinE}] We only prove estimate \eqref{bilinE2}, since the proof of estimate \eqref{bilinE1} would be similar.  We choose two extensions $\tilde{u}$ and $\tilde{v}$ of $u$ and $v$ satisfying 
\begin{equation} \label{bilinE3} 
\|\tilde{u}\|_{F^s} \le 2\|u\|_{F^s(T)} \quad \text{and} \quad \|\tilde{v}\|_{F^s} \le 2\|v\|_{F^s(T)}.
\end{equation}
Therefore $\partial_x\tilde{u}\partial_x^2\tilde{v}$ is an extension of $\partial_xu\partial_x^2v$ on $\mathbb R^2$ and we have from the definition of $N^s(T)$ and Minkowski inequality that
\begin{displaymath} 
\|\partial_xu\partial_x^2v\|_{N^s(T)}   
\le \Big(\sum_{k \ge 0}2^{2ks}\Big(\sum_{k_1,  k_2 \ge 0}\|P_k(\partial_xP_{k_1}\tilde{u}\partial_x^2P_{k_2}\tilde{v})\|_{N_k}\Big)^2 \Big)^{\frac12}.
\end{displaymath}
where we took the convention $P_0=P_{\le 0}$. Moreover,  we denote 
\begin{displaymath} 
\begin{array}{l}
A_1=\big\{(k_1,k_2) \in \mathbb Z_+^2 \ : \  |k_2-k| \le 3 \ \text{and} \ 0 \le k_1\le \max(k,k_2)-5\big\} ,\\ 
A_2=\big\{(k_1,k_2) \in \mathbb Z_+^2 \ : \  |k_1-k| \le 3 \ \text{and} \ 0 \le k_2\le \max(k,k_1)-5\big\}, \\ 
A_3=\big\{(k_1,k_2) \in \mathbb Z_+^2 \ : \  |k_1-k_2| \le 5, \ |k_1-k| \le 5 \ \text{and} \ k \ge 20\big\}, \\
A_4=\big\{(k_1,k_2) \in \mathbb Z_+^2 \ : \  |k_1-k_2| \le 3, \ 0 \le k \le \max(k_1,k_2)-5 \ \text{and} \ k_2 \ge 20\big\},\\ 
A_5=\big\{(k_1,k_2) \in \mathbb Z_+^2 \ : \ 0 \le k, \ k_1, \ k_2 \le 100\big\}.
\end{array}
\end{displaymath}
Note that for a given $k \in \mathbb Z_+$, some of these regions may be empty and others may overlap, but due to the frequency localization, we always have that 
\begin{equation} \label{bilinE4}
\begin{split}
\|\partial_xu\partial_x^2v\|_{N^s(T)}&  \lesssim 
\sum_{j=0}^5\Big(\sum_{k \ge 0}2^{2ks}\Big( \sum_{(k_1,k_2) \in A_j}\|P_k(\partial_xP_{k_1}\tilde{u}\partial_x^2P_{k_2}\tilde{v})\|_{N_k}\Big)^2\Big)^{\frac12} \\
& =: \sum_{j=0}^5S_j.
\end{split}
\end{equation}

To handle the sum $S_1$, we use estimate \eqref{high-low2} and the Cauchy-Schwarz inequality to obtain that 
\begin{equation} \label{bilinE5}
S_1 \lesssim \Big( \sum_{k \ge 0}2^{2ks}\Big(\sum_{k_1=0}^{k-5}\|P_{k_1}\tilde{u}\|_{F_{k_1}}\|P_{k}\tilde{v}\|_{F_{k}}\Big)^2\Big)^{\frac12} \lesssim \|\tilde{u}\|_{F^{0+}}\|\tilde{v}\|_{F^{s}},
\end{equation}
where we assumed without loss of generality that $\max(k,k_2)=k$. 
Similarly, we deduce from remark \ref{remahigh-low} that 
\begin{equation} \label{bilinE6} 
S_2 \lesssim \|\tilde{u}\|_{F^{s}}\|\tilde{v}\|_{F^{0+}}.
\end{equation}
Estimate \eqref{high-high2} leads to 
\begin{equation} \label{bilinE7}
S_3 \lesssim \Big( \sum_{k \ge 0}2^{2ks}\|P_{k_1}\tilde{u}\|_{F_{k_1}}^2\|P_{k}\tilde{v}\|_{F_{k}}^2\Big)^{\frac12} \lesssim \|\tilde{u}\|_{F^{0}}\|\tilde{v}\|_{F^{s}}.
\end{equation}
Next, we deal with the sum $S_4$.  Without loss of generality, assume that $\max(k_1,k_2)=k_2$. It follows from estimate \eqref{high-high-low2} and the Cauchy-Scwarz inequality in $k_2$ that 
\begin{equation} \label{bilinE8}
\begin{split} 
S_4 &\lesssim \Big( \sum_{k=0}^{k_2-5}2^{2k(s-2)}\Big(\sum_{k_2 \ge 0}k_22^{2k_2}\|P_{k_2}\tilde{u}\|_{F_{k_2}}\|P_{k_2}\tilde{v}\|_{F_{k_2}}\Big)^2\Big)^{\frac12} \\ & 
\lesssim \Big(\sum_{k_2 \ge 0}2^{2k_2s}\|P_{k_2}\tilde{u}\|_{F_{k_2}}^2 \Big)^{\frac12}\Big(\sum_{k_2 \ge 0}2^{2k_2}\|P_{k_2}\tilde{v}\|_{F_{k_2}}^2 \Big)^{\frac12} \\ & 
\lesssim \|\tilde{u}\|_{F^s}\|\tilde{v}\|_{F^1},
\end{split}
\end{equation}
since $s>1$. Finally, it is clear from estimate \eqref{low-low2} that
\begin{equation} \label{bilinE9} 
S_5 \lesssim \|\tilde{u}\|_{F^0}\|\tilde{v}\|_{F^0}.
\end{equation}
Therefore, we conclude the proof of estimate \eqref{bilinE2} gathering \eqref{bilinE3}--\eqref{bilinE9}.
 \end{proof}

\section{Energy estimates} \label{EE} 
As indicated in the introduction we assume for sake of simplicity that $c_3=0$. We also recall that, due to the short time bilinear estimates derived in the last section, we need to work with $\alpha=2$ in the definition of the spaces $F^s_{\alpha},$ $F^s_{\alpha}(T)$ and $F_{k,\alpha}$ and therefore we will omit the index $\alpha=2$ to simplify the notations. 

\subsection{Energy estimates for a smooth solution} \label{EE1}
Due to the linear estimate \eqref{linear1}, we need to control the norm $\|\cdot\|_{B^s(T)}$ of a solution $u$ to \eqref{5KdV}
as a function of $\|u_0\|_{H^s}$ and $\|u\|_{F^s(T)}$.  However, we are not able to estimate $\|u\|_{B^s(T)}$ directly. We need to modify the energy by a cubic term to cancel some bad terms appearing after a commutator estimate (see Remark \ref{rematecEE} below). 

Let us define $\psi(\xi):=\xi\eta'(\xi)$, where $\eta$ is defined in \eqref{eta} and $'$ denote the derivative, \textit{i.e.} $\eta'(\xi)=\frac{d}{d\xi}\eta(\xi)$. Then, for $k \ge 1$, we define $\psi_k(\xi)=\psi(2^{-k}\xi)$. We also denote by $Q_k$ the Littlewood-Paley multiplier of symbol $\psi_k$, \textit{i.e.} $Q_ku=\mathcal{F}_x\big(\psi_k\mathcal{F}_xu\big)$. From the definition of $\eta_k$ in \eqref{eta}, we observe that 
\begin{equation} \label{psik}
\psi_k(\xi)=2^{-k}\xi\eta'(2^{-k}\xi)=\xi \frac{d}{d\xi}\big(\eta(2^{-k}\xi) \big)=\xi \eta_k'(\xi).
\end{equation}
Finally, we define the new energy  by 
\begin{equation} \label{Ek} 
\begin{split}
E_k(u)(t)=&\|P_ku(\cdot,t)\|_{L^2}^2+\alpha\int_{\mathbb R}\big(uP_k\partial_x^{-1}uQ_{k}\partial_x^{-1}u\big)(x,t)dx 
\\ & +\beta\int_{\mathbb R}\big(uP_k\partial_x^{-1}uP_{k}\partial_x^{-1}u\big)(x,t)dx,
\end{split}
\end{equation}
for any $k \ge 1$, and
\begin{equation} \label{Es} 
E^s_T(u)=\|P_{\le 0}u(\cdot,0)\|_{L^2}^2+\sum_{k\ge 1}2^{2ks} \sup_{t_k \in [-T,T]}E_k(u)(t_k),
\end{equation}
where $\alpha$ and $\beta$ are two real numbers which will be fixed later. This modified energy may be seen as a localized version of the one introduced by Kwon in \cite{Kw}. The next lemma states that when $\|u\|_{L^{\infty}_TH^s_x}$ is small, then $E^s_T(u)$ and $\|u\|_{B^s(T)}^2$ are comparable.
\begin{lemma} \label{lemmaEE} 
Let $s>\frac12$. Then, there exists $0<\delta_0$ such that 
\begin{equation} \label{lemmaEE1} 
\frac12 \|u\|_{B^s(T)}^2 \le E^s_T(u) \le \frac32 \|u\|_{B^s(T)}^2,
\end{equation}
for all $u \in B^s(T) \cap C([-T,T];H^s(\mathbb R))$ satisfying $\|u\|_{L^{\infty}_TH^s_x} \le \delta_0$.
\end{lemma}

\begin{proof} First observe that, due to the Sobolev embedding $H^s(\mathbb R) \hookrightarrow L^{\infty}(\mathbb R)$,
\begin{displaymath} 
\Big|\int_{\mathbb R}\big(uP_k\partial_x^{-1}uQ_{k}\partial_x^{-1}u\big)(x,t_k)dx\Big|
\lesssim \|u\|_{L^{\infty}_TH^s_x}\sum_{|k-k'|\le 3}\|P_{k'}u(\cdot,t_k)\|_{L^2}^2,
\end{displaymath}
for all $k\ge 1$. It follows that  
\begin{displaymath} 
\begin{split}
E_k(u)(t) & \ge \|P_ku(t)\|_{L^2}^2-c|\alpha|\|u\|_{L^{\infty}_TH^s_x}\|P_k(u)(t)\|_{L^2}\sum_{|k-k'|\le 3}\|P_{k'}u(\cdot,t_k)\|_{L^2}^2\\ & \quad
-c|\beta|\|u\|_{L^{\infty}_TH^s_x}\|P_k(u)(t)\|_{L^2}^2,
\end{split}
\end{displaymath}
for any $t \in [-T,T]$ and $k \ge 1$.  Thus, if we choose $\|u\|_{L^{\infty}_TH^s} \le \delta_0$ with $\delta_0$ small enough, we obtain that 
\begin{displaymath} 
\begin{split}
E_k(u)(t) & \ge \frac34\|P_ku(t)\|_{L^2}^2-\frac1{50}\sup_{t_{k+1} \in [-T,T]}\|P_k(u)(t_{k+1})\|_{L^2}^2 \\ & \quad
-\frac1{50}\sup_{t_{k-1} \in [-T,T]}\|P_k(u)(t_{k-1})\|_{L^2}^2,
\end{split}
\end{displaymath}
which implies the first inequality in \eqref{lemmaEE1} after taking the supreme over $t \in [-T,T]$ and summing in $k \ge 1$. The second inequality in \eqref{lemmaEE1} follows similarly.
\end{proof}

\begin{proposition} \label{propEE}
Assume $s \ge 1$ and $T \in (0,1]$. Then, if $u \in C([-T,T];H^{\infty}(\mathbb R))$ is a solution to \eqref{5KdV} with $c_3=0$, we have that 
\begin{equation} \label{propEE1} 
E^s_T(u) \lesssim (1+\|u_0\|_{H^s})\|u_0\|_{H^s}^2+\big(1+ \|u\|_{F^{\frac34+}(T)}\big)\|u\|_{F^{\frac54}(T)}\|u\|_{F^s(T)}^2+\|u\|_{F^{\frac34+}(T)}^2\|u\|_{B^s(T)}^2.
\end{equation}
\end{proposition}

As a Corollary to Lemma \ref{lemmaEE} and Proposition \ref{propEE}, we deduce an \textit{a priori} estimate in $\|\cdot\|_{B^s(T)}$ for smooth solutions to \eqref{5KdV}. 
\begin{corollary} \label{coroEE}
Assume $s\ge 1$ and $T \in (0,1]$. Then, there exists $0<\delta_0 \le 1$ such that
\begin{equation} \label{coroEE1} 
\|u\|_{B^s(T)}^2 \lesssim \|u_0\|_{H^s}^2+\big(1+ \|u\|_{F^{\frac34+}(T)}\big)\|u\|_{F^{\frac54}(T)}\|u\|_{F^s(T)}^2+\|u\|_{F^{\frac34+}(T)}^2\|u\|_{B^s(T)}^2,
\end{equation}
for all solutions $u$ to \eqref{5KdV} with $c_3=0$ and satisfying $u \in C([-T,T];H^{\infty}(\mathbb R))$ and $\|u\|_{L^{\infty}_TH^{\frac12+}_x} < \delta_0$.
\end{corollary}

We split the proof of Proposition \ref{propEE} in several lemmas. 
\begin{lemma} \label{tecEE} 
Assume that $T \in (0,1]$, $k_1, \ k_2, \ k_3 \in \mathbb Z_+$  and that $u_j \in F_{k_j}$ for $j=1,2,3$. 
\begin{itemize} 
\item[(a)] In the case $k_{min} \le k_{max}-5$, it holds that
\begin{equation} \label{tecEE1}
\Big|\int_{\mathbb R \times [0,T]}u_1u_2u_3 dxdt \Big| \lesssim 2^{-k_{max}} \prod_{j=1}^3\|u_j\|_{F_{k_j}}.
\end{equation}
If moreover $k_{min} \ge 1$, we also have that
\begin{equation} \label{tecEE1b}
\Big|\int_{\mathbb R \times [0,T]}u_1u_2u_3 dxdt \Big| \lesssim 2^{-\frac32 k_{max}}2^{-k_{min}} \prod_{j=1}^3\|u_j\|_{F_{k_j}}.
\end{equation}
\item[(b)]In the case $|k_{min}-k_{max}| \le 10$,  it holds that
\begin{equation} \label{tecEE2}
\Big|\int_{\mathbb R \times [0,T]}u_1u_2u_3 dxdt \Big| \lesssim 2^{-7k_{max}/4} \prod_{j=1}^3\|u_j\|_{F_{k_j}}.
\end{equation}
\end{itemize}
\end{lemma}

The following technical result will be needed in the proof of Lemma \ref{tecEE}. 
\begin{lemma} \label{interval}
Assume $k \in \mathbb Z_+$ and $I \subset \mathbb R$ is an interval. Then 
\begin{equation} \label{interval1}
\sup_{j \in \mathbb Z_+} 2^{j/2} \big\| \eta_j(\tau-w(\xi))\mathcal{F}(\mathbf{1}_I(t) f)\big\|_{L^2} \lesssim \|\mathcal{F}(f)\|_{X_k},
\end{equation}
for all $f$ such that $\mathcal{F}(f) \in X_k$.
\end{lemma}

\begin{proof} Fix $j \in \mathbb Z_+$. We can also assume that $j \ge 5$. By writing  
\begin{displaymath} 
f=\sum_{q \ge 0}\mathcal{F}^{-1}\big[\eta_q(\tau-w(\xi))\mathcal{F}(f)(\xi,\tau)  \big]=: \sum_{q \ge 0}f_q,
\end{displaymath}
we have that
\begin{equation} \label{interval2}
2^{j/2} \big\| \eta_j(\tau-w(\xi))\mathcal{F}(\mathbf{1}_I(t) f)\big\|_{L^2} \le 2^{j/2}\sum_{q \ge 0} \big\| \eta_j(\tau-w(\xi))\mathcal{F}(\mathbf{1}_I(t) f_q)\big\|_{L^2}.
\end{equation}
On the one hand, Plancherel's identity implies that 
\begin{equation} \label{interval3}
2^{j/2}\sum_{q \ge j-5} \big\| \eta_j(\tau-w(\xi))\mathcal{F}(\mathbf{1}_I(t) f_q)\big\|_{L^2}
\lesssim \sum_{q \ge j-5}2^{q/2}\big\|\eta_q(\tau-w(\xi))\mathcal{F}(f) \big\|_{L^2}.
\end{equation}
On the other hand, we have that $|\mathcal{F}_t(\textbf{1}_I)(\tau)| \lesssim \frac1{|\tau|}$, since $I$ is an interval of $\mathbb R$. Thus, we deduce by applying the Cauchy-Schwarz inequality in $\tau'$ that 
\begin{equation} \label{interval4}
\begin{split} 
2^{j/2}\sum_{q =0}^{j-4} \big\| &\eta_j(\tau-w(\xi))\mathcal{F}(\mathbf{1}_I(t) f_q)\big\|_{L^2} 
\\ & \lesssim 2^{j/2}\sum_{q =0}^{j-4} \big\| \eta_j(\tau-w(\xi))\int_{\mathbb R}|\mathcal{F}(f)(\xi,\tau')|\frac{\eta_q(\tau'-w(\xi))}{|\tau-\tau'|}d\tau'\big\|_{L^2_{\xi,\tau}} \\ & \lesssim
\sum_{q =0}^{j-4}2^{q/2}\big\|\eta_q(\tau-w(\xi))\mathcal{F}(f) \big\|_{L^2},
\end{split}
\end{equation}
since $|\tau-\tau'| \sim 2^j$ in this case.

We deduce estimate \eqref{interval1} gathering \eqref{interval2}--\eqref{interval4} and taking the supreme in $j$.
\end{proof}

\begin{proof}[Proof of Lemma \ref{tecEE}] Assume without loss of generality that $k_1 \le k_2 \le k_3$. Moreover, due to the frequency localization,  we must have $|k_2-k_3|\le 4$. We first prove estimate \eqref{tecEE1}. Let $\beta: \mathbb R \rightarrow [0,1]$ be a smooth function supported in $[-1,1]$ with the property that 
\begin{displaymath} 
\sum_{m \in \mathbb Z}\beta^3(x-m)=1, \quad \forall \, x \in \mathbb R.
\end{displaymath}
Then, it follows that 
\begin{equation} \label{tecEE3}
\begin{split}
\Big|\int_{\mathbb R \times [0,T]}u_1u_2u_3 dxdt  \Big|  \lesssim \!
\sum_{|m| \le C2^{2k_3}}\Big| \int_{\mathbb R^2}\prod_{i=1}^3\big(\beta(2^{2k_3}t-m)\textbf{1}_{[0,T]}u_i\big)dxdt\Big|.
\end{split}
\end{equation}
Now we observe that the sum on the right-hand side of \eqref{tecEE2} is taken over the two disjoint sets
\begin{displaymath} 
\mathcal{A} = \big\{ m \in \mathbb Z \ : \ \beta(2^{2k_3}t-m)\textbf{1}_{[0,T]}=\beta(2^{2k_3}t-m)\big\},
\end{displaymath}
and 
\begin{displaymath} 
\mathcal{B} = \big\{ m \in \mathbb Z \ : \ \beta(2^{2k_3}t-m)\textbf{1}_{[0,T]} \neq\beta(2^{2k_3}t-m) \ \text{and} \  
\beta(2^{2k_3}t-m)\textbf{1}_{[0,T]} \neq 0\big\}.
\end{displaymath}

To deal with the sum over $\mathcal{A}$, we set
\begin{displaymath} 
f_{k_i,2k_3}^m=\eta_{\le 2k_3}(\tau-w(\xi))\big|\mathcal{F}\big(\beta(2^{2k_3}t-m)u_i\big) \big| 
\end{displaymath}
and 
\begin{displaymath} 
f_{k_i,j}^m=\eta_{j}(\tau-w(\xi))\big|\mathcal{F}\big(\beta(2^{2k_3}t-m)u_i\big) \big|, \quad \text{for} \ j > 2k_3
\end{displaymath}
for each $m \in \mathcal{A}$ and $i \in \{1,2,3\}$. Therefore, we deduce by using Plancherel's identity and estimates \eqref{L2bilin2b}, \eqref{L2bilin3}  that 
\begin{equation} \label{tecEE3b}
\begin{split}
\sum_{m \in \mathcal{A}}\Big| \int_{\mathbb R^2}\prod_{i=1}^3&\big(\beta(2^{2k_3}t-m)\textbf{1}_{[-T,T]}u_i\big)dxdt\Big| \\ & \lesssim \sup_{m \in \mathcal{A}}2^{2k_3} \sum_{j_1, j_2, j_3 \ge 2k_3}\int_{\mathbb R^2} f_{k_1,j_1}^m\ast f_{k_2,j_2}^m\cdot f_{k_3,j_3}^md\xi d\tau \\ & 
\lesssim \sup_{m \in \mathcal{A}}2^{-k_3}\prod_{i=1}^3\sum_{j_i\ge 2k_3}2^{j_i/2}\|f_{k_i,j_i}^m\|_{L^2}.
\end{split}
\end{equation}
This implies together with Lemma \ref{lemma2} that 
\begin{equation} \label{tecEE4}
\sum_{m \in \mathcal{A}}\Big| \int_{\mathbb R^2}\prod_{i=1}^3\big(\beta(2^{2k_3}t-m)\textbf{1}_{[0,T]}u_i\big)dxdt\Big| \lesssim  2^{-k_3} \prod_{j=1}^3\|u_j\|_{F_{k_j}}.
\end{equation}

Now observe that $\# \mathcal{B} \le 4$.  We set 
\begin{displaymath} 
g_{k_i,j}^m=\eta_{j}(\tau-w(\xi))\big|\mathcal{F}\big(\beta(2^{2k_3}t-m)\textbf{1}_{[0,T]}u_i\big) \big|, \end{displaymath}
for $i \in \{1,2,3\}$, $j \ge 0$ and $m \in \mathcal{B}$. Then, we deduce arguing as above and using Lemma \ref{interval} that 
\begin{equation} \label{tecEE5}
\begin{split}
\sum_{m \in \mathcal{B}}\Big| \int_{\mathbb R^2}\prod_{i=1}^3&\big(\beta(2^{2k_3}t-m)\textbf{1}_{[0,T]}u_i\big)dxdt\Big| \\ & \lesssim \sup_{m \in \mathcal{B}} \sum_{j_1, j_2, j_3 \ge 0}\int_{\mathbb R^2} g_{k_1,j_1}^m\ast g_{k_2,j_2}^m\cdot g_{k_3,j_3}^md\xi d\tau \\ & 
\lesssim \sup_{m \in \mathcal{B}}2^{-2k_3}\sum_{j_1, j_2, j_3 \ge 0}2^{-j_{med}/2}\prod_{i=1}^3\sup_{j_i \in \mathbb Z_+}2^{j_i/2}\|g_{k_i,j_i}^m\|_{L^2} \\ & 
\lesssim  2^{-2k_3}k_3 \prod_{j=1}^3\|u_j\|_{F_{k_j}}.
\end{split}
\end{equation}
Note that in the last step of \eqref{tecEE5}, we use the fact that $2^{j_{max}} \sim \max (2^{j_{med}},\Omega)$ to control the sum over $j_{max}$. Indeed, the case $2^{j_{max}} \sim 2^{j_{med}}$ is trivial, whereas in the case $2^{j_{max}} \sim \Omega$, we observe from \eqref{lresonance1} that $j_{max} \le 5k_3+6$.

We deduce estimate \eqref{tecEE1} gathering \eqref{tecEE3}--\eqref{tecEE5}.  
Note that estimate \eqref{tecEE1b} is obtained arguing as in \eqref{tecEE3}--\eqref{tecEE5} and by using \eqref{L2bilin2} instead of \eqref{L2bilin2b} and the fact that $2^{j_{max}} \gtrsim 2^{4k_3}2^{k_1}$ (c.f. Lemma \ref{lresonance}).

Finally, we only give a sketch of the proof of estimate \eqref{tecEE2} since it follows the same lines as the proof of estimate \eqref{tecEE1}. Note that under the assumption $|k_{min}-k_{med}| \le 4$, we have that $2^{k_1} \sim 2^{k_2} \sim 2^{k_3}$. Moreover, we can assume that $k_1 \ge 10$, since the proof is trivial otherwise by using \eqref{L2bilin1}. We introduce the same decomposition as in \eqref{tecEE3} and split the summation domain in $\mathcal{A}$ and $\mathcal{B}$. The estimates for the sum over the regions $\mathcal{A}$ and $\mathcal{B}$ follow by using \eqref{L2bilin3b} instead of \eqref{L2bilin2b} and \eqref{L2bilin3} 
and the fact that $2^{j_{max}} \gtrsim 2^{\frac52k_3}$ (c.f. Lemma \ref{lresonance}).

\end{proof}

\begin{lemma} \label{tec2EE} 
Assume that $T \in (0,1]$, $k, \ k_1 \in \mathbb Z_+$ satisfy $k_1 \le k-6$, $u \in F_{k_1}$ and $v \in F^0$. Then, it holds that 
\begin{equation} \label{tec2EE1}
\begin{split}
&\Big|\int_{\mathbb R \times [0,T]}P_kvP_k\partial_x\big(\partial_x^2vP_{k_1}u  \big) dxdt\\ -&\frac12 
\int_{\mathbb R \times [0,T]}P_k\partial_xvP_k\partial_xvP_{k_1}\partial_xu   dxdt  \\+&
 \int_{\mathbb R \times [0,T]}P_k\partial_xvQ_k\partial_xvP_{k_1}\partial_xu  dxdt \Big|\lesssim \Theta(k,k_1) \|P_{k_1}u\|_{F_{k_1}}\sum_{|k'-k| \le 3}\|P_{k'}v\|_{F_{k'}}^2,
\end{split}
\end{equation}
and
\begin{equation} \label{tec2EE2}
\begin{split}
\Big|\int_{\mathbb R \times [0,T]}P_kvP_k\big(P_{k_1}\partial_xu\partial_x^2v \big) &dxdt+
\int_{\mathbb R \times [0,T]}P_k\partial_xvP_k\partial_xvP_{k_1}\partial_xu  dxdt  \Big| \\ &
\lesssim \Theta(k,k_1) \|P_{k_1}u\|_{F_{k_1}}\sum_{|k'-k| \le 3}\|P_{k'}v\|_{F_{k'}}^2,
\end{split}
\end{equation}
where $\Theta(k,k_1)=2^{2k_1}$. Moreover, if $k_1 \ge 1$, we can choose $\Theta(k,k_1)=2^{k_1}2^{-k/2}$.
\end{lemma}

\begin{remark} \label{rematecEE} 
Lemma \ref{tecEE} does not permit to control the terms
\begin{displaymath}
\Big| \int_{\mathbb R \times [0,T]}P_k\partial_xvP_k\partial_xvP_{k_1}\partial_xu  dxdt \Big| \ 
\text{and} \  \Big| \int_{\mathbb R \times [0,T]}P_k\partial_xvQ_k\partial_xvP_{k_1}\partial_xu dxdt \Big|
\end{displaymath}
without loosing a $2^k$ factor, which would not be good to derive the energy estimates. For that reason, we need to modify the energy by a cubic term (c.f. \eqref{Es}) in order to cancel those two terms.
\end{remark}

\begin{proof}[Proof of Lemma \ref{tec2EE}] We first prove estimate \eqref{tec2EE1}. After integrating by part, we rewrite the term on the left-hand side of \eqref{tec2EE1} as 
\begin{displaymath} 
-\int_{\mathbb R \times [0,T]}P_k\partial_xv \Big( \big[P_k,P_{k_1}u \big]\partial_x^2v-Q_k\partial_xvP_{k_1}\partial_xu\Big) dxdt,
\end{displaymath}
where $[A,B]=AB-BA$ denotes the commutator of $A$ and $B$. Now, straightforward computations using \eqref{psik} lead to 
\begin{displaymath}
\begin{split}
\mathcal{F}\Big( \big[P_k,P_{k1}u &\big]\partial_x^2v-Q_k\partial_xvP_{k_1}\partial_xu\Big)(\xi,\tau) \\& = 
c\int_{\mathbb R^2}m(\xi,\xi_1) \mathcal{F}(P_{k_1}\partial_x^2u)(\xi_1,\tau_1) \mathcal{F}(v)(\xi-\xi_1,\tau-\tau_1) d\xi_1d\tau_1,
\end{split}
\end{displaymath}
where 
\begin{displaymath} 
\big|m(\xi,\xi_1)\big| := \Big|\frac{\eta_k(\xi)-\eta_k(\xi-\xi_1)-\eta_k'(\xi-\xi_1)\xi_1}{\xi_1^2}(\xi-\xi_1)^2 \Big| 
\lesssim 1,
\end{displaymath}
due to the Taylor-Lagrange theorem and the frequency localization on $\xi$ and $\xi_1$. Therefore estimate \eqref{tec2EE1} follows arguing exactly as in the proof of Lemma \ref{tecEE}.

To prove of estimate \eqref{tec2EE2}, we first observe integrating by parts that 
\begin{displaymath} 
\begin{split}
\int_{\mathbb R \times [0,T]}P_k\partial_xvP_k\partial_xvP_{k_1}\partial_xu  dxdt&=
-\int_{\mathbb R \times [0,T]}P_kvP_k\partial_x^2vP_{k_1}\partial_xu  dxdt \\ & \quad
-\int_{\mathbb R \times [0,T]}P_kvP_k\partial_xvP_{k_1}\partial_x^2u  dxdt
\end{split}
\end{displaymath} 
First, we apply estimates \eqref{tecEE1} and \eqref{tecEE1b} to obtain that
\begin{displaymath} 
\Big| \int_{\mathbb R \times [0,T]}P_kvP_k\partial_xvP_{k_1}\partial_x^2u_1  dxdt\Big| \lesssim \theta(k,k_1)\|P_{k_1}u\|_{F_{k_1}}\|P_kv\|_{F_k}^2.
\end{displaymath}
On the other hand, we observe that 
\begin{displaymath} 
\begin{split}
&\int_{\mathbb R \times [0,T]}P_kvP_k\big(\partial_x^2vP_{k_1}\partial_xu \big) dxdt-
\int_{\mathbb R \times [0,T]}P_kvP_k\partial_x^2vP_{k_1}\partial_xu  dxdt\\ 
&=\int_{\mathbb R \times [0,T]}P_kv \Big(\big[P_k,P_{k_1}\partial_xu \big]\partial_x^2v \Big)dxdt.
\end{split}
\end{displaymath}
An easy computation gives 
\begin{displaymath}
\begin{split}
\mathcal{F}\big( \big[P_k,P_{k1}\partial_xu &\big]\partial_x^2v\big)(\xi,\tau) \\& = 
c\int_{\mathbb R^2}\widetilde{m}(\xi,\xi_1) \mathcal{F}(P_{k_1}\partial_x^2u)(\xi_1,\tau_1) \mathcal{F}(\partial_xv)(\xi-\xi_1,\tau-\tau_1) d\xi_1d\tau_1,
\end{split}
\end{displaymath}
where 
\begin{displaymath} 
\big|\widetilde{m}(\xi,\xi_1)\big|=\Big| \frac{\eta_k(\xi)-\eta_k(\xi-\xi_1)}{\xi_1}(\xi-\xi_1)\Big| \lesssim 1,
\end{displaymath}
due to the mean value theorem and the frequency localization on $\xi$ and $\xi_1$. We finish the proof of estimate \eqref{tec2EE2} arguing exactly as in the proof of Lemma \ref{tec2EE}.
\end{proof}

\begin{lemma} \label{tec3EE} 
Assume that $T \in (0,1]$, $k_1, \ k_2, \ k_3, k_4 \in \mathbb Z_+$  and that $u_j \in F_{k_j}$ for $j=1,2,3,4$. 
If $k_{thd} \le k_{max}-5$, then it holds that
\begin{equation} \label{tec3EE1}
\Big|\int_{\mathbb R \times [0,T]}u_1u_2u_3u_4 dxdt \Big| \lesssim 2^{-k_{max}}2^{k_{min}/2} \prod_{j=1}^4\|u_j\|_{F_{k_j}}.
\end{equation}
If instead, $k_{min} \ll k_{thd} \sim k_{sub} \sim k_{max}$, then it holds that
\begin{equation} \label{tec3EE2}
\Big|\int_{\mathbb R \times [0,T]}u_1u_2u_3u_4 dxdt \Big| \lesssim 2^{-k_{max}}2^{k_{min}/2} \prod_{j=1}^4\|u_j\|_{F_{k_j}}.
\end{equation}
\end{lemma}
\begin{proof} The proof of estimates \eqref{tec3EE1}--\eqref{tec3EE2} follows arguing exactly as in the proof of  \eqref{tecEE1}. To prove estimate \eqref{tec3EE1}, we use estimates \eqref{L2trilin3}--\eqref{L2trilin4} instead of estimates \eqref{L2bilin2b} and \eqref{L2bilin3}. To prove estimate \eqref{tec3EE2}, we use estimate \eqref{L2trilin1} and observe that due to the frequency localization $\widetilde{\Omega} \sim 2^{5k_{max}}$, so that $j_{max} \ge 5k_{max}-C$, where $C$ is a fixed positive constant depending only on the frequency localization and $\widetilde{\Omega}$ was defined in \eqref{resonancetilde}. 
\end{proof}

Now we give the proof of Proposition \ref{propEE}. 
\begin{proof}[Proof of Proposition \ref{propEE}] 
Let $u \in C([-T,T];H^{\infty}(\mathbb R))$ be a solution to \eqref{5KdV} with $c_3=0$. We choose an extension $\widetilde{u}$ of $u$ on $\mathbb R^2$ satisfying 
\begin{equation} \label{propEE1b}
\widetilde{u}_{|_{\mathbb R\times [-T,T]}}=u \quad \text{and} \quad \|\widetilde{u}\|_{F^s} \le 2\|u\|_{F^s(T)}.
\end{equation} 
Then, for any $k \in \mathbb Z_+ \cap [1,+\infty)$ and $t \in [-T,T]$, we differentiate $E_k(u)$ with respect to $t$ and deduce using \eqref{5KdV} that 
\begin{equation} \label{propEE2}
\frac{d}{dt}E_k(u) = \mathcal{I}_k(u)+\mathcal{J}_k(u)+\alpha \mathcal{L}_k^1(u)+\alpha \mathcal{N}_k^1(u)+\beta \mathcal{L}_k^2(u)+\beta \mathcal{N}_k^2(u) ,
\end{equation}
where 
\begin{displaymath}
\mathcal{I}_k(u)=2c_1\int_{\mathbb R}P_kuP_k\partial_x\big((\partial_xu)^2\big)dx,
\end{displaymath}
\begin{displaymath} 
\mathcal{J}_k(u)=2c_2\int_{\mathbb R}P_kuP_k\partial_x\big(u\partial_x^2u\big)dx,
\end{displaymath}
\begin{displaymath} 
\begin{split}
\mathcal{L}_k^1(u)&=\int_{\mathbb R}\partial_x^5uP_k\partial_x^{-1}uQ_k\partial_x^{-1}udx+
\int_{\mathbb R}uP_k\partial_x^4uQ_k\partial_x^{-1}udx\\ & \quad +
\int_{\mathbb R}uP_k\partial_x^{-1}uQ_k\partial_x^4udx,
\end{split}
\end{displaymath} 

\begin{displaymath}  
\begin{split}
\mathcal{N}_k^1(u)&=c_1\int_{\mathbb R}\partial_x\big((\partial_xu)^2\big)P_k\partial_x^{-1}uQ_k\partial_x^{-1}udx+
c_1\int_{\mathbb R}uP_k\big((\partial_xu)^2\big)Q_k\partial_x^{-1}udx\\ & \quad +
c_1\int_{\mathbb R}uP_k\partial_x^{-1}uQ_k\big((\partial_xu)^2\big)dx+c_2
\int_{\mathbb R}\partial_x\big(u\partial_x^2u\big)P_k\partial_x^{-1}uQ_k\partial_x^{-1}udx \\ &
\quad+c_2\int_{\mathbb R}uP_k\big(u\partial_x^2u\big)Q_k\partial_x^{-1}udx
+c_2\int_{\mathbb R}uP_k\partial_x^{-1}uQ_k\big(u\partial_x^2u\big)dx,
\end{split}
\end{displaymath}

\begin{displaymath}  
\mathcal{L}_k^2(u)=\int_{\mathbb R}\partial_x^5uP_k\partial_x^{-1}uP_k\partial_x^{-1}udx+
2\int_{\mathbb R}uP_k\partial_x^4uP_k\partial_x^{-1}udx,
\end{displaymath}
and 
\begin{equation} \label{propEE2b}
\begin{split}
\mathcal{N}_k^2(u)&=c_1\int_{\mathbb R}\partial_x\big((\partial_xu)^2\big)P_k\partial_x^{-1}uP_k\partial_x^{-1}udx+
2c_1\int_{\mathbb R}uP_k\big((\partial_xu)^2\big)P_k\partial_x^{-1}udx\\ & \quad +
c_2\int_{\mathbb R}\partial_x\big(u\partial_x^2u\big)P_k\partial_x^{-1}uP_k\partial_x^{-1}udx 
+2c_2\int_{\mathbb R}uP_k\big(u\partial_x^2u\big)P_k\partial_x^{-1}udx.
\end{split}
\end{equation}
Now, we fix $t_k \in [-T,T]$. Without loss of generality, we can assume that $0<t_k \le T$. Therefore, we obtain integrating \eqref{propEE2} between $0$ and $t_k$ that 
\begin{equation} \label{propEE3}
\begin{split}
&E_k(u)(t_k)-E_k(u)(0) \\ & \le \Big|\int_{[0,t_k]}\big(\mathcal{I}_k(u)+\mathcal{J}_k(u)+\alpha \mathcal{L}_k^1(u)+\alpha \mathcal{N}_k^1(u)+\beta \mathcal{L}_k^2(u)+\beta \mathcal{N}_k^2(u)\big)dt\Big|.
\end{split}
\end{equation}

Next we estimate the right-hand side of \eqref{propEE3}. \\

\noindent \textit{Estimates for the cubic terms.} We deduce after some integrations by parts that 
\begin{displaymath} 
\begin{split}
\mathcal{L}_k^1(u)&=\int_{\mathbb R}\partial_x^3uP_k\partial_xuQ_k\partial_x^{-1}udx+
2\int_{\mathbb R}\partial_x^3uP_kuQ_kudx+\int_{\mathbb R}\partial_x^3uP_k\partial_x^{-1}uQ_k\partial_xudx
\\ &\ +
\int_{\mathbb R}\partial_x^2uP_k\partial_x^2uQ_k\partial_x^{-1}udx+
2\int_{\mathbb R}\partial_xuP_k\partial_x^2uQ_kudx
+\int_{\mathbb R}uP_k\partial_x^2uQ_k\partial_xudx\\ &  \ +
\int_{\mathbb R}\partial_x^2uP_k\partial_x^{-1}uQ_k\partial_x^{2}udx
+2\int_{\mathbb R}\partial_xuP_kuQ_k\partial_x^2udx
+\int_{\mathbb R}uP_k\partial_xuQ_k\partial_x^2udx \\ & 
=5\int_{\mathbb R}\partial_x^3uP_kuQ_kudx-5\int_{\mathbb R}\partial_xuP_k\partial_xuQ_k\partial_xudx.
\end{split}
\end{displaymath}
Similarly it holds that
\begin{displaymath} 
\mathcal{L}_k^2(u)
=5\int_{\mathbb R}\partial_x^3uP_kuP_kudx-5\int_{\mathbb R}\partial_xuP_k\partial_xuP_k\partial_xudx.
\end{displaymath}
We choose $\alpha=-\frac{2c_2}5$ and $\beta=\frac{c_2-4c_1}5$. Then it follows, after performing a dyadic decomposition on $u$, that 
\begin{equation} \label{propEE3b}
\Big|\int_{[0,t_k]}\big(\mathcal{I}_k(u)+\mathcal{J}_k(u)+\alpha\mathcal{L}_k^1(u)+\beta\mathcal{L}_k^2(u)\big)dt\Big| 
\lesssim  \sum_{j=1}^7T_j(k),
\end{equation}
for each $k \ge 1$, with 
\begin{displaymath} T_1(k)=\sum_{0 \le k_1 \le k-6}\Big|\int_{\mathbb R \times [0,t_k]}\Big(P_kuP_k\big(\partial_x^2uP_{k_1}\partial_xu\big) +P_k\partial_xuP_k\partial_xuP_{k_1}\partial_xu \big) dxdt\Big|, 
\end{displaymath}
\begin{displaymath} 
T_2(k)=\sum_{0 \le k_1 \le k-6}\Big| \int_{\mathbb R \times [0,t_k]}P_k\partial_xu \Big( \big[P_k,P_{k_1}u \big]\partial_x^2u-Q_k\partial_xuP_{k_1}\partial_xu\Big) dxdt\Big|,
\end{displaymath}
\begin{displaymath} 
T_3(k)=\sum_{k_1 \ge k-5, k_2 \ge 0}\Big| \int_{\mathbb R \times [0,t_k]}P_k^2uP_{k_1}\partial_xuP_{k_2}\partial_x^2udt\Big|,
\end{displaymath}
\begin{displaymath} 
T_4(k)=\sum_{k_1 \ge k-5, k_2 \ge 0}\Big| \int_{\mathbb R \times [0,t_k]}P_k^2\partial_xuP_{k_1}uP_{k_2}\partial_x^2udt\Big|,
\end{displaymath}
\begin{displaymath} 
T_5(k)=\sum_{k-5\le k_1 \le k+4}\Big| \int_{\mathbb R \times [0,t_k]}P_{k_1}\partial_xuP_k\partial_xu\big(Q_k\partial_xu+P_k\partial_xu\big)dxdt\Big|,
\end{displaymath}
\begin{displaymath} 
T_6(k)=\sum_{k_1 \le k-5}\Big|\int_{\mathbb R \times [0,t_k]}P_{k_1}\partial_x^3uP_ku\big(Q_ku+P_ku\big)dx \Big| \end{displaymath}
and
\begin{displaymath} 
T_7(k)=\sum_{k-4 \le k_1 \le k+4}\Big|\int_{\mathbb R \times [0,t_k]}P_{k_1}\partial_x^3uP_ku\big(Q_ku+P_ku\big)dx \Big|.
\end{displaymath}
Clearly, Lemma \ref{tec2EE} and the Cauchy-Schwarz inequality imply that 
\begin{equation} \label{propEE4} 
\begin{split}
T_1(k)+T_2(k) & \lesssim \sum_{0 \le k_1 \le k-6}2^{k_1/2}\|P_{k_1}\widetilde{u}\|_{F_{k_1}}\sum_{|k-k'| \le3}\|P_{k'}\widetilde{u}\|_{F_{k'}}^2 \\ & \lesssim \|\widetilde{u}\|_{F^{1/2+}}\sum_{|k-k'| \le3}\|P_{k'}\widetilde{u}\|_{F_{k'}}^2 .
\end{split}
\end{equation}
Similarly, we get applying estimate \eqref{tecEE1} if $k_1=0$, and estimate \eqref{tecEE1b} if $k_1>0$, that 
\begin{equation} \label{propEE5} 
T_6(k)  \lesssim  \|\widetilde{u}\|_{F^{\frac12+}}\sum_{|k-k'| \le3}\|P_{k'}\widetilde{u}\|_{F_{k'}}^2 .
\end{equation} 
Now, estimate \eqref{tecEE2} leads to 
\begin{equation} \label{propEE6} 
T_5(k)+T_7(k)  \lesssim  \|\widetilde{u}\|_{F^{\frac54}}\|P_{k}\widetilde{u}\|_{F_{k}}^2 .
\end{equation}
To estimate $T_3(k)$, when $k \ge 1$ is given, we denote 
\begin{displaymath} 
\begin{array}{l}
B_1= \big\{ (k,k_1) \in \mathbb Z_+^2 \ : \ |k-k_1|\le 3 \ \text{and} \  0 \le k_2 \le \max(k,k_1)-5 \big\},  \\ 
B_2=\big\{ (k,k_1) \in \mathbb Z_+^2 \ : \ |k-k_1|\le 5 \ \text{and} \  |k_2-k_1| \le 5 \big\}, \\
B_3= \big\{ (k,k_1) \in \mathbb Z_+^2 \ : \ |k_2-k_1|\le 3 \ \text{and} \  1 \le k \le \max(k_1,k_2)-5 \big\}.
\end{array}
\end{displaymath}
Thus, we deduce from the frequency localization that 
\begin{displaymath} 
T_3(k)=\sum_{j=1}^3\sum_{(k_1,k_2) \in B_j}\Big| \int_{\mathbb R \times [0,t_k]}P_k^2uP_{k_1}\partial_xuP_{k_2}\partial_x^2udt\Big|.
\end{displaymath}
To estimate the sum over $B_1$, we use estimate \eqref{tecEE1} in the case $k_2=0$ and estimate \eqref{tecEE1b} in the case $k_2 \ge 1$. It  follows that 
\begin{displaymath} 
\sum_{(k_1,k_2) \in B_1}\Big| \int_{\mathbb R \times [0,t_k]}P_k^2uP_{k_1}\partial_xuP_{k_2}\partial_x^2udt\Big| \lesssim  \|\widetilde{u}\|_{F^{\frac12+}}\sum_{|k'-k| \le 3}\|P_{k'}\widetilde{u}\|_{F_{k'}}^2.
\end{displaymath} 
The sum over $B_2$ is treated by using estimate \eqref{tecEE2}, which gives 
\begin{displaymath} 
\sum_{(k_1,k_2) \in B_2}\Big| \int_{\mathbb R \times [0,t_k]}P_k^2uP_{k_1}\partial_xuP_{k_2}\partial_x^2udt\Big| \lesssim  \|\widetilde{u}\|_{F^{\frac54}}\sum_{|k'-k| \le 3}\|P_{k'}\widetilde{u}\|_{F_{k'}}^2.
\end{displaymath} 
Finally, estimate \eqref{tecEE1b} (recall here that $k \ge 1$)  yields
\begin{displaymath}
 \sum_{(k_1,k_2) \in B_3}\Big| \int_{\mathbb R \times [0,t_k]}P_k^2uP_{k_1}\partial_xuP_{k_2}\partial_x^2udt\Big| \lesssim  2^{-k}\|P_k\widetilde{u}\|_{F_k}\sum_{k_1 \ge k+5}2^{3k_1/2}\|P_{k_1}\widetilde{u}\|_{F_{k_1}}^2.
\end{displaymath}
Moreover, observe that  the same estimates also hold for $T_4(k)$ (with even a better bound when the sum is taken over $B_3$). This implies that 
\begin{equation} \label{propEE7}
\begin{split}
T_3(k)+T_4(k)  &\lesssim \|\widetilde{u}\|_{F^{\frac54}}\sum_{|k'-k| \le 3}\|P_{k'}\widetilde{u}\|_{F_k}^2 \\  & \quad+2^{-k}\|P_k\widetilde{u}\|_{F_k}\sum_{k_1 \ge k+5}2^{3k_1/2}\|P_{k_1}\widetilde{u}\|_{F_{k_1}}^2.
\end{split}
\end{equation}
Therefore, we deduce  gathering \eqref{propEE3b}--\eqref{propEE7}, taking the supreme over $t_k \in [0,T]$, summing in $k$ and using \eqref{propEE1b} that 
\begin{equation} \label{propEE8} 
\begin{split}
&\sum_{k \ge 1}2^{2ks} \sup_{t_k \in [0,T]}\Big|\int_{[0,t_k]}\big(\mathcal{I}_k(u)+\mathcal{J}_k(u)+\alpha\mathcal{L}_k^1(u)+\beta\mathcal{L}_k^2(u)\big)dt\Big| 
\\ & \  \lesssim \sum_{k \ge 1}2^{2ks}\|\widetilde{u}\|_{F^{\frac54}}\|P_{k}\widetilde{u}\|_{F_k}^2+\sum_{k \ge 1}2^{k(2s-1)}\|P_k\widetilde{u}\|_{F_k}\sum_{k_1 \ge k+5}2^{3k_1/2}\|P_{k_1}\widetilde{u}\|_{F_{k_1}}^2 \\ 
& \lesssim \|u\|_{F^{\frac54}(T)}\|u\|_{F^s(T)}^2+\|u\|_{F^{\frac12+}(T)}\|u\|_{F^s(T)}^2.
\end{split}
\end{equation}
Note that we use that $s \ge 1$ and apply the Cauchy-Schwarz inequality in $k$ and $k_1$ to obtain the last inequality in 
\eqref{propEE8}. \\

\noindent \textit{Estimates for the fourth order terms.} We estimate the fourth order term corresponding to $ \mathcal{N}_k^2(u) $. After a few integration by parts in \eqref{propEE2b}, we get that 
\begin{equation} \label{propEE8b}
\Big|\int_{[0,t_k]} \mathcal{N}_k^2(u) dt\Big|  \lesssim \sum_{i=1}^4X_i(k).
\end{equation}
for each $k \ge 1$, whith 
\begin{displaymath}
\begin{split}
X_1(k)&=\Big|\int_{\mathbb R \times [0,t_k]}\big(\partial_xu \big)^2P_k\partial_x^{-1}uP_ku dxdt\Big| + \Big|\int_{\mathbb R \times [0,t_k]}u\partial_xu P_kuP_ku dxdt\Big|  \\ 
& \quad +\Big|\int_{\mathbb R \times [0,t_k]}u\partial_xu P_k\partial_xuP_k\partial_x^{-1}u dxdt\Big| ,
\end{split}
\end{displaymath}
\begin{displaymath}
X_2(k)= \Big|\int_{\mathbb R \times [0,t_k]}uP_k\big((\partial_xu)^2\big)P_k\partial_x^{-1}u dxdt\Big| ,
\end{displaymath}
\begin{displaymath} 
X_3(k)=\Big|\int_{\mathbb R \times [0,t_k]}\partial_xuP_k\big(u\partial_xu\big)P_k\partial_x^{-1}u dxdt\Big| 
\end{displaymath}
and 
\begin{displaymath}
X_4(k)= \Big|\int_{\mathbb R \times [0,t_k]}uP_k\big(u\partial_xu\big)P_ku dxdt\Big| .
\end{displaymath}
We use the Strichartz estimate \eqref{Bstrichartz1} with $\alpha=2$, estimate \eqref{lemma1.1} and H\"older's inequality to deduce that 
\begin{equation} \label{propEE9} 
\begin{split}
\sum_{k \ge 1}2^{2ks} \sup_{t_k \in [0,T]} X_1(k)  & \lesssim \big( \|u\|_{L^2_TL^{\infty}_x}+\|\partial_xu\|_{L^2_TL^{\infty}_x}\big) \|\partial_xu\|_{L^2_TL^{\infty}_x}\sum_{k \ge 1}2^{2ks}\|P_ku\|_{L^{\infty}_TL^2_x}^2 \\
& \lesssim \big(\|u\|_{F^{\frac12+}(T)}+\|u\|_{F^{\frac34+}(T)})\|u\|_{F^{\frac34+}(T)}\|u\|_{B^s(T)}^2.
\end{split}
\end{equation}

To deal with $X_2(k)$, we perform dyadic decompositions over $u$ and $\partial_xu$. Then 
\begin{displaymath} 
X_2(k) \le \sum_{j=1}^4 \sum_{(k_1,k_2) \in D_j} 
\Big| \int_{\mathbb R \times [0,t_k]}uP_k\big(P_{k_1}\partial_xuP_{k_2}\partial_xu\big)P_k\partial_x^{-1}udxdt \Big|,
\end{displaymath}
where 
\begin{equation} \label{propEE9b} 
\begin{array}{l}
D_1= \big\{ (k_1,k_2) \in \mathbb Z_+^2 \ : \ |k-k_1|\le 3 \ \text{and} \  0 \le k_2 \le \max(k,k_1)-5 \big\},  \\ 
D_2= \big\{ (k_1,k_2) \in \mathbb Z_+^2 \ : \ |k-k_2|\le 3 \ \text{and} \  0 \le k_1 \le \max(k,k_2)-5 \big\},  \\ 
D_3=\big\{ (k_1,k_2) \in \mathbb Z_+^2 \ : \ |k-k_1|\le 5 \ \text{and} \  |k_2-k_1| \le 5 \big\}, \\
D_4= \big\{ (k_1,k_2) \in \mathbb Z_+^2 \ : \ |k_2-k_1|\le 3 \ \text{and} \  1 \le k \le \max(k_1,k_2)-5 \big\}.
\end{array}
\end{equation}
By using H\"older's inequality and the Cauchy-Schwarz inequality, we can bound the sum over $D_1 \cup D_2$ by
\begin{displaymath} 
\|u\|_{L^{\infty}_{T,x}}\|D_x^{1+}u\|_{\widetilde{L^2_TL^{\infty}_x}}\sum_{|k'-k| \le 3}\|P_{k'}u\|_{L^{\infty}_TL^2_x}^2.
\end{displaymath}
Thus,  it follows from  estimates \eqref{lemma1.1} and \eqref{Bstrichartz1b} that
\begin{displaymath} 
\begin{split}
\sum_{(k_1,k_2) \in D_1 \cup D_2} 
\Big| \int_{\mathbb R \times [0,t_k]}&uP_k\big(P_{k_1}\partial_xuP_{k_2}\partial_xu\big)P_k\partial_x^{-1}udxdt \Big|
\\ & \lesssim \|u\|_{F^{\frac12+}(T)}\|u\|_{F^{\frac34+}(T)}\sum_{|k'-k| \le 3}\|P_{k'}u\|_{L^{\infty}_TL^2_x}^2.
\end{split}
\end{displaymath}
A similar bound holds over $D_3$.  In the region $D_4$, we have that 
\begin{displaymath} 
\begin{split}
\sum_{(k_1,k_2) \in D_4} 
\Big| &\int_{\mathbb R \times [0,t_k]}uP_k\big(P_{k_1}\partial_xuP_{k_2}\partial_xu\big)P_k\partial_x^{-1}udxdt \Big|
\\ & \lesssim \|u\|_{F^{\frac12+}(T)}2^{-k}\|P_ku\|_{L^{\infty}_TL^2_x}\sum_{k_1 \ge k+5}2^{2k_1}\|P_{k_1}u\|_{L^2_TL^{\infty}_x}\|P_{k_1}u\|_{L^{\infty}_TL^2_x}.
\end{split}
\end{displaymath}
Hence, we deduce after taking the supreme of $t_k$ over $[0,T]$, summing over $k \in \mathbb Z_+ \cap [1,+\infty)$ and using estimate \eqref{Bstrichartz1b} that 
\begin{equation} \label{propEE10} 
\sum_{k \ge 1}2^{2ks} \sup_{t_k \in [0,T]}X_2(k)    \lesssim \|u\|_{F^{\frac12+}(T)} \|u\|_{F^{\frac34+}(T)}\|u\|_{B^{s}(T)}^2.
 \end{equation}
 Similarly, we get that 
 \begin{equation} \label{propEE11} 
\sum_{k \ge 1}2^{2ks} \sup_{t_k \in [0,T]}X_3(k)    \lesssim \|u\|_{F^{\frac34+}(T)}^2\|u\|_{B^{s}(T)}^2.
 \end{equation}
 
 To deal with $X_4(k)$, we use the following decomposition
 \begin{equation} \label{propEE12}
 \begin{split}
X_4(k) &\le \sum_{k_1 \ge k-7} 
\Big| \int_{\mathbb R \times [0,t_k]}P_{k_1}uP_k\big(u\partial_xu\big)P_kudxdt \Big| \\ & \quad+
\sum_{j=1}^4\sum_{(k_1, k_2, k_3) \in E_j} 
\Big| \int_{\mathbb R \times [0,t_k]}P_{k_1}uP_k\big(P_{k_2}u\partial_xP_{k_3}u\big)P_kudxdt \Big|\\ & =: 
\sum_{j=0}^4X_{4,j}(k),
\end{split}
\end{equation} 
where 
\begin{displaymath} 
\begin{array}{l}
E_1= \big\{ (k_1,k_2,k_3) \in \mathbb Z_+^3 \ : \ 0 \le k_1 \le k-8, \ |k-k_3|\le 3, \  0 \le k_2 \le \max(k,k_3)-5 \big\},  \\ 
E_2= \big\{ (k_1,k_2,k_3) \in \mathbb Z_+^3 \ : \ 0 \le k_1 \le k-8, \ |k-k_2|\le 3, \  0 \le k_3 \le \max(k,k_2)-5 \big\},  \\ 
E_3= \big\{ (k_1,k_2,k_3) \in \mathbb Z_+^3 \ : \ 0 \le k_1 \le k-8, \ |k-k_2|\le 5, \   |k-k_3|\le 5 \big\}, \\
E_4= \big\{ (k_1,k_2,k_3) \in \mathbb Z_+^3 \ : \ 0 \le k_1 \le k-8, \ |k_2-k_3|\le 3, \   1 \le k \le \max(k_2,k_3)-5 \big\}.
\end{array}
\end{displaymath}
Observe that, according to estimates \eqref{lemma1.1} and \eqref{Bstrichartz1} 
\begin{equation} \label{propEE13} 
\begin{split}
X_{4,0}(k)&=\sum_{ k-7 \le k_1 \le k+3} 
\Big| \int_{\mathbb R \times [0,t_k]}u\partial_xuP_k\big(P_{k_1}uP_ku\big)dxdt \Big| \\ 
& \lesssim  \|u\|_{F^{\frac12+}(T)}\|u\|_{F^{\frac34+}(T)}\sum_{|k'-k|\le 7}\|P_{k'}u\|_{L^{\infty}_TL^2_x}^2.
\end{split}
\end{equation}
Now, by using estimate \eqref{tec3EE1}, we get that 
\begin{equation} \label{propEE14} 
X_{4,1}(k)+X_{4,2}(k)\lesssim \ \|\widetilde{u}\|_{F^{\frac12+}}\|\widetilde{u}\|_{F^{0+}}\sum_{|k'-k|\le 3}\|P_{k'}\widetilde{u}\|_{F_{k'}}^2.
\end{equation}
Over the region $E_3$, we deduce from estimates \eqref{lemma1.1} and \eqref{Bstrichartz1b} that
\begin{equation} \label{propEE15} 
X_{4,3}(k)\lesssim \|u\|_{F^{\frac12+}(T)}\|u\|_{F^{\frac34+}(T)}\sum_{|k'-k| \le 5}\|P_{k'}u\|_{L^{\infty}_TL^2_x}^2.
\end{equation}
Finally, estimate \eqref{lemma1.1} gives
\begin{equation} \label{propEE16} 
X_{4,4}(k)\lesssim \|u\|_{F^{\frac12+}(T)}\|P_ku\|_{L^{\infty}_TL^2_x}\sum_{k_3 \ge k+5}\|P_{k_3}u\|_{L^{\infty}_TL^2_x}\|P_{k_3}\partial_xu\|_{L^2_TL^{\infty}_x}.
\end{equation}
Thus, we deduce from \eqref{propEE12}--\eqref{propEE16} that 
\begin{equation} \label{propEE17} 
\sum_{k \ge 1}2^{2ks} \sup_{t_k \in [0,T]}X_4(k)    \lesssim \|u\|_{F^{\frac34+}(T)}^2\|u\|_{B^{s}(T)}^2+\|u\|_{F^{0+}(T)}\|u\|_{F^{\frac12+}(T)}\|u\|_{F^s(T)}^2.
 \end{equation}

Therefore, we conclude gathering \eqref{propEE8b}--\eqref{propEE11} and \eqref{propEE17} that 
 \begin{equation} \label{propEE18} 
 \begin{split}
 \sum_{k \ge 1}2^{2ks} \sup_{t_k \in [0,T]}&\Big|\int_{[0,t_k]} \mathcal{N}_k^2(u) dt\Big| \\ & \lesssim \|u\|_{F^{\frac34+}(T)}^2\|u\|_{B^s(T)}^2+ \|u\|_{F^{0+}(T)}\|u\|_{F^{\frac12+}(T)} 
 \|u\|_{F^{s}(T)}^2. 
 \end{split}
 \end{equation}
 By using the same arguments, we could obtain a similar bound for $\mathcal{N}_k^1(u)$. 
 
 We finish the proof of Proposition \ref{propEE} recalling the definition of the energy in \eqref{Es} and gathering estimates \eqref{propEE3}, \eqref{propEE8} and \eqref{propEE18}. 
\end{proof}

 \subsection{Energy estimates for the differences of two solutions}
In this subsection, we assume that $s \ge 2$. Let $u_1$ and $u_2$ be two solutions to the equation in \eqref{5KdV} with $c_3=0$ in the class \eqref{maintheo1} satisfying $u_1(\cdot,0)=\varphi_1$ and $u_2(\cdot,0)=\varphi_2$.  Then by setting $v=u_1-u_2$, we see that $v$ must satisfy
 \begin{equation} \label{diff5KdV}
\partial_tv=\partial^5_x v+2c_1\partial_xu_1\partial_x^2v+2c_1\partial_xv\partial_x^2u_2+c_2\partial_x(u_1\partial_x^2v)+c_2\partial_x(v\partial_x^2u_2),
\end{equation}
with $v(\cdot,0)=\varphi := \varphi_1-\varphi_2.$
As in subsection \ref{EE1}, we introduce the energy $\widetilde{E}^s_T(v)$ associated to \eqref{diff5KdV}. For $k \ge 1$, 
\begin{equation} \label{diffEk} 
\begin{split}
\widetilde{E}_k(v)(t)=&\|P_kv(\cdot,t)\|_{L^2}^2+\widetilde{\alpha}\int_{\mathbb R}\big(u_1P_k\partial_x^{-1}vQ_{k}\partial_x^{-1}v\big)(x,t)dx 
\\ & +\widetilde{\beta}\int_{\mathbb R}\big(u_1P_k\partial_x^{-1}vP_{k}\partial_x^{-1}v\big)(x,t)dx,
\end{split}
\end{equation}
and
\begin{equation} \label{diffEs} 
\widetilde{E}^s_T(u)=\|P_{\le 0}v(\cdot,0)\|_{L^2}^2+\sum_{k\ge 1}2^{2ks} \sup_{t_k \in [-T,T]}\widetilde{E}_k(v)(t_k),
\end{equation}
where $\widetilde{\alpha}$ and $\widetilde{\beta}$ are two real numbers which will be fixed later. As in Lemma \ref{lemmaEE}, we can compare $\widetilde{E}^s_T(v)$ with $\|v\|_{B^s(T)}$ if $\|u_1\|_{L^{\infty}_TH^s_x}$ is small enough.
\begin{lemma} \label{lemmaEEdiff} 
Let $s>\frac12$. Then, there exists $0<\delta_1$ such that 
\begin{equation} \label{lemmaEEdiff1} 
\frac12 \|v\|_{B^s(T)}^2 \le \widetilde{E}^s_T(v) \le \frac32 \|v\|_{B^s(T)}^2,
\end{equation}
for all $v \in B^s(T)$ as soon as $\|u_1\|_{L^{\infty}_TH^s_x} \le \delta_1$.
\end{lemma}

\begin{proposition} \label{propEEdiff}
Assume $T \in (0,1]$ and $s \ge 2$. Then, if $v$ is a solution to \eqref{diff5KdV}, we have that 
\begin{equation} \label{propEEdiff0} 
\begin{split}
&\widetilde{E}^0_T(v) \lesssim (1+\|\varphi_1\|_{H^{\frac12+}})\|\varphi\|_{L^2}^2 \\ & \quad+\big(1+\|u_1\|_{F^{\frac34+}(T)}\big)\big(\|u_1\|_{F^2(T)}
+\|u_2\|_{F^{2}(T)}\big)\big(\|v\|_{B^{0}(T)}^2+\|v\|_{F^{0}(T)}^2\big),
\end{split}
\end{equation}
and
\begin{equation} \label{propEEdiff1} 
\begin{split}
\widetilde{E}^s_T(v) &\lesssim (1+\|\varphi_1\|_{H^{\frac12+}})\|\varphi\|_{H^s}^2+\|v\|_{F^{0}(T)}\|u_2\|_{F^{s+2}(T)}\|v\|_{F^s(T)} \\ & \quad +\big(1+\Gamma_T^s(u_1)+\Gamma_T^s(u_2)\big)\big(\Gamma_T^s(u_1)+\Gamma_T^s(u_2)\big) \Gamma_T^s(v)^2,
\end{split}
\end{equation}
where 
\begin{displaymath} 
\Gamma_T^s(u):=\max \big\{\|u\|_{F^s(T)},\|u\|_{B^s(T)} \big\}.
\end{displaymath}
\end{proposition}

As a Corollary to Lemma \ref{lemmaEEdiff} and Proposition \ref{propEEdiff}, we deduce an \textit{a priori} estimate in $\|\cdot\|_{B^s(T)}$ for the solutions $v$ to the difference equation \eqref{diff5KdV}. 
\begin{corollary} \label{coroEEdiff}
Assume $T \in (0,1]$. Then, there exists $0<\delta_1 \le 1$ such that
\begin{equation} \label{coroEEdiff0} 
\|v\|_{B^0(T)}^2 \lesssim \|\varphi\|_{L^2}^2+\big(1+\|u_1\|_{F^{\frac34+}(T)}\big)\big(\|u_1\|_{F^2(T)}+\|u_2\|_{F^2(T)}\big)\big(\|v\|_{B^{0}(T)}^2+\|v\|_{F^{0}(T)}^2\big),
\end{equation}
and
\begin{equation} \label{coroEEdiff1} 
\begin{split}
\|v\|_{B^s(T)}^2 &\lesssim \|\varphi\|_{H^s}^2 +\|v\|_{F^{0}(T)}\|u_2\|_{F^{s+2}(T)}\|v\|_{F^s(T)} \\ & \quad +\big(1+\Gamma_T^s(u_1)+\Gamma_T^s(u_2)\big)\big(\Gamma_T^s(u_1)+\Gamma_T^s(u_2)\big) \Gamma_T^s(v)^2,
\end{split}
\end{equation}
for all solutions $v$ to \eqref{diff5KdV} with $\|u_1\|_{L^{\infty}_TH^{\frac12+}_x}< \delta_1$.
\end{corollary}

\begin{proof}[Proof of Proposition \ref{propEEdiff}] We argue as in the proof of Proposition \ref{propEE}. First, we choose extensions $\widetilde{v}$, $\widetilde{u}_1$ and $\widetilde{u}_2$ of $v$, $u_1$ and $u_2$ over $\mathbb R^2$ satisfying 
\begin{equation} \label{propEEdiff1} 
\|\widetilde{v}\|_{F^s} \le 2\|v\|_{F^s(T)} \quad \text{and} \quad \|\widetilde{u_i}\|_{F^s} \le 2\|u_i\|_{F^s(T)}, \ i=1,2.
\end{equation}
Then, for any $k \in \mathbb Z_+ \cap [1,+\infty)$ and $t \in [-T,T]$, we differentiate $\widetilde{E}_k(v)$ with respect to $t$ and deduce using \eqref{diff5KdV} that 
\begin{equation} \label{propEEdiff2}
\frac{d}{dt}\widetilde{E}_k(v) = \widetilde{\mathcal{I}}_k(v)+\widetilde{\mathcal{J}}_k(v)+\widetilde{\alpha} \widetilde{\mathcal{L}}_k^1(v)+\widetilde{\alpha} \widetilde{\mathcal{N}}_k^1(v)+\widetilde{\beta} \widetilde{\mathcal{L}}_k^2(v)+\widetilde{\beta} \widetilde{\mathcal{N}}_k^2(v) ,
\end{equation}
where 
\begin{displaymath}
\widetilde{\mathcal{I}}_k(v)=4c_1\int_{\mathbb R}P_kvP_k\big(\partial_xu_1\partial_x^2v\big)dx
+4c_1\int_{\mathbb R}P_kvP_k\big(\partial_xv\partial_x^2u_2\big)dx,
\end{displaymath}
\begin{displaymath} 
\widetilde{\mathcal{J}}_k(v)=2c_2\int_{\mathbb R}P_kvP_k\partial_x\big(u_1\partial_x^2v\big)dx
2c_2\int_{\mathbb R}P_kvP_k\partial_x\big(v\partial_x^2u_2\big)dx,
\end{displaymath}
\begin{displaymath} 
\begin{split}
\widetilde{\mathcal{L}}_k^1(v)&=\int_{\mathbb R}\partial_x^5u_1P_k\partial_x^{-1}vQ_k\partial_x^{-1}vdx+
\int_{\mathbb R}u_1P_k\partial_x^4vQ_k\partial_x^{-1}vdx\\ & \quad +
\int_{\mathbb R}u_1P_k\partial_x^{-1}vQ_k\partial_x^4vdx,
\end{split}
\end{displaymath} 

\begin{displaymath}  
\begin{split}
\widetilde{\mathcal{N}}_k^1&(v)=c_1\int_{\mathbb R}\partial_x\big((\partial_xu_1)^2\big)P_k\partial_x^{-1}vQ_k\partial_x^{-1}vdx+
2c_1\int_{\mathbb R}u_1P_k\partial_x^{-1}\big(\partial_xu_1\partial_x^2v\big)Q_k\partial_x^{-1}vdx\\ & \quad +
2c_1\int_{\mathbb R}u_1P_k\partial_x^{-1}\big(\partial_xv\partial_x^2u_2\big)Q_k\partial_x^{-1}vdx+
2c_1\int_{\mathbb R}u_1P_k\partial_x^{-1}vQ_k\partial_x^{-1}\big(\partial_xu_1\partial_x^2v\big)dx \\& \quad+
2c_1\int_{\mathbb R}u_1P_k\partial_x^{-1}vQ_k\partial_x^{-1}\big(\partial_xv\partial_x^2u_2\big)dx+c_2\int_{\mathbb R}\partial_x\big(u_1\partial_x^2u_1\big)P_k\partial_x^{-1}vQ_k\partial_x^{-1}vdx \\ &
\quad+c_2\int_{\mathbb R}u_1P_k\big(u_1\partial_x^2v\big)Q_k\partial_x^{-1}vdx+c_2\int_{\mathbb R}u_1P_k\big(v\partial_x^2u_2\big)Q_k\partial_x^{-1}vdx \\ & \quad 
+c_2\int_{\mathbb R}u_1P_k\partial_x^{-1}vQ_k\big(u_1\partial_x^2v\big)dx+c_2\int_{\mathbb R}u_1P_k\partial_x^{-1}vQ_k\big(v\partial_x^2u_2\big)dx,
\end{split}
\end{displaymath}

\begin{displaymath}  
\widetilde{\mathcal{L}}_k^2(v)=\int_{\mathbb R}\partial_x^5u_1P_k\partial_x^{-1}vP_k\partial_x^{-1}vdx+
2\int_{\mathbb R}u_1P_k\partial_x^4vP_k\partial_x^{-1}vdx,
\end{displaymath}
and 
\begin{equation} \label{propEEdiff2b}
\begin{split}
\widetilde{\mathcal{N}}_k^2(v)&=c_1\int_{\mathbb R}\partial_x\big((\partial_xu_1)^2\big)P_k\partial_x^{-1}vP_k\partial_x^{-1}vdx+
4c_1\int_{\mathbb R}u_1P_k\partial_x^{-1}\big(\partial_xu_1\partial_x^2v\big)P_k\partial_x^{-1}vdx\\ & \quad +
4c_1\int_{\mathbb R}u_1P_k\partial_x^{-1}\big(\partial_xv\partial_x^2u_2\big)P_k\partial_x^{-1}vdx+c_2\int_{\mathbb R}\partial_x\big(u_1\partial_x^2u_1\big)P_k\partial_x^{-1}vP_k\partial_x^{-1}vdx  \\ & \quad 
+2c_2\int_{\mathbb R}u_1P_k\big(u_1\partial_x^2v\big)P_k\partial_x^{-1}vdx+2c_2\int_{\mathbb R}u_1P_k\big(v\partial_x^2u_2\big)P_k\partial_x^{-1}vdx.
\end{split}
\end{equation}
Now, we fix $t_k \in [-T,T]$. Without loss of generality, we can assume that $0<t_k \le T$. Therefore, we obtain integrating \eqref{propEEdiff2} between $0$ and $t_k$ that 
\begin{equation} \label{propEEdiff3}
\begin{split}
&\widetilde{E}_k(v)(t_k)- \widetilde{E}_k(v)(0) \\ & \le \Big|\int_{[0,t_k]}\big(\widetilde{\mathcal{I}}_k(v)+\widetilde{\mathcal{J}}_k(v)+\widetilde{\alpha} \widetilde{\mathcal{L}}_k^1(v)+\widetilde{\alpha} \widetilde{\mathcal{N}}_k^1(v)+\widetilde{\beta} \widetilde{\mathcal{L}}_k^2(v)+\widetilde{\beta} \widetilde{\mathcal{N}}_k^2(v)\big)dt\Big|.
\end{split}
\end{equation}

Next we estimate the right-hand side of \eqref{propEEdiff3}. \\

\noindent \textit{Estimates for the cubic terms.} We deduce after some integrations by parts that 
\begin{displaymath} 
\widetilde{\mathcal{L}}_k^1(v)
=5\int_{\mathbb R}\partial_x^3u_1P_kvQ_kvdx-5\int_{\mathbb R}\partial_xu_1P_k\partial_xvQ_k\partial_xvdx.
\end{displaymath}
and
\begin{displaymath} 
\widetilde{\mathcal{L}}_k^2(v)
=5\int_{\mathbb R}\partial_x^3u_1P_kvP_kvdx-5\int_{\mathbb R}\partial_xu_1P_k\partial_xvP_k\partial_xvdx.
\end{displaymath}
We choose $\widetilde{\alpha}=-\frac{2c_2}5$ and $\widetilde{\beta}=\frac{c_2-4c_1}5$. Then it follows, after performing a dyadic decomposition on $v$, that 
\begin{equation} \label{propEEdiff3b}
\Big|\int_{[0,t_k]}\big(\widetilde{\mathcal{I}}_k(v)+\widetilde{\mathcal{J}}_k(v)+
\widetilde{\alpha}\widetilde{\mathcal{L}}_k^1(v)+\widetilde{\beta}\widetilde{\mathcal{L}}_k^2(v)\big)dt\Big| 
\lesssim  \sum_{j=1}^9\widetilde{T}_j(k),
\end{equation}
for each $k \ge 1$, with 
\begin{displaymath} 
\widetilde{T}_1(k)=\sum_{0 \le k_1 \le k-6}\Big|\int_{\mathbb R \times [0,t_k]}\Big(P_kvP_k\big(P_{k_1}\partial_xu_1\partial_x^2v\big) +P_k\partial_xvP_k\partial_xvP_{k_1}\partial_xu_1 \big) dxdt\Big|, 
\end{displaymath}
\begin{displaymath} 
\widetilde{T}_2(k)=\sum_{0 \le k_1 \le k-6}\Big| \int_{\mathbb R \times [0,t_k]}P_k\partial_xv \Big( \big[P_k,P_{k_1}u_1 \big]\partial_x^2v-Q_k\partial_xvP_{k_1}\partial_xu_1\Big) dxdt\Big|,
\end{displaymath}
\begin{displaymath} 
\widetilde{T}_3(k)=\sum_{k_1 \ge k-5, k_2 \ge 0}\Big| \int_{\mathbb R \times [0,t_k]}P_k^2vP_{k_1}\partial_xu_1P_{k_2}\partial_x^2vdt\Big|,
\end{displaymath}
\begin{displaymath} 
\widetilde{T}_4(k)=\sum_{k_1 \ge k-5, k_2 \ge 0}\Big| \int_{\mathbb R \times [0,t_k]}P_k^2\partial_xvP_{k_1}u_1P_{k_2}\partial_x^2vdt\Big|,
\end{displaymath}
\begin{displaymath} 
\widetilde{T}_5(k)=\sum_{k_1, k_2 \ge 0}\Big| \int_{\mathbb R \times [0,t_k]}P_k^2vP_{k_1}\partial_xvP_{k_2}\partial_x^2u_2dt\Big|,
\end{displaymath}
\begin{displaymath} 
\widetilde{T}_6(k)=\sum_{k_1, k_2 \ge 0}\Big| \int_{\mathbb R \times [0,t_k]}P_k^2\partial_xvP_{k_1}vP_{k_2}\partial_x^2u_2dt\Big|,
\end{displaymath}
\begin{displaymath} 
\widetilde{T}_7(k)=\sum_{k-5\le k_1 \le k+4}\Big| \int_{\mathbb R \times [0,t_k]}P_{k_1}\partial_xu_1P_k\partial_xv\big(Q_k\partial_xv+P_k\partial_xv\big)dxdt\Big|,
\end{displaymath}
\begin{displaymath} 
\widetilde{T}_8(k)=\sum_{k_1 \le k-5}\Big|\int_{\mathbb R \times [0,t_k]}P_{k_1}\partial_x^3u_1P_kv\big(Q_kv+P_kv\big)dx \Big| \end{displaymath}
and
\begin{displaymath} 
\widetilde{T}_9(k)=\sum_{k-4 \le k_1 \le k+4}\Big|\int_{\mathbb R \times [0,t_k]}P_{k_1}\partial_x^3u_1P_kv\big(Q_kv+P_kv\big)dx \Big|.
\end{displaymath}
Clearly, Lemma \ref{tec2EE} and the Cauchy-Schwarz inequality imply that 
\begin{equation} \label{propEEdiff4} 
\begin{split}
\widetilde{T}_1(k)+\widetilde{T}_2(k) & \lesssim \sum_{0 \le k_1 \le k-6}2^{k_1/2}\|P_{k_1}\widetilde{u}_1\|_{F_{k_1}}\sum_{|k-k'| \le3}\|P_{k'}\widetilde{v}\|_{F_{k'}}^2 \\ & \lesssim \|\widetilde{u}_1\|_{F^{\frac12+}}\sum_{|k-k'| \le3}\|P_{k'}\widetilde{v}\|_{F_{k'}}^2 .
\end{split}
\end{equation}
Similarly, we get applying estimate \eqref{tecEE1} if $k_1=0$, and estimate \eqref{tecEE1b} if $k_1>0$, that 
\begin{equation} \label{propEEdiff5} 
\widetilde{T}_8(k)  \lesssim  \|\widetilde{u}_1\|_{F^{\frac12+}}\sum_{|k-k'| \le3}\|P_{k'}\widetilde{v}\|_{F_{k'}}^2 .
\end{equation} 
Now, estimate \eqref{tecEE2} leads to 
\begin{equation} \label{propEEdiff6} 
\widetilde{T}_7(k)+\widetilde{T}_9(k)  \lesssim  \|\widetilde{u}_1\|_{F^{\frac54}}\|P_{k}\widetilde{v}\|_{F_{k}}^2 .
\end{equation}
Arguing exactly as in \eqref{propEE7}, we get that 
\begin{equation} \label{propdiffEE6b}
\begin{split}
\widetilde{T}_3(k)+\widetilde{T}_4(k) &\lesssim \|\widetilde{v}\|_{F^{\frac12+}}\sum_{|k'-k| \le 3}\|P_{k'}\widetilde{u}_1\|_{F_{k'}}\|P_{k'}\widetilde{v}\|_{F_{k'}} \\ & \quad+\|\widetilde{u}_1\|_{F^{\frac54}}\sum_{|k'-k| \le 3}\|P_{k'}\widetilde{v}\|_{F_{k'}}^2 \\ & \quad+2^{-k}\|P_k\widetilde{v}\|_{F_k}\sum_{k_1 \ge k+5}2^{3k_1/2}\|P_{k_1}\widetilde{u}_1\|_{F_{k_1}}\|P_{k_1}\widetilde{v}\|_{F_{k_1}}.
\end{split}
\end{equation}
This implies after taking the suprem of $t_k$ over $[0,T]$ and summing in $k \in \mathbb Z_+ \cap [1,+\infty)$ that 
\begin{equation} \label{propdiffEE7} 
\sum_{k \ge 1} 2^{2ks}\sup_{t_k \in [0,T]}\Big( \widetilde{T}_3(k)+\widetilde{T}_4(k)\Big) \lesssim \|\widetilde{u}_1\|_{F^{\frac54}}\|\widetilde{v}\|_{F^s}^2+\|\widetilde{u}_1\|_{F^{s}}\|\widetilde{v}\|_{F^{\frac12+}}\|\widetilde{v}\|_{F^s},
\end{equation}
whenever $s \ge 1$ and 
\begin{equation} \label{propdiffEE7b} 
\sum_{k \ge 1} \sup_{t_k \in [0,T]}\Big( \widetilde{T}_5(k)+\widetilde{T}_6(k)\Big) \lesssim \|\widetilde{u}_1\|_{F^{2}}\|\widetilde{v}\|_{F^0}^2,
\end{equation}
at the $L^2$-level. Note that to obtain \eqref{propdiffEE7b}, we need to modify the first term on the right-hand side of \eqref{propdiffEE6b} by putting all the derivative on $\|P_{k'}\widetilde{u}_1\|_{F_{k'}}$.

To bound $\widetilde{T}_5(k)$ and $\widetilde{T}_6(k)$, we split the domain of summation over the $\{D_j\}_{j=1}^4$ defined in \eqref{propEE9b}.  For example, we explain how to deal with $\widetilde{T}_6(k)$. We have that
\begin{equation} \label{propdiffEE8}
\widetilde{T}_6(k)=\sum_{j=1}^4\sum_{(k_1, k_2) \in D_j}\Big| \int_{\mathbb R \times [0,t_k]}P_k^2\partial_xvP_{k_1}vP_{k_2}\partial_x^2u_2dt\Big|.
\end{equation}
By using estimates \eqref{tecEE1} when $k_2=0$, \eqref{tecEE1b} when $k_2 \ge 1$ and the Cauchy-Schwarz inequality in $k_2$, we deduce that 
\begin{equation} \label{propdiffEE9}
\begin{split}
\sum_{(k_1, k_2) \in D_1}\Big|& \int_{\mathbb R \times [0,t_k]}P_k^2\partial_xvP_{k_1}vP_{k_2}\partial_x^2u_2dt\Big|
 \\ & \lesssim \sum_{0 \le k_2 \le \max(k,k')-5}2^{k_2/2}\|P_{k_2}\widetilde{u}_2\|_{F_{k_2}}\sum_{|k-k'|\le 3}\|P_{k'}\widetilde{v}\|_{F_{k'}}^2
 \\ & \lesssim \|\widetilde{u}_2\|_{F^{\frac12+}}\sum_{|k-k'|\le 3}\|P_{k'}\widetilde{v}\|_{F_{k'}}^2.
\end{split}
\end{equation}
We treat the summation over $D_2$ similarly. Estimate \eqref{tecEE1} when $k_1=0$, estimate \eqref{tecEE1b} when $k_1 \ge 1$ and the Cauchy-Schwarz inequality in $k_1$ imply that 
\begin{equation}  \label{propdiffEE10}
\begin{split}
\sum_{(k_1, k_2) \in D_2}\Big|& \int_{\mathbb R \times [0,t_k]}P_k^2\partial_xvP_{k_1}vP_{k_2}\partial_x^2u_2dt\Big|
 \\ & \lesssim \sum_{0 \le k_1 \le \max(k,k_2)-5}2^{-k_1}\|P_{k_1}\widetilde{v}\|_{F_{k_1}}\sum_{|k-k_2|\le 3}\|P_{k}\widetilde{v}\|_{F_{k}}2^{2k_2}\|\widetilde{u}_2\|_{F_{k_2}}
 \\ & \lesssim \|\widetilde{v}\|_{F^0}\sum_{|k-k_2|\le 3}\|P_{k}\widetilde{v}\|_{F_{k}}2^{2k_2}\|\widetilde{u}_2\|_{F_{k_2}}.
 \end{split}
\end{equation}
Estimate \eqref{tecEE2} gives that 
\begin{equation}  \label{propdiffEE11}
\sum_{(k_1, k_2) \in D_3}\Big| \int_{\mathbb R \times [0,t_k]}P_k^2\partial_xvP_{k_1}vP_{k_2}\partial_x^2u_2dt\Big|
 \lesssim \|\widetilde{u}_2\|_{F^{\frac54}}\sum_{|k-k'|\le 3}\|P_{k'}\widetilde{v}\|_{F_{k'}}^2.
\end{equation}
Finally, it follows from estimate \eqref{tecEE1b} that 
\begin{equation}  \label{propdiffEE12}
\begin{split}
\sum_{(k_1, k_2) \in D_4}\Big| \int_{\mathbb R \times [0,t_k]}&P_k^2\partial_xvP_{k_1}vP_{k_2}\partial_x^2u_2dt\Big|
  \\ & \lesssim \|P_k\widetilde{v}\|_{F_k}\sum_{k_2\ge k+5}2^{k_2/2}\|P_{k_2}\widetilde{v}\|_{F_{k_2}}
 \|P_{k_2}\widetilde{u}_2\|_{F_{k_2}}.
 \end{split}
\end{equation}
Thus, we conclude gathering \eqref{propdiffEE8}--\eqref{propdiffEE12}, taking the supreme over $t_k \in [0,T]$, summing in $k \in \mathbb Z_+ \cap [1,+\infty)$ that
\begin{equation} \label{propdiffEE13} 
\sum_{k \ge 1} 2^{2ks}\sup_{t_k \in [0,T]}\Big( \widetilde{T}_5(k)+\widetilde{T}_6(k)\Big) \lesssim \|\widetilde{u}_2\|_{F^{\frac54}}\|\widetilde{v}\|_{F^s}^2+\|\widetilde{u}_2\|_{F^{s+2}}\|\widetilde{v}\|_{F^0}\|\widetilde{v}\|_{F^s},
\end{equation}
whenever $s \ge 1$ and 
\begin{equation} \label{propdiffEE14} 
\sum_{k \ge 1} \sup_{t_k \in [0,T]}\Big( \widetilde{T}_5(k)+\widetilde{T}_6(k)\Big) \lesssim \|\widetilde{u}_2\|_{F^{2}}\|\widetilde{v}\|_{F^0}^2,
\end{equation}
at the $L^2$-level.

Therefore, we deduce gathering \eqref{propEEdiff4}--\eqref{propdiffEE7} and \eqref{propdiffEE13} that 
\begin{equation} \label{propdiffEE15} 
\begin{split}
\sum_{k \ge 1}&2^{2ks} \sup_{t_k \in [0,T]}\Big|\int_{[0,t_k]}\big(\widetilde{\mathcal{I}}_k(v)+\widetilde{\mathcal{J}}_k(v)+\widetilde{\alpha}\widetilde{\mathcal{L}}_k^1(v)+\widetilde{\beta}\widetilde{\mathcal{L}}_k^2(v)\big)dt\Big| 
\\ 
& \lesssim \big(\|u_1\|_{F^{\frac54}(T)}+\|u_2\|_{F^{\frac54}(T)}\big)\|v\|_{F^s(T)}^2 
+\|u_1\|_{F^s(T)}\|v\|_{F^{\frac12+}(T)}\|v\|_{F^{s}(T)}  \\ 
& \quad +\|u_2\|_{F^{s+2}(T)}\|v\|_{F^{0}(T)}\|v\|_{F^s(T)},
\end{split}
\end{equation}
if $s \ge 1$, whereas 
\begin{equation} \label{propdiffEE16} 
\begin{split}
\sum_{k \ge 1}\sup_{t_k \in [0,T]}\Big|\int_{[0,t_k]}\big(\widetilde{\mathcal{I}}_k(v)+&\widetilde{\mathcal{J}}_k(v)+\widetilde{\alpha}\widetilde{\mathcal{L}}_k^1(v)+\widetilde{\beta}\widetilde{\mathcal{L}}_k^2(v)\big)dt\Big| 
\\ 
& \lesssim \big(\|u_1\|_{F^2(T)}+\|u_2\|_{F^2(T)}\big)\|v\|_{F^0(T)}^2 ,
\end{split}
\end{equation}
at the $L^2$ level. \\

\noindent \textit{Estimates for the fourth order terms.} 
We estimate the fourth order term corresponding to $ \widetilde{\mathcal{N}}_k^2(v) $. After a few integration by parts in \eqref{propEEdiff2b}, we get that 
\begin{equation} \label{propdiffEE17}
\Big|\int_{[0,t_k]}\widetilde{\mathcal{N}}_k^2(v) dt\Big|  \lesssim \sum_{i=1}^5\widetilde{X}_i(k).
\end{equation}
for each $k \ge 1$, whith 
\begin{displaymath}
\begin{split}
\widetilde{X}_1(k)&=\Big|\int_{\mathbb R \times [0,t_k]}\big(\partial_xu_1 \big)^2P_k\partial_x^{-1}vP_kv dxdt\Big| + \Big|\int_{\mathbb R \times [0,t_k]}u_1\partial_xu_1 P_kvP_kv dxdt\Big|  \\ 
& \quad +\Big|\int_{\mathbb R \times [0,t_k]}u_1\partial_xu_1 P_k\partial_xvP_k\partial_x^{-1}v dxdt\Big| ,
\end{split}
\end{displaymath}
\begin{displaymath} 
\widetilde{X}_2(k)= \Big|\int_{\mathbb R \times [0,t_k]}u_1P_k\partial_x^{-1}\big(\partial_xv\partial_x^2u_2\big)P_k\partial_x^{-1}v dxdt\Big| ,
\end{displaymath}
\begin{displaymath}
\widetilde{X}_3(k)= \Big|\int_{\mathbb R \times [0,t_k]}u_1P_k\partial_x^{-1}\big(\partial_xu_1\partial_x^2v\big)P_k\partial_x^{-1}v dxdt\Big| ,
\end{displaymath}
\begin{displaymath}
\widetilde{X}_4(k)= \Big|\int_{\mathbb R \times [0,t_k]}u_1P_k\big(v\partial_x^2u_2\big)P_k\partial_x^{-1}v dxdt\Big| .
\end{displaymath} 
and
\begin{displaymath}
\widetilde{X}_5(k)= \Big|\int_{\mathbb R \times [0,t_k]}u_1P_k\big(u_1\partial_x^2v\big)P_k\partial_x^{-1}v dxdt\Big| .
\end{displaymath}

We use the Strichartz estimate \eqref{Bstrichartz1} with $\alpha=2$, estimate \eqref{lemma1.1} and H\"older's inequality to deduce that 
\begin{equation} \label{propEEdiff18} 
\begin{split}
\sum_{k \ge 1}2^{2ks} \sup_{t_k \in [0,T]} \widetilde{X}_1(k)  & \lesssim \big( \|u_1\|_{L^2_TL^{\infty}_x}+\|\partial_xu_1\|_{L^2_TL^{\infty}_x}\big) \|\partial_xu_1\|_{L^2_TL^{\infty}_x}\sum_{k \ge 1}2^{2ks}\|P_kv\|_{L^{\infty}_TL^2_x}^2 \\
& \lesssim \big(\|u_1\|_{F^{\frac12+}(T)}+\|u_1\|_{F^{\frac34+}(T)})\|u_1\|_{F^{\frac34+}(T)}\|v\|_{B^s(T)}^2,
\end{split}
\end{equation}
for any $s \ge 0$.

To handle $\widetilde{X}_2(k)$, we perform the following decomposition
 \begin{equation} \label{propEEdiff19}
 \begin{split}
\widetilde{X}_2(k) & \lesssim
\sum_{j=1}^4\sum_{(k_1, k_2, k_3) \in F_j} 
\Big| \int_{\mathbb R \times [0,t_k]}P_{k_1}u_1P_k\partial_x^{-1}\big(P_{k_2}\partial_xv\partial_xP_{k_3}\partial_x^2u_2\big)P_k\partial_x^{-1}vdxdt \Big|\\
& =: \sum_{j=0}^4\widetilde{X}_{2,j}(k),
\end{split}
\end{equation} 
where 
\begin{displaymath} 
\begin{array}{l}
F_1= \big\{ (k_1,k_2,k_3) \in \mathbb Z_+^3 \ : \ 0 \le k_1 \le k+3, \ |k-k_3|\le 3, \  0 \le k_2 \le \max(k,k_3)-5 \big\},  \\ 
F_2= \big\{ (k_1,k_2,k_3) \in \mathbb Z_+^3 \ : \ 0 \le k_1 \le k+3, \ |k-k_2|\le 3, \  0 \le k_3 \le \max(k,k_2)-5 \big\},  \\ 
F_3= \big\{ (k_1,k_2,k_3) \in \mathbb Z_+^3 \ : \ 0 \le k_1 \le k+3, \ |k-k_2|\le 8, \   |k-k_3|\le 8 \big\}, \\
F_4= \big\{ (k_1,k_2,k_3) \in \mathbb Z_+^3 \ : \ 0 \le k_1 \le k+3, \ |k_2-k_3|\le 3, \   1 \le k \le \max(k_2,k_3)-8 \big\}.
\end{array}
\end{displaymath}
By applying H\"older's inequality, we can bound $\widetilde{X}_{2,1}(k)$ by 
\begin{equation} \label{propEEdiff20}
\sum_{0 \le k_1 \le k+3}\|P_{k_1}u_1\|_{L^{\infty}_{T,x}} 
\sum_{0 \le k_2 \le \max(k,k_3)-5}2^{k_2}\|P_{k_2}v\|_{L^2_TL^{\infty}_x}\sum_{|k-k_3|\le 3}
\|P_{k_3}u_2\|_{L^{\infty}_TL^2_x}\|P_{k}v\|_{L^{\infty}_TL^2_x},
\end{equation}
which implies after suing the Sobolev embedding, the Cauchy-Schwarz inequality and estimate \eqref{Bstrichartz1b}
\begin{equation} \label{propEEdiff21}
\sum_{k \ge 1}2^{2ks}\sup_{t_k \in [0,T]}\widetilde{X}_{2,1}(k) \lesssim \|u_1\|_{F^{\frac12+}(T)}\|v\|_{F^{\frac34+}(T)}
\|u_2\|_{B^s(T)}\|v\|_{B^s(T)},
\end{equation}
for any $s \ge 0$. On the other by putting the $L^{\infty}_TL^2_x$ norm on $P_{k_2}v$ and the $L^2_TL^{\infty}_x$ norm on $P_{k_3}u_2$ in \eqref{propEEdiff20}, we get that 
\begin{equation} \label{propEEdiff22}
\sum_{k \ge 1}\sup_{t_k \in [0,T]}\widetilde{X}_{2,1}(k) \lesssim \|u_1\|_{F^{\frac12+}(T)}\|u_2\|_{F^{\frac34+}(T)}
\|v\|_{B^0(T)}^2, 
\end{equation}
at the $L^2$ level. By using similar arguments, we get that 
\begin{equation} \label{propEEdiff22b}
\sum_{k \ge 1}2^{2ks}\sup_{t_k \in [0,T]}\big(\widetilde{X}_{2,2}(k)+\widetilde{X}_{2,3}(k)\big)\lesssim \|u_1\|_{F^{\frac12+}(T)}
\|u_2\|_{F^{\frac34+}(T)}\|v\|_{B^s(T)}^2,
\end{equation}
for any $s \ge 0$. Finally, we use estimate \eqref{tec3EE1} to bound $\widetilde{X}_{2,4}(k)$ by
\begin{displaymath} 
\sum_{0 \le k_1 \le k+3}\|P_{k_1}\widetilde{u}_1\|_{F_{k_1}} 2^{-\frac{3k}2}\|P_k\widetilde{v}\|_{F_k}
\sum_{\max(k_2,k_3)\ge k+8}\sum_{|k_2-k_3|\le 3} \|P_{k_2}\widetilde{v}\|_{F_{k_2}}2^{2k_3}\|P_{k_3}\widetilde{u}_2\|_{F_{k_3}},
\end{displaymath}
which implies after summing over $k \in \mathbb Z_+ \cap [1,+\infty)$
\begin{equation} \label{propEEdiff23} 
\sum_{k \ge 1}2^{2ks}\sup_{t_k \in [0,T]}\widetilde{X}_{2,4}(k) \lesssim \|\widetilde{u}_1\|_{F^{0+}}\|\widetilde{u}_2\|_{F^{2}}\|\widetilde{v}\|_{F^{s}}^2,
\end{equation} 
for all $s \ge 0$. Therefore, we conclude gathering estimates \eqref{propEEdiff1} and \eqref{propEEdiff19}--\eqref{propEEdiff23} that
\begin{equation} \label{propEEdiff24} 
\begin{split}
\sum_{k \ge 1}2^{2ks}\sup_{t_k \in [0,T]}\widetilde{X}_{2}(k) & \lesssim \|u_1\|_{F^{\frac12+}(T)}\|u_2\|_{F^{2}(T)}\big(\|v\|_{B^s(T)}^2+\|v\|_{F^{s}(T)}^2\big) \\ & \quad+\|u_1\|_{F^{\frac12+}(T)}\|v\|_{F^{\frac34+}(T)}
\|u_2\|_{B^s(T)}\|v\|_{B^s(T)},
\end{split}
\end{equation}
for any $s \ge 0$ and 
\begin{equation} \label{propEEdiff25} 
\sum_{k \ge 1}\sup_{t_k \in [0,T]}\widetilde{X}_{2}(k)  \lesssim \|u_1\|_{F^{\frac12+}(T)}\|u_2\|_{F^{2}(T)}\big(\|v\|_{B^0(T)}^2+\|v\|_{F^0(T)}^2\big),
\end{equation}
at the $L^2$ level.

By using the same arguments as for $\widetilde{X}_{2}(k) $, we have that 
\begin{equation} \label{propEEdiff26} 
\begin{split}
\sum_{k \ge 1}2^{2ks}\sup_{t_k \in [0,T]}\widetilde{X}_{3}(k) & \lesssim \|u_1\|_{F^{\frac12+}(T)}\|u_1\|_{F^{2}(T)}\big(\|v\|_{B^s(T)}^2+\|v\|_{F^{s}(T)}^2\big) \\ & \quad+\|u_1\|_{F^{\frac12+}(T)}\|v\|_{F^{\frac34+}(T)}
\|u_1\|_{B^s(T)}\|v\|_{B^s(T)},
\end{split}
\end{equation}
for any $s \ge 0$ and 
\begin{equation} \label{propEEdiff27} 
\sum_{k \ge 1}\sup_{t_k \in [0,T]}\widetilde{X}_{3}(k)  \lesssim \|u_1\|_{F^{\frac12+}(T)}\|u_1\|_{F^{2}(T)}\big(\|v\|_{B^0(T)}^2+\|v\|_{F^0(T)}^2\big),
\end{equation}
at the $L^2$ level. 

To deal with $\widetilde{X}_{4}(k)$ at the $L^2$ level, we observe after integrating by parts that
\begin{equation} \label{propEEdiff28}
\begin{split}
\widetilde{X}_4(k) &\le \Big|\int_{\mathbb R \times [0,t_k]}
\partial_xu_1P_k\big(v\partial_xu_2\big)P_k\partial_x^{-1}v dxdt\Big|  \\ & \quad
+\Big|\int_{\mathbb R \times [0,t_k]}
u_1P_k\big(\partial_xv\partial_xu_2\big)P_k\partial_x^{-1}v dxdt\Big|   \\ & \quad
+\Big|\int_{\mathbb R \times [0,t_k]}
u_1P_k\big(v\partial_xu_2\big)P_kv dxdt\Big| \\ & =: \sum_{j=1}^3\widetilde{X}_{4,j}(k).
\end{split}
\end{equation}
Arguing exactly as for $X_2(k)$ in \eqref{propEE10}, we deduce that 
\begin{equation} \label{propEEdiff29} 
\begin{split}
&\sum_{k \ge 1}2^{2ks}\sup_{t_k \in [0,T]}\big(\widetilde{X}_{4,1}(k)+\widetilde{X}_{4,2}(k) \big)
 \\ &  \lesssim \|u_1\|_{F^{\frac34+}(T)}\|u_2\|_{F^2(T)}\|v\|_{B^s(T)}^2
 +\|u_1\|_{F^{\frac34+}(T)}\|v\|_{F^{\frac34+}(T)}
\|u_2\|_{B^s(T)}\|v\|_{B^s(T)},
\end{split}
\end{equation}
for all $s \ge 0$ and 
\begin{equation} \label{propEEdiff30} 
\sum_{k \ge 1}\sup_{t_k \in [0,T]}\widetilde{X}_4(k) \lesssim \|u_1\|_{F^{\frac34+}(T)}\|u_2\|_{F^2(T)}\|v\|_{B^0(T)}^2,
\end{equation}
at the $L^2$ level. To estimate $\widetilde{X}_{4,3}(k)$ at the $H^s$-level, we use the same decomposition as for $X_4(k)$ in \eqref{propEE12}. It follows that 
\begin{displaymath}
\begin{split}
\sum_{k \ge 1}2^{2ks}\sup_{t_k \in [0,T]}\widetilde{X}_{4,3}(k) &\lesssim \|u_1\|_{F^{\frac12+}(T)}\|u_2\|_{F^{\frac34+}(T)}\|v\|_{B^s(T)}^2
\\ & \quad +\|u_2\|_{F^{\frac34+}(T)}\|v\|_{F^{\frac12+}(T)}\|u_1\|_{B^s(T)}\|v\|_{B^s(T)} \\ & \quad 
+\|u_1\|_{F^{\frac12+}(T)}\|v\|_{F^{\frac12+}(T)}\|u_2\|_{F^s(T)}\|v\|_{F^s(T)},
\end{split}
\end{displaymath}
which implies together with \eqref{propEEdiff28} and \eqref{propEEdiff29} 
\begin{equation} \label{propEEdiff31}
\begin{split}
&\sum_{k \ge 1}2^{2ks}\sup_{t_k \in [0,T]}\widetilde{X}_4(k) \\ & 
 \lesssim \|u_1\|_{F^{\frac34+}(T)}\|u_2\|_{F^2(T)}\|v\|_{B^s(T)}^2 
 +\|u_1\|_{F^{\frac12+}(T)}\|v\|_{F^{\frac12+}(T)}\|u_2\|_{F^s(T)}\|v\|_{F^s(T)}\\ &
 \quad +\big(\|u_1\|_{F^{\frac34+}(T)}+\|u_2\|_{F^{\frac34+}(T)}\big)\|v\|_{F^{\frac34+}(T)}
\big(\|u_1\|_{B^s(T)}+\|u_2\|_{B^s(T)}\big)\|v\|_{B^s(T)},
 \end{split}
\end{equation}
for all $s \ge 0$.

Finally, we treat the term $\widetilde{X}_5(k)$. After integrating by parts, we obtain that 
\begin{equation} \label{propEEdiff32} 
\begin{split}
\widetilde{X}_5(k) &\le \Big|\int_{\mathbb R \times [0,t_k]}\partial_xu_1P_k\big(u_1\partial_xv\big)P_k\partial_x^{-1}v dxdt\Big| \\ & \quad +\Big|\int_{\mathbb R \times [0,t_k]}u_1P_k\big(\partial_xu_1\partial_xv\big)P_k\partial_x^{-1}v dxdt\Big|
\\ & \quad +\Big|\int_{\mathbb R \times [0,t_k]}u_1P_k\big(u_1\partial_xv\big)P_kv dxdt\Big| \\ & := \sum_{j=1}^3\widetilde{X}_{5,j}(k).
\end{split}
\end{equation}
By using the same arguments as above, we deduce that 
Arguing exactly as for $X_2(k)$ in \eqref{propEE10}, we deduce that 
\begin{equation} \label{propEEdiff33} 
\begin{split}
&\sum_{k \ge 1}2^{2ks}\sup_{t_k \in [0,T]}\big(\widetilde{X}_{5,1}(k)+\widetilde{X}_{5,2}(k) \big)
 \\ &  \lesssim \|u_1\|_{F^{\frac34+}(T)}\|u_1\|_{F^2(T)}\|v\|_{B^s(T)}^2
 +\|u_1\|_{F^{\frac34+}(T)}\|v\|_{F^{\frac34+}(T)}
\|u_1\|_{B^s(T)}\|v\|_{B^s(T)},
\end{split}
\end{equation}
for all $s \ge 0$ and 
\begin{equation} \label{propEEdiff34} 
\sum_{k \ge 1}\sup_{t_k \in [0,T]}\big(\widetilde{X}_{5,1}(k)+\widetilde{X}_{5,2}(k) \big)
 \lesssim \|u_1\|_{F^{\frac34+}(T)}\|u_1\|_{F^2(T)}\|v\|_{B^0(T)}^2,
\end{equation}
at the $L^2$ level. To handle $\widetilde{X}_{5,3}(k)$, we perform the same decomposition as for $X^4(k)$ in \eqref{propEE12}. It follows that 
\begin{equation} \label{propEEdiff35}
\begin{split}
\sum_{k \ge 1}2^{2ks}\sup_{t_k \in [0,T]}\widetilde{X}_{5,3}(k) &\lesssim \|u_1\|_{F^{\frac12+}(T)}\|v\|_{F^{\frac34+}(T)}\|u_1\|_{B^s(T)}\|v\|_{B^s(T)}
\\ & \quad +\|u_1\|_{F^{\frac12+}(T)}\|u_1\|_{F^{\frac34+}(T)}\big(\|v\|_{F^s(T)}^2 +\|v\|_{B^s(T)}^2\big),
\end{split}
\end{equation}
for any $s \ge 0$ and 
\begin{equation} \label{propEEdiff36}
\begin{split}
\sum_{k \ge 1}\sup_{t_k \in [0,T]}\widetilde{X}_{5,3}(k) &\lesssim \|u_1\|_{F^{\frac12+}(T)}\|u_1\|_{F^{\frac34+}(T)}\big(\|v\|_{F^0(T)}^2 +\|v\|_{B^0(T)}^2\big),
\end{split}
\end{equation}
at the $L^2$ level. Thus, we deduce from \eqref{propEEdiff33} and \eqref{propEEdiff35} that 
\begin{equation} \label{propEEdiff37}
\begin{split}
\sum_{k \ge 1}2^{2ks}\sup_{t_k \in [0,T]}\widetilde{X}_5(k) &\lesssim \|u_1\|_{F^{\frac34+}(T)}\|v\|_{F^{\frac34+}(T)}\|u_1\|_{B^s(T)}\|v\|_{B^s(T)}
\\ & \quad +\|u_1\|_{F^{\frac34+}(T)}\|u_1\|_{F^2(T)}\big(\|v\|_{F^s(T)}^2 +\|v\|_{B^s(T)}^2\big),
\end{split}
\end{equation}
for any $s \ge 0$, and from \eqref{propEEdiff34} and \eqref{propEEdiff36} that 
\begin{equation} \label{propEEdiff38}
\begin{split}
\sum_{k \ge 1}\sup_{t_k \in [0,T]}\widetilde{X}_5(k) &\lesssim \|u_1\|_{F^{\frac34+}(T)}\|u_1\|_{F^2(T)}\big(\|v\|_{F^0(T)}^2 +\|v\|_{B^0(T)}^2\big),
\end{split}
\end{equation}
at the $L^2$ level.

Therefore, we deduce from \eqref{propdiffEE17}, \eqref{propEEdiff18}, \eqref{propEEdiff25}, \eqref{propEEdiff27}, \eqref{propEEdiff30} and \eqref{propEEdiff38} that 
\begin{displaymath} 
\begin{split}
\sum_{k \ge 1}\sup_{t_k \in [0,T]}\Big|&\int_{[0,t_k]}\widetilde{\mathcal{N}}_k^2(v) dt\Big|\\ &\lesssim \|u_1\|_{F^{\frac34+}(T)}\big(\|u_1\|_{F^2(T)}+\|u_2\|_{F^2(T)}\big)\big(\|v\|_{F^0(T)}^2 +\|v\|_{B^0(T)}^2\big),
\end{split}
\end{displaymath}
which together with \eqref{propEEdiff3} and \eqref{propdiffEE16} implies estimate \eqref{propEEdiff0}, since the bound for the term corresponding to $\widetilde{\mathcal{N}}_k^1(v) $ would be similar.

Similarly, we deduce from \eqref{propdiffEE17}, \eqref{propEEdiff18}, \eqref{propEEdiff24}, \eqref{propEEdiff26}, \eqref{propEEdiff31} and \eqref{propEEdiff37} that
\begin{displaymath} 
\begin{split}
&\sum_{k \ge 1}2^{2ks}\sup_{t_k \in [0,T]}\Big|\int_{[0,t_k]}\widetilde{\mathcal{N}}_k^2(v) dt\Big|\\ &\lesssim 
\big(\|u_1\|_{F^{\frac34+}(T)}+\|u_2\|_{F^{\frac34+}(T)}\big)\|v\|_{F^{\frac34+}(T)}\big(\|u_1\|_{B^s(T)}+\|u_2\|_{B^s(T)}\big)\|v\|_{B^s(T)}
\\ & \quad +\|u_1\|_{F^{\frac34+}(T)}\big(\|u_1\|_{F^2(T)}+\|u_2\|_{F^2(T)}\big)\big(\|v\|_{F^s(T)}^2 +\|v\|_{B^s(T)}^2\big) \\ & \quad +\|u_1\|_{F^{\frac12+}(T)}\|v\|_{F^{\frac12+}(T)}\|u_2\|_{F^s(T)}\|v\|_{F^s(T)},
\end{split}
\end{displaymath}
which together with \eqref{propEEdiff3} and \eqref{propdiffEE15} implies estimate \eqref{propEEdiff1}.

This concludes the proof of Proposition \ref{propEEdiff}. 

\end{proof}

\section{Proof of Theorem \ref{maintheo}}

We recall that, for sake of simplicity, we are proving Theorem \ref{maintheo} in the case $c_3=0$. The starting point is a well-posedness result for smooth solutions which follows from Theorem 3.1 in \cite{Po}.
\begin{theorem} \label{smoothsol} 
For all $u_0 \in H^{\infty}(\mathbb R)$, there exist a positive time $T$ and a unique solution $u \in C([-T,T];H^{\infty}(\mathbb R))$ to the initial value problem \eqref{5KdV}. Moreover $T=T(\|u_0\|_{H^4})$ can be chosen as a nonincreasing function of its argument.
\end{theorem}

\subsection{\textit{A priori} estimates for smooth solutions} The main result of this subsection reads as follows. 

\begin{proposition} \label{apriori} 
Assume $s \ge \frac54$. For any $M>0$, there exists a positive time $T=T(M)$ such that for any initial data $u_0 \in H^{\infty}(\mathbb R)$satisfying $\|u_0\|_{H^s} \le M$, the solution $u$ obtained in Theorem \ref{smoothsol} is defined on $[-T,T]$ and satisfies 
\begin{equation} \label{apriori1}
u \in C([-T,T];H^{\infty}(\mathbb R)) \quad \text{and} \quad \|u\|_{L^{\infty}_TH^s_x} \lesssim \|u_0\|_{H^s}.
\end{equation}
\end{proposition}

The following technical lemma will be needed in the proof of Proposition \ref{apriori}. 
\begin{lemma} \label{apriorilemma} 
Assume $s \in \mathbb R_+$, $T>0$ and $u \in C([-T,T];H^{\infty}(\mathbb R))$. We define 
\begin{equation} \label{apriorilemma1b} 
\Lambda_{T'}^s(u):= \max\big\{\|u\|_{B^s(T')},\big\|\partial_x\big(u\partial_x^2u\big)\big\|_{N^s(T')}, \big\|\partial_xu\partial_x^2u\big\|_{N^s(T')} \big\},
\end{equation}
for any $0 \le T' \le T$. Then $:T'\mapsto \Lambda^s_{T'}(u)$ is nondecreasing and continuous on $[0,T)$. Moreover 
\begin{equation} \label{apriorilemma2} 
\lim_{T' \to 0} \Lambda_{T'}^s(u) \lesssim \|u(0)\|_{H^s}.
\end{equation}
\end{lemma}

\begin{proof} It is clear from the definition of $B^s(T')$ and the fact that $u \in C([-T,T];H^{\infty}(\mathbb R))$ that $:T' \mapsto \|u\|_{B^s(T')}$ is nondecreasing and continuous on $[0,T]$ and that 
\begin{equation} \label{apriorilemma3}  
\lim_{T' \to 0} \|u\|_{B^s(T')} \lesssim \|u(0)\|_{H^s}.
\end{equation}

In order to deal with the other components of $\Lambda^s_{T'}(u)$ in \eqref{apriorilemma1b}, it suffices to prove that given $f \in C([-T,T];H^{\infty}(\mathbb R))$,
\begin{equation} \label{apriorilemma4} 
: T' \in [0,T) \mapsto \|f\|_{N^s(T')} \quad \text{is continuous and nondecreasing},
\end{equation} 
and
\begin{equation} \label{apriorilemma5} 
\lim_{T' \to 0} \|f\|_{N^s(T')}=0
\end{equation}

It is clear from the definition of $N^s$ that 
\begin{equation} \label{apriorilemma6} 
\|\widetilde{f}\|_{N^s} \lesssim \|\widetilde{f}\|_{L^2_tH^s_x} ,
\end{equation}
for any $\widetilde{f} \in L^2_tH^s_x$.  Then, we deduce by applying estimate \eqref{apriorilemma6} to $\widetilde{f}(x,t)=\chi_{[-T',T']}(t)f(x,t)$ that
\begin{equation} \label{apriorilemma7} 
\|f\|_{N^s(T')} \le \|\widetilde{f}\|_{N^s} \lesssim \|\widetilde{f}\|_{L^2_tH^s_x}   \lesssim (T')^{1/2}\|f\|_{L^{\infty}_TH^s_x} \underset{T'\to 0}{\longrightarrow} 0,
\end{equation}
which gives \eqref{apriorilemma5}.

Now, we turn to the proof of \eqref{apriorilemma4}. The fact that $:T' \in [0,T) \mapsto \|f\|_{N^s(T')}$ is a nondecreasing function  follows directly from the definition of $N^s(T')$. To prove the continuity of $:T' \in (0,T) \mapsto   \|f\|_{N^s(T')}$ at some fixed time $T_0'\in (0,T)$, we introduce the scaling operator $D_r(f)(x,t):=f(x/r^{\frac15},t/r)$, for $r$ close enough to $1$. Hence, we have from \eqref{apriorilemma7} and the triangle inequality that 
\begin{equation} \label{apriorilemma8} 
\begin{split}
\Big| \|f\|_{N^s(T')}-\|D_{T'/T_0'}(f)\|_{N^s(T')}  \Big| &\le \|f-D_{T'/T_0'}(f)\|_{N^s(T')} \\ 
&\lesssim (T')^{\frac12}\|f-D_{T'/T_0'}(f)\|_{L^{\infty}_{T}H^s_x}  \underset{T'\to T_0'}{\rightarrow} 0,
\end{split}
\end{equation} 
since $f \in C([-T,T];H^{\infty}(\mathbb R))$.
Then, it remains to show that 
\begin{equation} \label{apriorilemma9} 
\lim_{r \to 1}\|D_r(f)\|_{N^s(rT'_0)}=\|f\|_{N^s(T_0')},
\end{equation}
to conclude the proof of \eqref{apriorilemma4}.  We observe that \eqref{apriorilemma9} would follow from the inequalities 
\begin{equation} \label{apriorilemma10} 
\|f\|_{N^s(T_0')} \le \liminf_{r \to 1}\|D_{r}(f)\|_{N^s(rT'_0)},
\end{equation}
and 
\begin{equation} \label{apriorilemma11}
\limsup_{r \to 1}\|D_{r}(f)\|_{N^s(rT'_0)} \le \|f\|_{N^s(T_0')}.
\end{equation} 

First, we begin with the proof of \eqref{apriorilemma10}. Let $\epsilon$ be an arbitrarily small positive number. For $r$ close to $1$, we choose an extension $\widetilde{f}_r$ of $D_r(f)$ outside of $[-rT_0',rT_0']$ satisfying
\begin{equation}  \label{apriorilemma12}
\widetilde{f}_{r|_{[-rT_0',rT_0']}}=D_{r}(f) \quad \text{and} \quad \|\widetilde{f}_r\|_{N^s} \le \|D_{r}(f)\|_{N^s(rT'_0)}+\epsilon.
\end{equation}
Note that, since $f \in C([-T,T];H^{\infty}(\mathbb R))$, we have $\|D_r(f)\|_{N^s(rT_0')} \le M$, where $M$ is a positive constant independent of $r$.
We also observe that $D_{1/r}(\widetilde{f}_r)$ is an extension of $f$ outside of $[-T_0',T_0']$, so that 
\begin{equation} \label{apriorilemma13}
\|f\|_{N^s(T_0')} \le \|D_{1/r}(\widetilde{f}_r)\|_{N^s}.
\end{equation}
Moreover, we will prove that 
\begin{equation} \label{apriorilemma14} 
\|D_{1/r}(\widetilde{f}_r)\|_{N^s} \le \psi(r)\|\widetilde{f}_r\|_{N^s}, 
\end{equation}
where $\psi$ is a continuous function defined in a neighborhood of $1$ and satisfying $\lim_{r\to 1}\psi(r)=1$. Then estimate \eqref{apriorilemma10} would be deduced gathering estimates \eqref{apriorilemma12}, \eqref{apriorilemma13} and \eqref{apriorilemma14}. 

To prove estimate \eqref{apriorilemma14}, we first fix $k \in \mathbb Z_+$. Then, by definition, 
\begin{equation} \label{apriorilemma15}
\|P_kD_{1/r}(\widetilde{f}_r) \|_{N_k}=\sup_{\widetilde{t} \in \mathbb R}\big\| 
(\tau-w(\xi)+i2^{2k})^{-1}\mathcal{F}\big[\eta_0(2^{2k}(\cdot-\widetilde{t}))P_kD_{1/r}(\widetilde{f}_r) \big]\big\|_{X_k}.
\end{equation}
We observe that 
\begin{displaymath} 
\eta_0(2^{2k}(\cdot-\widetilde{t}))D_{1/r}(\widetilde{f}_r)=D_{1/r}\big(\eta_0^r(2^{2k}(\cdot-r\widetilde{t}))\widetilde{f}_r\big),
\end{displaymath}
where $\eta_0^r(t)=\eta_0(r^{-1}t)$. Hence, 
\begin{displaymath} 
\mathcal{F}\big[\eta_0(2^{2k}(\cdot-\widetilde{t}))P_kD_{1/r}(\widetilde{f}_r) \big](\xi,\tau)=r^{-\frac65}\eta_k(\xi)\mathcal{F}\big[\eta_0^r(2^{2k}(\cdot-r\widetilde{t}))\widetilde{f}_r\big](\xi/r^{\frac15},\tau/r),
\end{displaymath}
so, we deduce from the definition of $X_k$ in \eqref{Xk} that the right-hand side of \eqref{apriorilemma15} is equal to 
\begin{equation} \label{apriorilemma16} 
r^{-3/5}\sup_{\widetilde{t} \in \mathbb R}\sum_{j \ge 0}2^{j/2}\Big\| 
\frac{\eta_j\big(r(\tau-w(\xi))\big)}{\big|r(\tau-w(\xi))+i2^{2k}\big|}\eta_k(r^{\frac15}\xi)\mathcal{F}\big[\eta_0^r(2^{2k}(\cdot-\widetilde{t}))\widetilde{f}_r\big]\Big\|_{L^2_{\xi,\tau}}.
\end{equation}
Moreover, we use that 
\begin{displaymath} 
\Big|\frac{a^2+2^{4k}}{r^2a^2+2^{4k}}-1 \Big|=|1-r^2|\frac{a^2}{r^2a^2+2^{4k}} \le 4|1-r^2|,
\end{displaymath}
for any $a \in \mathbb R_+,$ $k \in \mathbb Z_+$ and $r \in [1/2,2]$, to bound the $L^2$ norm corresponding to a fixed $j \in \mathbb Z_+$ in \eqref{apriorilemma16} by
\begin{equation} \label{apriorilemma17}
\big(1+4|1-r^2|\big)^{1/2}\Big\| 
\frac{\eta_j\big(r(\tau-w(\xi))\big)}{\big|\tau-w(\xi)+i2^{2k}\big|}\eta_k(r^{\frac15}\xi)\mathcal{F}\big[\eta_0^r(2^{2k}(\cdot-\widetilde{t}))\widetilde{f}_r\big]\Big\|_{L^2_{\xi,\tau}}.
\end{equation}
Now, the mean value theorem gives that 
\begin{equation} \label{apriorilemma18}
\begin{split}
\big|\eta_j(r(\tau-w(\xi)))-\eta_j(\tau-w(\xi)) \big|& \lesssim |r-1| \sup_{s \in [1,r]} 2^j\big|\eta_j'(s(\tau-w(\xi)))\big| \\
& \lesssim |r-1| \sum_{|j'-j| \le 1, j'\in \mathbb Z_+}\eta_j(\tau-w(\xi)),
\end{split}
\end{equation}
if $r \in [3/4,5/4]$. Therefore, we deduce after gathering \eqref{apriorilemma15}--\eqref{apriorilemma18} and summing over $j \in \mathbb Z_+$ that 
\begin{equation} \label{apriorilemma19} 
\begin{split}
\|P_kD_{1/r}&(\widetilde{f}_r) \|_{N_k} \\ & \le \varphi(r) \sup_{\widetilde{t} \in \mathbb R}
\big\|(\tau-w(\xi)+i2^{2k})^{-1}\eta_k(r^{\frac15}\xi)\mathcal{F}\big[\eta_0^r(2^{2k}(\cdot-\widetilde{t}))\widetilde{f}_r\big]\big\|_{X_k},
\end{split}
\end{equation}
where $\varphi$ is a continuous function defined in a neighborhood of $1$  and satisfying $\lim_{r \to 1}\varphi(r)=1$.  In order to deal with $\eta_0^r$ appearing on the right-hand side of \eqref{apriorilemma19}, we get from the fundamental theorem of calculus that
\begin{equation} \label{apriorilemma19b}
\eta_0^r(2^{2k}(t-\widetilde{t}))-\eta_0(2^{2k}(t-\widetilde{t}))=\int_1^{r^{-1}}\gamma_s(2^{2k}(t-\widetilde{t}))ds,
\end{equation}
where $\gamma_s(t)=t\eta_0'(st)$. Moreover, we use that 
\begin{displaymath}
\eta_0\big(2^{2k}(t-(\widetilde{t}+2\cdot2^{-2k}))\big)+\eta_0\big(2^{2k}(t-(\widetilde{t}-2\cdot2^{-2k}))\big)  =1
\end{displaymath}
 on the support of the integral on the right-hand side of \eqref{apriorilemma19b}.  Hence, it follows from Minkowski's inequality and estimates \eqref{coro1b.1} and \eqref{apriorilemma19} that 
\begin{equation} \label{apriorilemma20} 
\begin{split}
\|P_kD_{1/r}&(\widetilde{f}_r) \|_{N_k} \\ & \le \widetilde{\varphi}(r) \sup_{\widetilde{t} \in \mathbb R}
\big\|(\tau-w(\xi)+i2^{2k})^{-1}\eta_k(r^{\frac15}\xi)\mathcal{F}\big[\eta_0(2^{2k}(\cdot-\widetilde{t}))\widetilde{f}_r\big]\big\|_{X_k},
\end{split}
\end{equation}
where $ \widetilde{\varphi}$ is a function possessing the same properties as $\varphi$. Moreover, we observe arguing as in \eqref{apriorilemma18} that 
\begin{equation} \label{apriorilemma21}
\big|\eta_k(r^{\frac15}\xi))-\eta_k(\xi) \big| \lesssim  |r^{\frac15}-1| \sum_{|k'-k| \le 1, k'\in \mathbb Z_+}\eta_k(\xi).
\end{equation}
Thus, we deduce gathering estimates \eqref{apriorilemma20} and \eqref{apriorilemma21} that
\begin{equation} \label{apriorilemma22} 
\begin{split}
&\|P_kD_{1/r}(\widetilde{f}_r) \|_{N_k} \\ & \le \widetilde{\varphi}(r)\big(1+|r^{\frac15}-1|\big)\|P_k\widetilde{f}_r\|_{N_k} \\ 
& \ +\widetilde{\varphi}(r)|r^{\frac15}-1| \sup_{\widetilde{t} \in \mathbb R}
\big\|(\tau-w(\xi)+i2^{2k})^{-1}\eta_{k-1}(\xi)\mathcal{F}\big[\eta_0(2^{2k}(\cdot-\widetilde{t}))\widetilde{f}_r\big]\big\|_{X_{k-1}} 
\\ 
& \ +\widetilde{\varphi}(r)|r^{\frac15}-1| \sup_{\widetilde{t} \in \mathbb R}
\big\|(\tau-w(\xi)+i2^{2k})^{-1}\eta_{k+1}(\xi)\mathcal{F}\big[\eta_0(2^{2k}(\cdot-\widetilde{t}))\widetilde{f}_r\big]\big\|_{X_{k+1}} .
\end{split} 
\end{equation}
To deal with the second term on the right-hand side of \eqref{apriorilemma22}, we notice that $\big|\tau-w(\xi) +i2^{2k}\big|^{-1} \le \big|\tau-w(\xi) +i2^{2(k-1)}\big|^{-1}$ and $\eta_0(2^{2(k-1)}(\cdot-\widetilde{t}))=1$ on the support of $\eta_0(2^{2k}(\cdot-\widetilde{t}))$. Then, it follows from estimate \eqref{coro1b.1} that 
\begin{equation} \label{apriorilemma23} 
 \sup_{\widetilde{t} \in \mathbb R}
\big\|(\tau-w(\xi)+i2^{2k})^{-1}\eta_{k-1}(\xi)\mathcal{F}\big[\eta_0(2^{2k}(\cdot-\widetilde{t}))\widetilde{f}_r\big]\big\|_{X_{k-1}} \lesssim \|P_{k-1}\widetilde{f}_r\|_{N_{k-1}}.
\end{equation}
On the other hand, we have that $\big|\tau-w(\xi) +i2^{2k}\big|^{-1} \le 4\big|\tau-w(\xi) +i2^{2(k+1)}\big|^{-1}$. Moreover, we observe that 
\begin{displaymath}
2^{2(k+1)}\int_{-\tilde{t}-2^4\cdot2^{-2(k+1)}}^{-\tilde{t}+2^4\cdot2^{-2(k+1)}}\eta_0\big(2^{2(k+1)}(t+s)\big)ds=\int \eta_0(s)ds>0,
\end{displaymath}
if $t$ lies in the support of $\eta_0(2^{2k}(\cdot-\tilde{t}))$. Therefore, we deduce from Minkowski's inequality  and by using estimate \eqref{coro1b.1} that 
\begin{equation} \label{apriorilemma25} 
 \sup_{\widetilde{t} \in \mathbb R}
\big\|(\tau-w(\xi)+i2^{2k})^{-1}\eta_{k+1}(\xi)\mathcal{F}\big[\eta_0(2^{2k}(\cdot-\widetilde{t}))\widetilde{f}_r\big]\big\|_{X_{k+1}} \lesssim \|P_{k+1}\widetilde{f}_r\|_{N_{k+1}}.
\end{equation}
We conclude the proof of estimate \eqref{apriorilemma14} by gathering estimates \eqref{apriorilemma22}, \eqref{apriorilemma23}, \eqref{apriorilemma25}, summing over $k \in \mathbb Z_+$ and recalling the definition of $N^s$ in \eqref{Ns}. 

The proof of estimate \eqref{apriorilemma11} follows in a similar way (it is actually easier).
\end{proof}

\begin{proof}[Proof of Proposition \ref{apriori}] Fix $s \ge \frac54$. First, it is worth noticing that we can always assume that the initial data $u_0$ have small $H^s$-norm by using a scaling argument.  

Indeed, if $u$ is a solution to the IVP \eqref{5KdV} on the time interval $[0,T]$, then $u_{\lambda}(x,t)=\lambda^2u(\lambda x,\lambda^5 t)$ is also a solution to the equation in \eqref{5KdV} with initial data $u_{\lambda}(\cdot,0)=\lambda^2u_0(\lambda \cdot)$ on the time interval $[0,\lambda^{-5}T]$. For $\epsilon>0$, let us denote by $\mathcal{B}^s(\epsilon)$ the ball of $H^s(\mathbb R)$ centered at the origin with radius $\epsilon$. Since 
\begin{displaymath} 
\|u_{\lambda}(\cdot,0)\|_{H^s} \lesssim \lambda^{\frac32}(1+\lambda^s)\|u_0\|_{H^s},
\end{displaymath}
we can always force $u_{\lambda}(\cdot,0)$ to belong to $\mathcal{B}^s(\epsilon)$ by choosing $\lambda \sim\epsilon^{\frac23}\|u_0\|_{H^s}^{-\frac23} $. Therefore, it is enough to prove that if $u_0 \in \mathcal{B}^s(\epsilon)$, then Proposition \ref{apriori} holds with $T=1$. This would imply that Proposition \ref{apriori} holds for arbitrarily large initial data in $H^s(\mathbb R)$ with a time  $T\sim \lambda^5 \sim \|u_0\|_{H^s}^{-\frac{10}3}$. 

Now, fix $u_0 \in H^{\infty}(\mathbb R) \cap \mathcal{B}^s(\epsilon)$ and let $u \in C([-T,T];H^{\infty})$ the solution to \eqref{5KdV} given by Theorem \ref{smoothsol} where $0<T\le 1$. We obtain gathering the linear estimate \eqref{linear1}, the bilinear estimates \eqref{bilinE1}--\eqref{bilinE2} and the energy estimate \eqref{coroEE1} that
\begin{equation} \label{apriori2}
\Lambda^{\sigma}_T(u)^2 \lesssim \|u_0\|_{H^{\sigma}}^2+\big(\Lambda^s_T(u)+\Lambda^s_T(u)^2\big)\Lambda^{\sigma}_T(u)^2,
\end{equation}
for any $\sigma \ge s$ as soon as $\Lambda^s_T(u) < \widetilde{\delta}_0$.  Here, $\widetilde{\delta}_0$ is a small positive number chosen\footnote{the choice is possible by using estimates \eqref{lemma1.1} and \eqref{linear1}.} such that $\|u\|_{L^{\infty}_TH^s_x}<\delta_0$ as soon as $\Lambda^s_T(u)< \widetilde{\delta}_0$, where $\delta_0$ is given by Corollary \ref{coroEE}. Estimates \eqref{linear1}, \eqref{apriori2} with $\sigma=s$, Lemma \ref{apriorilemma} and a continuity argument ensure the existence of $\epsilon_s>0$ and $C_s>0$ such that $ \Gamma^s_T(u) \le C_s \epsilon$ provided $\|u_0\|_{H^s} \le \epsilon \le \epsilon_s$ where $\Gamma^s_T(u)$ is defined by 
\begin{equation} \label{apriorilemma1}
\Gamma_{T}^s(u):= \max\big\{\|u\|_{B^s(T)},\|u\|_{F^s(T)} \big\}.
\end{equation} 
Thus, estimates \eqref{lemma1.1}, \eqref{linear1} and \eqref{apriori2} yield
\begin{equation} \label{apriori3} 
\|u\|_{L^{\infty}_TH^{\sigma}_x} \lesssim \Gamma^{\sigma}_T(u) \lesssim \|u_0\|_{H^{\sigma}},
\end{equation}
for all $\sigma \ge s$, provided $\|u_0\|_{H^s} \le \epsilon \le \epsilon_s$.

Therefore, using estimate \eqref{apriori3} with $\sigma=4$ we can reapply the result of Theorem \ref{smoothsol} a finite number of times and extend the solution $u$ on the time interval $[-1,1]$. This concludes the proof of Proposition \ref{apriori}.
\end{proof}

\subsection{$L^2$- Lipschitz bound for the difference of two solutions and uniqueness} 

Let $u_1$ and $u_2$ be two solutions of the equation in \eqref{5KdV} define on a time interval $[-T,T]$ for some $0<T\le 1$ and with respective initial data $u_1(\cdot,0)=\varphi_1$ and $u_2(\cdot,0)=\varphi_2$. We also assume that $\varphi_1, \ \varphi_2 \in \mathcal{B}^2(\epsilon)$ and 
\begin{equation} \label{maintheo3} 
\Gamma_T^2(u_i) \le C_2 \epsilon \le C_2 \epsilon_2, \quad \text{for} \ i=1,2, 
\end{equation}
where $\Gamma^2_T(\cdot)$ is defined in \eqref{apriorilemma1}. Moreover, according to \eqref{lemma1.1}, we can choose $\epsilon$ small enough such that $\|u_i\|_{L^{\infty}_TH^2_x} < \delta_1$ where $\delta_1$ is given in Corollary \ref{coroEEdiff}.

Let define $v$ by $v=u_1-u_2$. Observe that $v$ is a solution to equation \eqref{diff5KdV} and also to
\begin{displaymath} 
\partial_tv=\partial^5_x v+c_1\partial_x\big(\partial_x(u_1+u_2)\partial_xv\big)+c_2\partial_x(u_1\partial_x^2v)+c_2\partial_x(v\partial_x^2u_2).
\end{displaymath}
Then, we conclude gathering estimates \eqref{linear1}, \eqref{bilinL2E1}, \eqref{bilinL2E2} and \eqref{coroEEdiff0} that there exists $0<\widetilde{\epsilon}_2 \le \epsilon_2$ such that 
\begin{equation} \label{maintheo4} 
\Gamma^0_T(v) \lesssim \|\varphi_1-\varphi_2\|_{L^2},
\end{equation}
provided $u_1$ and $u_2$ satisfy \eqref{maintheo3} with $0<\epsilon \le \widetilde{\epsilon}_2 $. 

We now state our uniqueness result. 
\begin{proposition} \label{uniqueness} 
Let $u_1$ and $u_2$ be two solutions to the equation in \eqref{5KdV} in the class \eqref{maintheo1} with $s=2$, defined on a time interval $[-T,T]$ for some $T>0$  and satisfying $u_1(\cdot,0)=u_2(\cdot,0)=\varphi$. Then $u_1=u_2$ on $[-T,T]$.
\end{proposition} 

\begin{proof} Let us define $K:=\max\big\{\Gamma_T^2(u_1),\Gamma_T^2(u_2) \big\}.$ As in the proof of Proposition \ref{apriori}, we use the scaling property of \eqref{5KdV} and define $u_{i,\lambda}=\lambda^2u_i(\lambda x,\lambda^5 t)$, for $i=1,2$ and $\lambda>0$, which are also solutions to the equation in \eqref{5KdV} on the time interval $[-S,S]$ with $S=\lambda^{-5}T$ and with initial data $\varphi_{\lambda}=\lambda^2\varphi(\lambda \cdot)$.  Moreover, since 
\begin{displaymath} 
\|u_{i,\lambda}\|_{L^{\infty}_SH^2_x} + \|u_{i,\lambda}\|_{B^2(S)} 
\lesssim \lambda^{\frac32}(1+\lambda^2)\big(\|u_i\|_{L^{\infty}_TH^2_x}+\|u_i\|_{B^2(T)}\big) \lesssim \lambda^{\frac32}(1+\lambda^2)K,
\end{displaymath}
for $i=1,2$, we can always choose $\lambda=\lambda(K)$ small enough such that 
\begin{equation} \label{uniqueness1} 
\|\varphi_{\lambda}\|_{H^2} \le \epsilon, \ \|u_{i,\lambda}\|_{B^2(S)} \le C_2 \epsilon/(3c) \le C_2 \widetilde{\epsilon}_2/(3c)
\ \text{and} \ \|u_{i,\lambda}\|_{F^2(S)} \le C(K),
\end{equation}
for $i=1,2$, and where $c$ is the implicit constant appearing in the first inequality of \eqref{uniqueness3} below.

Since $\|u_{i,\lambda}\|_{F^2(S)} < \infty$, there exists $n \in \mathbb Z_+$ such that
\begin{equation} \label{uniqueness2}
\|P_{> n}u_{i,\lambda}\|_{F^2(S)} \le C_2\epsilon/3, \quad i=1,2.
\end{equation}
On the other hand, we deduce from \eqref{linear1}, \eqref{apriorilemma7} and \eqref{uniqueness1} that 
\begin{equation} \label{uniqueness3} 
\begin{split}
\|P_{\le n}u_{i,\lambda}\|_{F^2(S)} & \lesssim \|u_{i,\lambda}\|_{B^2(S)}+
S^{\frac12}2^{2n}\big\|P_{\le n}\partial_x\big((\partial_xu_{i,\lambda})^2\big)\big\|_{L^{\infty}_SL^2_x} \\ & \quad+
S^{\frac12}2^{2n}\big\|P_{\le n}\partial_x\big(u_{i,\lambda}\partial_x^2u_{i,\lambda}\big)\big\|_{L^{\infty}_SL^2_x} \\ 
& \lesssim C_2\epsilon/3+S^{\frac12}2^{3n}\|u_{i,\lambda}\|_{L^{\infty}_SH^2_x}^2.
\end{split}
\end{equation}
By choosing $S_1=S_1(K)$ small enough, we deduce from \eqref{uniqueness1}--\eqref{uniqueness3} that $u_{\lambda,1}$ and  $u_{\lambda,2}$ satisfy the smallness condition \eqref{maintheo3} on $[-S_1,S_1]$, \textit{i.e.} 
\begin{displaymath}  
\Gamma_{S_1}^2(u_{i,\lambda}) \le C_2 \epsilon \le C_2 \widetilde{\epsilon}_2, \quad \text{for} \ i=1,2. 
\end{displaymath}
This implies from \eqref{maintheo4} that $u_{1,\lambda}\equiv u_{2,\lambda}$ on $[-S_1,S_1]$. By applying this argument a finite number of times, we see that the equality holds in fact in $[-S,S]$. Then it follows after changing variables that  $u_{1}\equiv u_{2}$ on $[-T,T]$.
\end{proof}

\subsection{Existence} 
Let $2 \le s < 4$ and $u_0 \in H^s(\mathbb R)$.  By using a scaling argument as in the proof of Proposition \ref{apriori}, we can assume that $u_0 \in \mathcal{B}^s(\epsilon)$, with $\epsilon < \widetilde{\epsilon}_s \le \min (\epsilon_s,\frac{C_2}{C_s}\widetilde{\epsilon}_2)$. Note here that $\widetilde{\epsilon}_s$ will be determined later.

We will use the Bona-Smith argument (c.f. \cite{BS}).  Let $\rho \in \mathcal{S}(\mathbb R)$ with $\rho \ge 0$, $\int\rho
dx=1$, and $\int x^k\rho(x)dx=0, \ k \in \mathbb Z_{+}$, $0 \le k \le [s]+1$. For any $\lambda>0$, define $\rho_{\lambda}(x)=\lambda^{-1}\rho(\lambda^{-1}x)$. The following lemma, whose proof can be found in \cite{BS}
(see also Proposition 2.1 in \cite{KP}), gathers the properties of the smoothing operators which will be used in this
section.
\begin{lemma} \label{lemma existence1}
Let $s \ge 0$, $\phi \in H^s(\mathbb R)$ and for any $\lambda>0$,
$\phi_{\lambda}=\rho_{\lambda} \ast \phi$. Then,
\begin{equation} \label{lemma existence1.1}
\|\phi_{\lambda}\|_{H^{s+\alpha}} \lesssim \lambda^{-\alpha}
\|\phi\|_{H^s}, \quad \forall \alpha \ge0,
\end{equation}
and
\begin{equation} \label{lemma existence1.2}
\|\phi-\phi_{\lambda}\|_{H^{s-\beta}} \underset{\lambda
\rightarrow 0}{=} o(\lambda^{\beta}), \quad \forall \beta \in
[0,s].
\end{equation}
\end{lemma}

Now we regularize the initial data by letting $u_{0,\lambda}=\rho_{\lambda} \ast u_0$. Since $u_{0,\lambda} \in H^{\infty}(\mathbb R)$, we deduce from Theorem \ref{smoothsol} that for any $\lambda>0$, there exists a positive time $T_{\lambda}$ and a unique solution 
\begin{displaymath} 
u_{\lambda} \in C([-T_{\lambda},T_{\lambda}];H^{\infty}(\mathbb R)) \quad \text{satisfying} \quad u_{\lambda}(\cdot,0)=u_{0,\lambda}.
\end{displaymath} 

We observe that $\|u_{0,\lambda}\|_{H^s} \le \|u_0\|_{H^s} \le \epsilon$. Thus, it follows from the proof of Proposition \ref{apriori}  and estimate \eqref{lemma existence1.1}, that the sequence of solutions $\{u_{\lambda}\}$ can be extended on the time interval $[-1,1]$ and satisfy 
\begin{equation} \label{maintheo5}
\Gamma^s_1(u_{\lambda}) \le C_s\epsilon \le \min(C_2\widetilde{\epsilon}_2,C_s\widetilde{\epsilon}_s),
\end{equation}
\begin{equation} \label{maintheo6} 
\Gamma^s_1(u_{\lambda}) \lesssim \|u_0\|_{H^s} \quad \text{and} 
\quad \Gamma^{s+2}_1(u_{\lambda}) \lesssim \|u_{0,\lambda}\|_{H^{s+2}} \lesssim \lambda^{-2}\|u_0\|_{H^s},
\end{equation}
for all $\lambda>0$.

Then, we deduce from \eqref{maintheo4} and  \eqref{lemma existence1.2} that for any $0<\lambda' < \lambda$,
\begin{equation} \label{maintheo7}
\Gamma^0_1(u_{\lambda}-u_{\lambda'}) \lesssim \|u_{0,\lambda}-u_{0,\lambda'}\|_{L^2_x} \underset{\lambda \rightarrow 0}{=} o(\lambda^{s}).
\end{equation}
Moreover, we obtain gathering estimates \eqref{linear1}, \eqref{bilinE1}--\eqref{bilinE2}, \eqref{coroEEdiff1}, 
\eqref{maintheo5} and choosing $\widetilde{\epsilon}_s$ small enough that
\begin{equation} \label{maintheo8}
\Gamma^s_1(u_{\lambda}-u_{\lambda'}) \lesssim \|u_{0,\lambda}-u_{0,\lambda'}\|_{H^s_x} +\Gamma^{s+2}_1(u_{\lambda})\Gamma^0_1(u_{\lambda}-u_{\lambda'}),
\end{equation}
since $s \ge 2$. This combined with \eqref{lemma existence1.2}, \eqref{maintheo6} and \eqref{maintheo7} yields
\begin{equation} \label{maintheo9} 
\|u_{\lambda}-u_{\lambda'}\|_{L^{\infty}_1H^s_x} \lesssim \Gamma^s_1(u_{\lambda}-u_{\lambda'}) \underset{\lambda \rightarrow 0}{\longrightarrow} 0 \ . 
\end{equation}

Therefore, we conclude that $\{u_{\lambda}\}$  converges in the norm $\Gamma_1^s$ to a solution $u$ of \eqref{5KdV} in the class \eqref{maintheo1}.

\begin{remark}
Observe that the convergence of $\{u_{\lambda}\}$ in $C([-1,1];H^1(\mathbb R))$ would be enough to obtain that the limit $u$ satisfies the equation in \eqref{5KdV} in the weak sense. 
\end{remark}

\subsection{Continuity of the flow map data-solution} Observe that for $s \ge 4$, the result was already proved in Theorem 3.1 in \cite{Po}. Then it is enough to prove it for $2 \le s < 4$. Let $u_0 \in H^s(\mathbb R)$. Once again we can assume by using a scaling argument that $u_0 \in \mathcal{B}^s(\epsilon)$ with $0<\epsilon \le \bar{\epsilon}< \widetilde{\epsilon}_s$ and where $\widetilde{\epsilon}_s$ was determined in the previous subsection. Then, the solution $u$ emanating from $u_0$ is defined on the time interval $[-1,1]$ and satisfies $u \in C([-1,1];H^s(\mathbb R))$.

Let $\theta>0$ be given. It suffices to prove that for any initial data $v_0 \in \mathcal{B}^s(\epsilon)$ with $\|u_0-v_0\|_{H^s} \le \delta$, where $\delta=\delta(\theta)>0$ will be fixed later, the solution $v \in C([-1,1];H^s(\mathbb R))$  emanating from $v_0$ satisfies 
\begin{equation} \label{maintheo10} 
\|u-v\|_{L^{\infty}_1H^s_x} \le \theta.
\end{equation}

For any $\lambda>0$, we normalize the initial data $u_0$ and $v_0$ by defining $u_{0,\lambda}=\rho_{\lambda} \ast u_0$ and $v_{0,\lambda}=\rho_{\lambda} \ast v_0$ as in the previous subsection and consider the associated smooth solutions $u_{\lambda}, \ v_{\lambda} \in C([-1,1];H^{\infty}(\mathbb R))$. Then it follows from the triangle inequality that 
\begin{equation} \label{maintheo11} 
\|u-v\|_{L^{\infty}_1H^s_x} \le \|u-u_{\lambda}\|_{L^{\infty}_1H^s_x}+\|u_{\lambda}-v_{\lambda}\|_{L^{\infty}_1H^s_x}
+\|v-v_{\lambda}\|_{L^{\infty}_1H^s_x} \ .
\end{equation}
On the one hand, according to \eqref{maintheo9}, we can choose $\lambda_0$ small enough so that 
\begin{equation} \label{maintheo12} 
\|u-u_{\lambda_0}\|_{L^{\infty}_1H^s_x}+\|v-v_{\lambda_0}\|_{L^{\infty}_1H^s_x} \le 2\theta/3.
\end{equation}
On the other hand, we get from \eqref{lemma existence1.1} that 
\begin{displaymath} 
\|u_{0,\lambda_0}-v_{0,\lambda_0}\|_{H^4} \lesssim \lambda_0^{-(4-s)}\|u_0-v_0\|_{H^4} \lesssim \lambda_0^{-(4-s)}\delta.
\end{displaymath}
Therefore, by using the continuity of the flow map for smooth initial data (c.f. Theorem 3.1 in \cite{Po}), we can choose $\delta>0$ small enough such that 
\begin{equation} \label{maintheo13} 
\|u_{\lambda_0}-v_{\lambda_0}\|_{L^{\infty}_1H^s_x} \le \theta/3.
\end{equation}
Estimate \eqref{maintheo10} is concluded gathering \eqref{maintheo11}--\eqref{maintheo13}.

\section{Appendix: how to deal with the cubic term $\partial_x(u^3)$.} 
In this appendix, we explain what are the main modifications needed to deal with cubic term $\partial_x(u^3)$ (\textit{i.e.} in the case where $c_3 \neq 0$). As above, we fix $\alpha=2$ in the definition of the spaces $F^s_{\alpha}(T), \ N^s_{\alpha}(T), \ F^s_{\alpha}, \ N^s_{\alpha}, \ F_{k,\alpha}, \ N_{k,\alpha}$ and write those spaces without the index $\alpha=2$, since there is no risk of confusion.

\subsection{Short time trilinear estimate} In this subsection, we prove the trilinear estimate for the nonlinear term $\partial_x(u^3)$.
\begin{proposition} \label{trilinear} 
Let $s \ge 0$ and $T \in (0,1]$ be given. Then, it holds that 
\begin{equation} \label{trilinear1} 
\begin{split}
\|\partial_x(uvw)\|_{N^s(T)} & \lesssim \|u\|_{F^0(T)}\|v\|_{F^0(T)}\|w\|_{F^s(T)}+\|u\|_{F^0(T)}\|w\|_{F^0(T)}\|v\|_{F^s(T)} \\ & \quad +\|v\|_{F^0(T)}\|w\|_{F^0(T)}\|u\|_{F^s(T)},
\end{split}
\end{equation}
for all $u, \ v, \ w \in F^s(T)$.
\end{proposition} 

We split the proof of Proposition \ref{trilinear} in several technical lemmas depending of the frequency interactions. 
\begin{lemma} \label{high-low-low-high}[high $\times$ low $\times$ low $\rightarrow$ high]
Assume that $k, \ k_1, \ k_2, \ k_3 \in \mathbb Z_+$ satisfy $k \ge 20$, $|k_3-k| \le 5$ and $0 \le k_1 \le k_2 \le k_3 -10$. 
Then, 
\begin{equation} \label{high-low-low-high1} 
\big\| P_k\partial_x\big(u_{k_1}v_{k_2}w_{k_3}\big)\big\|_{N_k} \lesssim 2^{-3k/2}\|u_{k_1}\|_{F_{k_1}}\|v_{k_2}\|_{F_{k_2}}\|w_{k_3}\|_{F_{k_3}},
\end{equation}
for all $u_{k_1} \in F_{k_1}, \ v_{k_2} \in F_{k_2}$ and $w_{k_3} \in F_{k_3}$.
\end{lemma}

\begin{proof} Arguing exactly as in the proof of Lemma \ref{high-low}, it suffices to prove that 
\begin{equation} \label{high-low-low-high2} 
\begin{split}
2^{k}\sum_{j \ge 2k}2^{-j/2}\big\|\textbf{1}_{D_{k,j}}&\cdot \big(f_{k_1,j_1}\ast f_{k_2,j_2} \ast f_{k_3,j_3}\big) \big\|_{L^2_{\xi,\tau}} \\ & \lesssim 2^{j_1/2}\|f_{k_1,j_1}\|_{L^2}2^{j_2/2}\|f_{k_2,j_2}\|_{L^2}2^{j_3/2}\|f_{k_3,j_3}\|_{L^2},
\end{split}
\end{equation}
where the functions $f_{k_i,j_i}$ are localized in $D_{k_i,j_i}$, with $j_i \ge 2k$, for $i=1,2,3$.

But, we deduce from estimates \eqref{coroL2bilin2b} and \eqref{coroL2bilin3} that 
\begin{displaymath} 
\begin{split}
2^{k}\sum_{j \ge 2k}&2^{-j/2}\big\|\textbf{1}_{D_{k,j}}\cdot \big(f_{k_1,j_1}\ast f_{k_2,j_2}\ast f_{k_3,j_3}\big) \big\|_{L^2_{\xi,\tau}} 
\\ & \lesssim 2^{k}\sum_{j \ge 2k}2^{-j/2}2^{k_1/2}2^{-2k}2^{j_1/2}\|f_{k_1,j_1}\|_{L^2}2^{j_2/2}\|f_{k_2,j_2}\|_{L^2}2^{j_3/2}\|f_{k_3,j_3}\|_{L^2},
\end{split}
\end{displaymath}
which implies estimate \eqref{high-low-low-high2} after summing over $j$. 
\end{proof}

\begin{lemma} \label{high-high-low-high}[high $\times $ high $\times$ low $\rightarrow$ high]
Assume that $k, \ k_1, \ k_2, \ k_3 \in \mathbb Z_+$ satisfy $k \ge 20$, $|k_3-k| \le 5$, $k_3 -10 \le k_2 \le k_3$ and $0 \le k_1 \le k_2-20$. 
Then, 
\begin{equation} \label{high-high-low-high1} 
\big\| P_k\partial_x\big(u_{k_1}v_{k_2}w_{k_3}\big)\big\|_{N_k} \lesssim 2^{-k}\|u_{k_1}\|_{F_{k_1}}\|v_{k_2}\|_{F_{k_2}}\|w_{k_3}\|_{F_{k_3}},
\end{equation}
for all $u_{k_1} \in F_{k_1}, \ v_{k_2} \in F_{k_2}$ and $w_{k_3} \in F_{k_3}$.
\end{lemma}

\begin{proof} Once again, it is enough to prove that estimate \eqref{high-low-low-high2} remains true in this case. According to the frequency localization, we have that $\widetilde{\Omega} \sim 2^{5k_{max}}$, where $\widetilde{\Omega}$ is defined in \eqref{resonancetilde}. This yields $j_{max} \ge 5k-20$. Therefore, it follows from estimate \eqref{coroL2trilin1} that
\begin{equation} \label{high-high-low-high2}
\begin{split}
2^{k}&\sum_{j \ge 2k}2^{-j/2}\big\|\textbf{1}_{D_{k,j}}\cdot \big(f_{k_1,j_1}\ast f_{k_2,j_2}\ast f_{k_3,j_3}\big) \big\|_{L^2_{\xi,\tau}} 
\\ & \lesssim 2^{k}\sum_{j \ge 2k}2^{-j/2}2^{(j_1+j_2+j_3+j)/2}2^{(k_1+k_2)/2}2^{-(j_{max}+j_{sub})/2}\prod_{i=1}^3\|f_{k_i,j_i}\|_{L^2},
\end{split}
\end{equation}
which provides the bound in estimate \eqref{high-low-low-high2}, in both cases $j_{max}=j$ and $j_{max} \neq j_{max}$. This finishes the proof of Lemma \ref{high-high-low-high}.
\end{proof}

\begin{lemma} \label{high-high-high-high}[high $\times $ high $\times$ high $\rightarrow$ high]
Assume that $k, \ k_1, \ k_2, \ k_3 \in \mathbb Z_+$ satisfy $k \ge 20$, $|k_3-k| \le 5$, $k_3 -10 \le k_2 \le k_3$ and  
$k_2-30 \le k_1 \le k_2$. 
Then, 
\begin{equation} \label{high-high-high-high1} 
\big\| P_k\partial_x\big(u_{k_1}v_{k_2}w_{k_3}\big)\big\|_{N_k} \lesssim \|u_{k_1}\|_{F_{k_1}}\|v_{k_2}\|_{F_{k_2}}\|w_{k_3}\|_{F_{k_3}},
\end{equation}
for all $u_{k_1} \in F_{k_1}, \ v_{k_2} \in F_{k_2}$ and $w_{k_3} \in F_{k_3}$.
\end{lemma}

\begin{proof} We argue exactly as in the proof of Lemma \ref{high-high-low-high} and observe that estimate \eqref{high-high-low-high2} leads to estimate \eqref{high-low-low-high2} even without using that $j_{max} \ge 5k-20$, which is not always satisfied in this case. Instead, it is sufficient to use that $j, \ j_i \ge 2k$ for all $i=1,2,3$.
\end{proof}

\begin{lemma} \label{high-high-high-low}[high $\times $ high $\times$ high $\rightarrow$ low]
Assume that $k, \ k_1, \ k_2, \ k_3 \in \mathbb Z_+$ satisfy $k_3-5 \le k_2 \le k_3$, $k_2 -10 \le k_1 \le k_2$ and  $20 \le k \le k_1-10$. 
Then, 
\begin{equation} \label{high-high-high-low1} 
\big\| P_k\partial_x\big(u_{k_1}v_{k_2}w_{k_3}\big)\big\|_{N_k} \lesssim 2^{-(k_3+k)/2}\|u_{k_1}\|_{F_{k_1}}\|v_{k_2}\|_{F_{k_2}}\|w_{k_3}\|_{F_{k_3}},
\end{equation}
for all $u_{k_1} \in F_{k_1}, \ v_{k_2} \in F_{k_2}$ and $w_{k_3} \in F_{k_3}$.
\end{lemma}

\begin{proof} We argue as in the proof of Lemma \ref{high-high-low}. Thus it is enough to prove that 
\begin{equation} \label{high-high-high-low2} 
\begin{split}
2^{2(k_3-k)}2^{k}\sum_{j \ge 0}2^{-j/2}\big\|\textbf{1}_{D_{k,j}}&\cdot \big(f_{k_1,j_1}^m\ast f_{k_2,j_2}^m \ast f_{k_3,j_3}^m\big) \big\|_{L^2_{\xi,\tau}} \\ & \lesssim 2^{j_1/2}\|f_{k_1,j_1}^m\|_{L^2}2^{j_2/2}\|f_{k_2,j_2}^m\|_{L^2}2^{j_3/2}\|f_{k_3,j_3}^m\|_{L^2},
\end{split}
\end{equation}
where the functions $f_{k_i,j_i}^m$ are localized in $D_{k_i,j_i}$, with $j_i \ge 2k_3$, for $i=1,2,3$. According to estimate \eqref{coroL2trilin1}, we can bound the left-hand side of \eqref{high-high-high-low2} by
\begin{displaymath} 
2^{2k_3-k}\sum_{j \ge 0}2^{-j/2}2^{(j_1+j_2+j_3+j)/2}2^{(k+k_1)/2}2^{-(j_{max}+j_{sub})/2}\prod_{i=1}^3\|f_{k_i,j_i}\|_{L^2}.
\end{displaymath}
 Moreover, we have $\widetilde{\Omega} \sim 2^{5k_{max}}$ in this case, so that $j_{max} \ge 5k_3-20$. This implies estimate \eqref{high-high-high-low2} in both cases $j=j_{max}$ and $j \neq j_{max}$. 
\end{proof}

\begin{lemma} \label{high-high-low-low}[high $\times $ high $\times$ low $\rightarrow$ low]
Assume that $k, \ k_1, \ k_2, \ k_3 \in \mathbb Z_+$ satisfy $k \ge 20$, $k_3-5 \le k_2 \le k_3$  and  $0 \le k_1, k \le k_2-10$. 
Then, 
\begin{equation} \label{high-high-low-low1} 
\big\| P_k\partial_x\big(u_{k_1}v_{k_2}w_{k_3}\big)\big\|_{N_k} \lesssim 2^{-(k_3+k)/2}\|u_{k_1}\|_{F_{k_1}}\|v_{k_2}\|_{F_{k_2}}\|w_{k_3}\|_{F_{k_3}},
\end{equation}
for all $u_{k_1} \in F_{k_1}, \ v_{k_2} \in F_{k_2}$ and $w_{k_3} \in F_{k_3}$.
\end{lemma}

\begin{proof} Following the proof of Lemma \ref{high-high-high-low}, we need to prove that estimate \eqref{high-high-high-low2} still holds in this case. This is a direct consequence of estimates \eqref{coroL2trilin3} and \eqref{coroL2trilin4}.
\end{proof}

\begin{lemma} \label{low-low-low-low}[low $\times $ low $\times$ low $\rightarrow$ low]
Assume that $k, \ k_1, \ k_2, \ k_3 \in \mathbb Z_+$ satisfy $0 \le k, \ k_1, \ k_2, \ k_3 \le 200$. 
Then, 
\begin{equation} \label{low-low-low-low1} 
\big\| P_k\partial_x\big(u_{k_1}v_{k_2}w_{k_3}\big)\big\|_{N_k} \lesssim \|u_{k_1}\|_{F_{k_1}}\|v_{k_2}\|_{F_{k_2}}\|w_{k_3}\|_{F_{k_3}},
\end{equation}
for all $u_{k_1} \in F_{k_1}, \ v_{k_2} \in F_{k_2}$ and $w_{k_3} \in F_{k_3}$.
\end{lemma}

\begin{proof} It follows arguing as in Lemma \ref{low-low}.
\end{proof}

Finally, we give the proof of Proposition \ref{trilinear}. 
\begin{proof}[Proof of Proposition \ref{trilinear}] 
Fix $s \ge 0$. We choose two extensions $\tilde{u}$, $\tilde{v}$ and $\tilde{w}$ of $u$, $v$ and $w$ satisfying 
\begin{displaymath}
\|\tilde{u}\|_{F^s} \le 2\|u\|_{F^s(T)}, \quad \|\tilde{v}\|_{F^s}\le 2\|v\|_{F^s(T)} \quad \text{and} \quad \|\tilde{w}\|_{F^s} \le 2\|w\|_{F^s(T)}.
\end{displaymath}
Therefore $\partial_x(\tilde{u}\tilde{v}\tilde{w})$ is an extension of $\partial_x(uvw)$ on $\mathbb R^2$ and we have from the definition of $N^s(T)$ and Minkowski inequality that
\begin{displaymath}  
\|\partial_x\big(uvw\big)\|_{N^s(T)}   
\le \Big(\sum_{k \ge 0}2^{2ks}\Big(\sum_{k_1,  k_2, k_3 \ge 0}\big\|P_k\partial_x\big(P_{k_1}\tilde{u}P_{k_2}\tilde{v}P_{k_3}\tilde{w}\big)\big\|_{N_k}\Big)^2 \Big)^{\frac12}.
\end{displaymath}
Note that by symmetry, we can always assume that $0 \le k_1 \le k_2 \le k_3$. Moreover, we denote
\begin{displaymath} 
\begin{array}{l}
G_1=\big\{(k_1,k_2,k_3) \in \mathbb Z_+^3 \ : \  k \ge 20, \ |k_3-k| \le 5 ,\  0 \le k_1\le k_2 \le k_3-10\big\} ,
\\ 
G_2=\big\{(k_1,k_2,k_3) \in \mathbb Z_+^3 \ : \  k \ge 20, \ |k_3-k| \le 5 \ |k_3-k_2| \le 10, \ 0 \le k_1 \le k_2-20\big\}, \\ 
G_3=\big\{(k_1,k_2,k_3) \in \mathbb Z_+^3 \ : \  k \ge 20, \ |k_3-k| \le 5 \ |k_3-k_2| \le 10, \ |k_1-k_2| \le 30\big\}, \\
G_4=\big\{(k_1,k_2,k_3) \in \mathbb Z_+^3 \ : \  k_3-5 \le k_2 \le k_3, \ k_2-10 \le k_1 \le k_2,  \ 20 \le k \le k_1-10
\big\},\\ 
G_5=\big\{(k_1,k_2,k_3) \in \mathbb Z_+^3 \ : \ k_3-5 \le k_2 \le k_3, \ 20 \le k,k_1 \le k_2-10\big\}, \\ 
G_6=\big\{(k_1,k_2,k_3) \in \mathbb Z_+^3 \ : \ 0 \le k, \ k_1, \ k_2, \ k_3 \le 200\big\}.
\end{array}
\end{displaymath}

Note that for a given $k \in \mathbb Z_+$, some of these regions may be empty and others may overlap, but due to the frequency localization, we always have that 
\begin{equation} \label{trilinear2} 
\|\partial_x\big(uvw\big)\|_{N^s(T)}   
\le \sum_{i=1}^6\Big(\sum_{k \ge 0}2^{2ks}\Big(\sum_{(k_1,  k_2, k_3) \in G_i}\big\|P_k\partial_x\big(P_{k_1}\tilde{u}P_{k_2}\tilde{v}P_{k_3}\tilde{w}\big)\big\|_{N_k}\Big)^2 \Big)^{\frac12}.
\end{equation}
We conclude the proof of Proposition \ref{trilinear} by applying respectively Lemmas \ref{high-low-low-high}--\ref{low-low-low-low} to each of the sum appearing on the right-hand side of \eqref{trilinear2}.
\end{proof}

\subsection{Modifications to the energy estimates}  We only explain how to deal with the \textit{a priori} estimates, since the modifications would be similar to derive estimates for the differences of two solutions. The main point is to derive an analog to Proposition \ref{propEE} in the case where $c_3 \neq 0$. 

\begin{proposition} \label{propEEwithc3}
Assume $s \ge \frac54$ and $T \in (0,1]$. Then, if $u \in C([-T,T];H^{\infty}(\mathbb R))$ is a solution to \eqref{5KdV}, we have that 
\begin{equation} \label{propEEwithc3.1} 
E^s_T(u) \lesssim (1+\|u_0\|_{H^s})\|u_0\|_{H^s}^2+\big(1+\Gamma_T^s(u)+\Gamma_T^s(u)^2\big)\Gamma_T^s(u)^3,
\end{equation}
where 
\begin{displaymath} 
\Gamma_T^s(u):=\max \big\{\|u\|_{F^s(T)},\|u\|_{B^s(T)} \big\}.
\end{displaymath}
\end{proposition}

\begin{proof} The proof of Proposition \ref{propEEwithc3} follows the same strategy as the one of Proposition \ref{propEE}. The unique difference is that we need to add the terms 
$\mathcal{M}_k^1(u)$, $\alpha \mathcal{M}_k^1(u)$ and $\beta \mathcal{M}_k^2(u)$
to the right-hand side of \eqref{propEE2}, where
\begin{displaymath}
\mathcal{K}_k(u)=2c_3\int_{\mathbb R}P_kuP_k\partial_x\big(u^3\big)dx,
\end{displaymath}
\begin{displaymath}  
\begin{split}
\mathcal{M}_k^1(u)&=c_3\int_{\mathbb R}\partial_x\big(u^3\big)P_k\partial_x^{-1}uQ_k\partial_x^{-1}udx+
c_3\int_{\mathbb R}uP_k\big(u^3\big)Q_k\partial_x^{-1}udx\\ & \quad +
c_3\int_{\mathbb R}uP_k\partial_x^{-1}uQ_k\big(u^3\big)dx,
\end{split}
\end{displaymath}
and
\begin{equation} \label{Mk2}  
\mathcal{M}_k^2(u)=c_3\int_{\mathbb R}\partial_x\big(u^3\big)P_k\partial_x^{-1}uP_k\partial_x^{-1}udx+
2c_3\int_{\mathbb R}uP_k\big(u^3\big)P_k\partial_x^{-1}udx.
\end{equation}
Therefore, it suffices to bound 
\begin{displaymath} 
\sum_{k \ge 1}2^{2ks}\sup_{t_k \in [0,T]}\Big|\int_{[0,t_k]}\big(\mathcal{K}_k(u)+\alpha\mathcal{M}_k^1(u)+\beta\mathcal{M}_k^2(u)\big)dt \Big|
\end{displaymath}
by the the terms appearing on the right-hand side of \eqref{propEEwithc3.1}.

We first treat the fourth-order term corresponding to $\mathcal{K}_k(u)$. We perform the same dyadic decomposition as in the proof of Proposition \ref{trilinear}. Thus, 
\begin{equation} \label{propEEwithc3.2} 
\begin{split}
&\sum_{k \ge 1}2^{2ks}\sup_{t_k \in [0,T]}\Big|\int_{[0,t_k]}\mathcal{K}_k(u)dt\Big| \\ &
\lesssim \sum_{i=1}^6\sum_{k \ge 1}2^{2ks}\sum_{(k_1,k_2,k_3) \in G_i}\sup_{t_k \in [0,T]}
\Big|\int_{\mathbb R \times [0,t_k]}P_kuP_k\partial_x\big(P_{k_1}uP_{k_2}uP_{k_3}u\big)dx\Big|.
\end{split}
\end{equation}
By using respectively estimate \eqref{tec3EE1} for the sums over $G_1$ and $G_5$ and estimate \eqref{tec3EE2} for the sums over $G_2$ and $G_4$, the corresponding terms on the right-hand side of \eqref{propEEwithc3.2} can be bounded by 
\begin{equation} \label{propEEwithc3.3}
\|u\|_{F^0(T)}\|u\|_{F^{\frac12+}(T)}\|u\|_{F^s(T)}^2.
\end{equation}
In the regions $G_3$ and $G_6$, we use estimates \eqref{lemma1.1} and \eqref{Bstrichartz1b} to bound the corresponding terms by 
\begin{equation} \label{propEEwithc3.4}
\|u\|_{F^0(T)}\|u\|_{F^{\frac34+}(T)}\|u\|_{B^s(T)}^2.
\end{equation}
Observe that \eqref{propEEwithc3.3} and \eqref{propEEwithc3.4} are controlled by the second term on the right-hand side of \eqref{propEEwithc3.1}.

Next, we deal with the fifth order term corresponding to $\mathcal{M}_k^2(u)$ and observe that the one corresponding to $\mathcal{M}_k^1(u)$ could be treated similarly. It follows from estimate \eqref{lemma1.1} that 
\begin{displaymath} 
\Big|\int_{\mathbb R \times [0,t_k]}\partial_x\big(u^3\big)P_k\partial_x^{-1}uP_k\partial_x^{-1}udx\Big| 
\lesssim \|u\|_{F^{\frac12+}(T)}^3\|P_ku\|_{L^{\infty}_TL^2_x}\|\partial_x^{-1}P_ku\|_{L^{\infty}_TL^2_x},
\end{displaymath}
which leads to the bound in \eqref{propEEwithc3.1} after summing over $k \in \mathbb Z_+ \cap [1,+\infty)$ and taking the supreme over $t_k \in [0,T]$. Finally, to deal with the second term on the right-hand side of \eqref{Mk2}, we introduce a dyadic decomposition $$\partial_x(u^3)=\sum_{k_1,k_2,k_3}\partial_x\big(P_{k_1}uP_{k_2}uP_{k_3}u \big),$$ and use estimates \eqref{lemma1.1} and \eqref{Bstrichartz1b} to obtain the right estimate. 

This finishes the proof of Proposition \ref{propEEwithc3}.
\end{proof}

\vspace{0,5cm}

\noindent \textbf{Acknowledgments.} This research was carried out when the second author was visiting the Department
of Mathematics of the University of Chicago, whose hospitality is gratefully acknowledged. D.P. would like to thank Luc Molinet for helpful conversations about this work. The authors would also like to thank Tadahiro Oh for pointing out a technical mistake in a previous version of the proof of Lemma \ref{apriorilemma}.

\bibliographystyle{amsplain}

\begin{thebibliography}{99}

\bibitem{An} 
\newblock J. Angulo, 
\newblock \emph{On the instability of solitary-wave solutions for fifth-order
water wave models,} 
\newblock Electron. J. Diff. Eq.,  \textbf{6} (2003), 1--18.

\bibitem{Be} 
\newblock D. J. Benney, 
\newblock \emph{A general theory for interactions between short and long waves,} 
\newblock Stud. Appl. Math., \textbf{56} (1977), 81--94.

\bibitem{BS} 
\newblock J. L. Bona and R. Smith, 
\newblock \emph{The initial value problem for the Korteweg-de Vries equation,} 
\newblock Philos. Trans. R. Soc. Lond., Ser. A, \textbf{278} (1975), 555--601.

\bibitem{CCT} M. Christ, J. Colliander and T. Tao, 
\newblock \emph{A priori bounds and weak solutions for the nonlinear Schr\"odinger 
equation in Sobolev spaces of negative order,} 
\newblock J. Funct. Anal., \textbf{254} (2008), no. 2, 368--395.

\bibitem{CGK}  W. Craig, P. Guyenne and H. Kalisch,
\newblock \emph{Hamiltonian long wave expansions for free surfaces and
interfaces,}
\newblock  Comm. Pure Appl. Math., \textbf{58} (2005), 1587--1641.

\bibitem{CG} 
\newblock W. Craig and M. Groves, 
\newblock \emph{Hamiltonian long-wave approximations to the water-wave problem,}
\newblock Wave Motion, \textbf{19} (1994), 367--389.

\bibitem{CGL} 
\newblock W. Chen, Z. Guo and Z. Liu, 
\newblock \emph{Sharp well-posedness for a fifth-order shallow water wave equation,}
\newblock J. Math. Anal. Appl., \textbf{369} (2010), 133--143.

\bibitem{CLMW} 
\newblock W. Chen, J. Li, C. Miao and J. Wu, 
\newblock \emph{Low regularity solutions of two fifth-order KdV type equations,} 
\newblock J. Anal. Math., \textbf{127} (2009), 221--238.

\bibitem{Guo} 
\newblock Z. Guo, 
\newblock \emph{Local well-posedness for dispersion generalized Benjamin-Ono equations in Sobolev spaces,} 
\newblock preprint (2008), arxiv:0812.1825v2. 

\bibitem{Guo1}
\newblock Z. Guo, 
\newblock \emph{Local wellposedness and a priori bounds for the modified Benjamin-Ono equation,} 
\newblock Adv. Diff. Eq., \textbf{16} (2011), 1087--1137.

\bibitem{GKK} 
\newblock Z. Guo, C. Kwak and S. Kwon, 
\newblock \emph{Rough solutions of the fifth-order KdV equations,} 
\newblock preprint (2012).

\bibitem{Gr} 
\newblock A. Gr\"unrock, 
\newblock \emph{On the hierarchies of higher order mKdV and KdV equations,} 
\newblock Cent. Eur. J. Math., \textbf{8} (2010), 500--536. 

\bibitem{Gr2} 
\newblock A. Gr\"unrock,
\newblock \emph{A bilinear Airy estimate with application to gKdV-3,} 
\newblock Diff. Int. Eq., \textbf{18} (2005), 1333--1339.

\bibitem{He} 
\newblock S. Herr, 
\newblock \emph{Well-posedness for equations of Benjamin-Ono type,}
\newblock Illinois J. Math., \textbf{51} (2007), 951--976.
 
\bibitem{IKT}
\newblock A. D. Ionescu, C. E. Kenig and D. Tataru, 
\newblock \emph{Global well-posedness of the KP-I initial value problem in the energy space,}
\newblock Invent. Math., \textbf{173} (2008), 265--304. 

\bibitem{Ka} T. Kato,
\newblock \emph{Well-posedness for the fifth order KdV equation,}
\newblock preprint (2010), arxiv:1011.3956v1.

\bibitem{KP} 
\newblock T. Kato and G. Ponce, 
\newblock \emph{On nonstationary flows of viscous and ideal fluids in $L^p_s(\mathbb R^2)$},
\newblock Duke Math. J., \textbf{50} (1987), 487--499.


\bibitem{KK}
\newblock  C. E. Kenig, K. D. Koenig,
\newblock \emph{On the local well-posedness of the Benjamin-Ono and modified Benjamin-Ono equations,}
\newblock  Math. Res. Lett., \textbf{10} (2003), 879--895. 

\bibitem{KPV} 
\newblock C. E. Kenig, G. Ponce and L. Vega, 
\newblock {\emph Oscillatory integrals and regularity of
dispersive equations}, 
\newblock Indiana Univ. Math. J., {\bf 40} (1991), 33--69.

\bibitem{KPV3} 
\newblock C. E. Kenig, G. Ponce and L. Vega,
\newblock \emph{On the hierarchy of the generalized KdV equations,}
\newblock Proc. Lyon Workshop on singular limits of dispersive waves, \textbf{320} (1994),
347--356.

\bibitem{KPV4}
\newblock C. E. Kenig, G. Ponce and L. Vega,
\newblock \emph{Higher-order nonlinear dispersive equations,}
\newblock Proc. Amer. Math. Soc., \textbf{122} (1994), 157--166.

\bibitem{Kic} 
\newblock S. Kichenassamy, 
\newblock \emph{Existence of solitary waves for water-wave models,} 
\newblock Nonlinearity, \textbf{10} (1997), 133--151.

\bibitem{KT2} H. Koch and D.Tataru,
 \newblock\emph{A priori bounds for the 1D cubic NLS in negative Sobolev spaces,}
 \newblock  Int. Math. Res. Not., (2007), no. 16.
 
 \bibitem{KT}
\newblock  H. Koch and N. Tzvetkov,
\newblock \emph{On the local well-posedness of the Benjamin-Ono equation in $H^s(\mathbb R)$,}
\newblock  Int. Math. Res. Not., \textbf{26} (2003), 1449--1464.

\bibitem{Kw}
\newblock  S. Kwon,
\newblock \emph{On the fifth order KdV equation: local well-posedness and lack of uniform continuity
of the solution map,}
\newblock  J. Diff. Eq., \textbf{245} (2008), 2627--2659.

\bibitem{Lev} 
\newblock S. P. Levandosky, 
\newblock \emph{A stability analysis of fifth-order water wave models,} 
\newblock Physica D, \textbf{125} (1999), 222--240.

\bibitem{Lev2} 
\newblock S. P. Levandosky, 
\newblock \emph{Stability of solitary waves of a fifth-order water wave model,} 
\newblock Physica D, \textbf{227} (2007), 162--172.


\bibitem{MST} L. Molinet, J-C. Saut, and N. Tzvetkov,
\newblock \emph{Ill-posedness issues for the Benjamin-Ono and related equations,}
\newblock SIAM J. Math. Anal.  \textbf{33} (2001), no. 4, 982--988.

\bibitem{Ol} 
\newblock P. J. Olver, 
\newblock \emph{Hamiltonian and non-Hamiltonian models for water-waves,} 
\newblock Lecture Notes in Physics, Springer, Berlin, \textbf{195} (1995), 273--290.

\bibitem{La} 
\newblock P. Lax, 
\newblock \emph{Integrals of nonlinear equations of evolution and solitary waves}, 
\newblock Comm. Pure Appl. Math., \textbf{21} (1968), 467--490.

\bibitem{Mo} 
\newblock L. Molinet, 
\newblock \emph{Sharp ill-posedness results for the KdV and mKdV equations on the torus,} 
\newblock preprint (2011), arxiv: 1105.3601v3.

\bibitem{Pi}
\newblock D. Pilod,
\newblock \emph{On the Cauchy problem for higher-order nonlinear dispersive equations,}
\newblock  J. Diff. Eq. \textbf{245} (2008), 2055--2077. 

\bibitem{Po} G. Ponce,
\newblock \emph{Lax pairs and higher-order models for water waves,}
\newblock J. Diff. Eq., \textbf{102} (1993), 360--381.

\bibitem{Sa} J.-C. Saut, 
\newblock \emph{Quelques g\'en\'eralisations de l'\'equation de Korteweg- de Vries, II,} 
\newblock J. Diff. Eq., \textbf{33} (1979), 320--335.

\bibitem{Ta} D. Tataru, 
\newblock \emph{Local and global results for wave maps I,} 
\newblock Comm. Part. Diff. Eq., \textbf{23} (1998), 1781--1793.


\end{thebibliography}

\end{document}